\documentclass[11pt]{amsart}
    \usepackage{algorithm2e}
    \usepackage{tikz-cd}
    \usepackage{amsmath}
    \usepackage{amssymb}
    \usepackage{amsthm}
    \usepackage{latexsym}
    \usepackage{graphicx}
    \usepackage{hyperref}
    \usepackage{enumerate}
    \usepackage[all]{xy}
    \usepackage{amsfonts}
    \usepackage{amstext}
    \usepackage{mathrsfs}
    \usepackage{lingmacros}
    \usepackage{pdfsync}
    \usepackage[colorinlistoftodos]{todonotes}
    \usepackage{hyperref}
    \usepackage{imakeidx}
    \hypersetup{
      colorlinks   = true,    
      urlcolor     = blue,    
      linkcolor    = blue,    
      citecolor    = red      
    }
    
    \newtheorem{thm}{Theorem}[section]
    \newtheorem{cor}[thm]{Corollary}
    \newtheorem{lem}[thm]{Lemma}
    \newtheorem*{lem*}{Lemma}
    \newtheorem{prop}[thm]{Proposition}
    \newtheorem{defn}[thm]{Definition}

    \newtheorem{rem}[thm]{Remark}

    \newcommand{\bbA}{\mathbb{A}}
    \newcommand{\bbC}{\mathbb{C}}
    
    \newcommand{\bbN}{\mathbb{N}}
    \newcommand{\bbQ}{\mathbb{Q}}
    \newcommand{\bbR}{\mathbb{R}}
    \newcommand{\bbZ}{\mathbb{Z}}
    \newcommand{\bo}{\mathbf{o}}
    \newcommand{\cA}{\mathcal{A}}

    \newcommand{\cF}{\mathcal{F}}
    \newcommand{\cH}{\mathcal{H}}
    \newcommand{\cI}{\mathcal{I}}
    \newcommand{\cL}{\mathcal{L}}
    \newcommand{\cO}{\mathcal{O}}
    \newcommand{\cM}{\mathcal{M}}
    \newcommand{\cP}{\mathcal{P}}
    \newcommand{\cU}{\mathcal{U}}
    
    \newcommand{\ga}{\alpha}
    \newcommand{\gb}{\beta}
    \newcommand{\gc}{\gamma}
    \newcommand{\gC}{\Gamma}

    \newcommand{\go}{\omega}
    \newcommand{\gO}{\Omega}
    \newcommand{\gl}{\lambda}
    \newcommand{\gL}{\Lambda}
    \newcommand{\gk}{\kappa}
    \newcommand{\gs}{\sigma}
    
    \newcommand{\fa}{\mathfrak{a}}
    \newcommand{\fb}{\mathfrak{b}}
    \newcommand{\fc}{\mathfrak{c}}
    
    \newcommand{\fg}{\mathfrak{g}}

    \newcommand{\fs}{\mathfrak{s}}
    \newcommand{\ft}{\mathfrak{t}}
    \newcommand{\fT}{\mathfrak{T}}
    \newcommand{\fp}{\mathfrak{p}}
    \newcommand{\fq}{\mathfrak{q}}
    \newcommand{\fR}{\mathfrak{R}}
    
    \newcommand{\fu}{\mathfrak{u}}
    \newcommand{\fw}{\mathfrak{w}}
    \newcommand{\fz}{\mathfrak{z}}

    \newcommand{\uX}{\textup{X}}
    \newcommand{\uA}{\textup{A}}
    \newcommand{\uB}{\textup{B}}
    \newcommand{\uC}{\textup{C}}
    \newcommand{\uD}{\textup{D}}

    \newcommand{\Ltwoxi}{L^2(G(F)\backslash G(\bbA_F),\xi)}
    \newcommand{\cCO}{\mathcal{C}(\Omega)}

    \newcommand{\Gen}{\mathtt{Gen}}
    \newcommand{\Std}{\mathtt{Std}}
    \newcommand{\rl}{\mathtt{rl}}
    \newcommand{\uw}{\underline{w}}
    \newcommand{\usig}{\underline{\sigma}}
    \newcommand{\upl}{\underline{\hat{P}_L}}
    \newcommand{\sub}{\mathtt{SUB}}
    \newcommand{\ttrue}{\mathtt{true}}
    \newcommand{\tfalse}{\mathtt{false}}

    \SetKwComment{Comment}{/* }{ */}
    \makeindex[columns=2]
    
\begin{document}
    \title[Residue distributions and iterated residues]{Residue distributions, iterated residues, and the spherical automorphic spectrum} 
    
    \author{Marcelo De Martino}
    \address{M.D.M: Department of Eletronics and Informations Systems\\
    Clifford Research Group\\
    9000 Ghent\\
    Belgium\\
    email: marcelo.goncalvesdemartino@ugent.be}
    \author{Volker Heiermann}
    \address{
    V.H.: Aix Marseille Universit\'e\\
    CNRS \\
    Centrale Marseille\\
    I2M UMR 7373\\
    13453\\ 
    Marseille\\
    France\\
    email: volker.heiermann@univ-amu.fr}
    \author{Eric Opdam}
    \address{E.O.: Korteweg de Vries Institute for Mathematics\\
    University of Amsterdam\\
    P.O. Box 94248\\
    1090 GE Amsterdam\\
    The Netherlands\\
     email: e.m.opdam@uva.nl}
    
    \date{\today}
    \date{\today}
    
    \begin{abstract}  
    Let $G$ be a split reductive group over a number field $F$. We consider the computation of the inner product of two $K$-spherical pseudo Eisenstein series of $G$ supported in $[T,\cO(1)]$ by means of residues, following a classical approach initiated by Langlands. We show that only the  singularities of the intertwining operators due to the poles of the completed Dedekind zeta function $\Lambda_F$ contribute to the spectrum, 
    while the singularities caused by the zeroes of $\Lambda_F$ do not contribute
    to any of the iterated residues which arise as a result of the necessary contour shifts. 

    In the companion paper \cite{DMHO} we use this result to explicitly determine the spectral measure of $\Ltwoxi^K_{[T,\cO(1)]}$
    by a comparison of the iterated residues with the residue distributions of \cite{HO1}. 
    \end{abstract}
    
    \maketitle
    
    \section{Introduction and main results}\label{s:IntroRes}
    \subsection{Introduction}\label{ss:intro}
    Let $G$ be a connected unramified reductive group defined over a number field and with maximal torus $T$, and with maximal compact subgroup $K\subset G(\bbA_F)$. 
    Let $\chi$ be an unramified automorphic character of $T$ and denote by $\xi=\chi|_{Z_G}$ its restriction to the centre $Z_G$ of $G$. Throughout, we will denote by
    $\Ltwoxi$ the space of all functions on $G(F)\backslash G(\bbA_F)$ on which $Z_G(\bbA_F)$ acts by $\xi$ and that are square-integrable modulo the centre. We will use the subscript $[T,\cO(\chi)]$ to denote the cuspidal support of the functions involved.
    In \cite{DMHO} we compute the explicit spectral measure of the subspace 
    $\Ltwoxi^K_{[T,\mathcal{O}(\chi)],0}\subset \Ltwoxi^K_{[T,\mathcal{O}(\chi)]}$ 
    spanned by the \emph{normalised} pseudo Eisenstein series. Following Moeglin we expect that
    under suitable conditions we have equality:
    \begin{equation}\label{eq:norm=all}
    \Ltwoxi^K_{[T,\mathcal{O}(\chi)],0}=\Ltwoxi^K_{[T,\mathcal{O}(\chi)]}.
    \end{equation}
    As application of our main Theorem \ref{thm:mainC} we prove in \cite{DMHO} that when $G$ is split, $F$ is a number field, and $\chi=1$, 
    then this equality always holds, 
    and therefore yields the explicit spectral decomposition of $\Ltwoxi)^K_{[T,\cO(1)]}$.
    
    The equation (\ref{eq:norm=all}) expresses that in this situation, the residues of the singularities 
    of the intertwining operators caused by the critical zeroes of the relevant $L$-functions of $F$ do not 
    contribute to the spectral decomposition. 
    \subsubsection{Residual pole spaces and nilpotent orbits of $\hat\fg$}
    In order to express the result, we introduce some notations.
    Let $G$ be split over $F$,  
    and let $V=\fa^{G,*}_T$, with the natural action of the Weyl group $W=W(G,T)$.
    We equip $V$ with a $W$-invariant inner product\footnote{None of the results depend on this inner product, 
    but it is convenient to equip $V$ with the structure of a Euclidean space.} 
    and let $R\subset V$\index{$R$} be the root system of $(G,T)$, and let ${\hat{R}}\subset V^*$\index{$R$!${\hat{R}}$} denote the coroot system. 
    We fix a Borel $B\supset T$ with positive 
    subset $R_+\subset R$\index{$R$!$R_+$}. 
    Define the rational function $c$ on $V_\bbC$ by 
    \begin{equation}\label{eq:c}\index{$c$}
    c(\gl)=\prod_{\ga\in R_+}\frac{\hat{\ga}(\gl)+1}{\hat{\ga}(\gl)}.
    \end{equation}
    An important role is played by the concept of  
    \emph{residual} pole space, defined as follows. Given an affine subspace $L\subset V$, we define the set of pole coroots \index{$\hat{P}_L$ of $L$} by  
     $\hat{P}_L:=\{\hat{\ga}\in \hat{R}\mid \hat{\ga}|_L=1\}$  and 
    the set of singular coroots \index{$\hat{Z}_L$ of $L$} by 
     $\hat{Z}_L:=\{\hat{\ga}\in \hat{R}\mid \hat{\ga}|_L=0\}$.
    Then $L\subset V$\index{$L$ pole space!$L\subset V$ residual pole space} is called a residual pole space if $|\hat{P}_L|\geq |\hat{Z}_L|+\textup{codim}(L)$.
    It is known that if $L$ is a residual pole space then in fact \cite[Section 4]{HO1}, \cite[Theorem 7.1, Remark 7.3]{O-Supp} we have 
    $|\hat{P}_L| = |\hat{Z}_L|+\textup{codim}(L)$. In other words, the rational $(n,0)$-form
    \begin{equation}
    \gO_r(\lambda):=\frac{d\lambda}{c(\lambda)c(-\lambda)}\index{$\gO_r(\lambda)$} 
    \end{equation}
    on $V_\bbC$ has only ``simple poles''. Importantly, $\gO_r$ is $W$-invariant. 
    We denote by $\cL(\gO_r)$\index{$\cL(\go)$!$\cL(\gO_r)$} the set of residual pole spaces; this is a finite $W$-invariant set
    \cite[Introduction]{O-Supp}. 
    In general, if $L\subset V$ is an affine subspace we have a unique decomposition $L=c_L+V^L$\index{$V^L\subset V$ underlying linear subspace of $L$} 
    \index{$c_L$ center of $L$} where $V^L\subset V$ is a linear subspace, and 
    $c_L=V_L\cap L$ where $V_L=(V^L)^\perp$\index{$V_L\subset V$ subspace of $L$ spanned by $\hat{P}_L$}, 
    the \emph{center} of $L$. We call 
    $L^{\textup{\textup{temp}}}:=c_L+iV^L\subset L_\bbC\subset V_\bbC$ \index{$L$ pole space!$L^\textup{temp}$} the tempered real form of $L_\bbC$.
    If $L$ is a residual pole space it is known that $V_L$ is spanned by $\hat{P}_L$ \cite[Introduction]{O-Supp}, and 
    that \cite[Introduction, Remark 7.3, Remark 7.5]{O-Supp} the set of $W$-orbits of residual pole spaces 
    is in bijection with the set of weighted Dynkin diagrams of the nilpotent orbits of the Lie algebra $\hat{\fg}$
    of the dual group $\hat{G}$, via the map 
    \begin{align*}
    \cL(\gO_r)/W&\to\{\mathrm{Weighted\ Dynkin\ diagrams\ of\ nilpotent\ orbits\ of\ }\hat{\fg}\}\\
    WL&\mapsto 2c_{L,+}
    \end{align*}
    where $c_{L,+}$ denotes the dominant element of $Wc_L$. This shows that there is a canonical bijection 
    \begin{equation}\label{eq:BC}
    \cL(\gO_r)/W \longleftrightarrow \hat{\mathcal{N}}/\hat{G}
    \end{equation}
    where $\hat{\mathcal{N}}\subset \hat\fg$ denotes the nilpotent cone.
    We call $L\subset \cL(\gO_r)$ \emph{standard} if $V^L=\hat{\fz}^L\subset \hat\fg$ is the center of a standard Levi 
    subalgebra $\hat\fg^L$ with root system $\hat{R}_L\subset \hat{R}$, and $c_L$ is dominant relative to $\hat{R}_{L,+}$. 
    We fix a set $\cL({\gO_r})_+$\index{$\cL(\go)$!$\cL({\gO_r})_+$} of standard ${\gO_r}$-pole spaces forming a complete set of 
    representatives of the set of $W$-orbits of ${\gO_r}$-pole spaces.
    \subsubsection{The explicit inner product formula of pseudo Eisenstein series}
    We introduce the entire function 
    \begin{equation}\label{eq:r}\index{$r(\gl)$}
    r(\gl)=\prod_{\ga\in R_+}\rho({\hat{\ga}}(\gl)) 
    \end{equation}
    on $V_\bbC$, where $\rho(s)=\gL^{reg}_F(s)$ is the completed Dedekind zeta function of $F$, 
    regularised at $0,1$. 
    Given $\fR\gg0$ we denote by $\cP^\fR(V_\bbC)$\index{$\cP^\fR(V_\bbC)$} the space of holomorphic functions $\phi$ defined on $B_\fR:=\{\gl\in V_\bbC\mid |\textup{Re}(\gl)|<\fR\}$
    such that for all $N\in\bbN$ there exists a $c_N>0$ such that $|\phi(\gl)|\leq c_N(1+|\textup{Im}(\gl)|)^{-N}$. 
    Let $\phi,\psi\in \cP^\fR(V_\bbC)$, and let 
    $\theta_\phi, \theta_\psi\in \Ltwoxi^K_{[T,\cO(1)]}$ denote the 
    associated pseudo Eisenstein series of $G$. 
    Consider the meromorphic function 
    \begin{equation}
    \Sigma_{W}(\phi)(\gl)=\sum_{w\in W}c(-w\gl)\frac{r(\gl)}{r(w\gl)}\phi(-w\gl)\index{$\Sigma_{W}(\phi)$}
    \end{equation}
    The goal of the present paper is to complete the proof of (\ref{eq:norm=all}) for $F$ a number field, $G$ split 
    and $\chi=1$. Equivalently, the task at hand is to prove \cite[Main Theorem]{DMHO}) stating that: 
    \begin{equation}\label{eq:+sym}
    (\theta_\phi,\theta_\psi)=\sum_{L\in \cL(\gO_r)_+} |W|^{-1}\int_{WL^{\textup{temp}}}
    \overline{\Sigma_W(r\phi)(\lambda)} \Sigma_W(r\psi)(\lambda) \frac{d\nu_{WL}(\lambda)}{r(-\lambda)r(\lambda)}.
    \end{equation}
    Here $\nu_{WL}$ is a $W$-invariant nonnegative measure whose support is $\cup_{L\in\cL(\gO_r)}L^{\textup{temp}}$ and whose precise 
    description is as follows:  
    \begin{equation}\label{eq:nuLdef}
    \nu_{WL}:=\sum_{L'\subseteq WL} \nu_{L'}
    \end{equation}
    where $\nu_L$ the push forward of the smooth measure on $L^{\textup{temp}}$ given by (with 
    $\gl=c_L+i\gl^L\in L^{\textup{temp}}$):
    \begin{equation}\label{eq:nuL}
    d\nu_L(\lambda):=\frac{|W(\Phi_L)_{c_L}|}{|A_{\bo_{WL}}|}\frac{\prod'_{\alpha\in R_L}{|\hat\ga}(c_L)|}
    {\prod'_{\alpha\in R_L}|{\hat\ga}(c_L) + 1|}\prod_{\alpha\in R^+\backslash R_L}
    \frac{c_L({\hat\ga})^2+{\hat\ga}(\lambda^L)^2}{(c_L({\hat\ga})-1)^2+{\hat\ga}(\lambda^L)^2}
    d\lambda^L,
    \end{equation}
    and, if $r^L := \dim V^L$ and $\{y_1,\ldots, y_{r^L}\}$ denotes a basis of the lattice 
    $X_*(T) \cap (V_L)^\perp=(V^L)^*$ then:
    \begin{equation}
    d\lambda^L := \frac{dy_1\wedge\cdots\wedge dy_{r^L}}{(-2\pi \sqrt{-1})^{r^L}}.
    \end{equation}
    where $\prod'$ denotes the product of the nonzero factors.
    \subsubsection{Langlands' formula for the inner product of two pseudo-Eisenstein series.}
     The proof of (\ref{eq:+sym}) is a residue computation starting with Langlands' formula (cf. \cite{MW2}):
    \begin{align}\label{eq:Lang}
    (\theta_\phi,\theta_\psi)&=\int_{\textup{Re}(\lambda)=\lambda_0\gg0} 
    \sum_{w\in W}\prod_{\alpha\in\Phi^+\cap w^{-1}\Phi^-}
    \frac{{\Lambda}(\alpha^\vee(\lambda))}{{\Lambda}(\alpha^\vee(\lambda)+1)}
    \phi^-(-w\lambda)\psi(\lambda)d\lambda\\
    \nonumber&=\int_{\textup{Re}(\lambda)=\lambda_0\gg0} 
    \Sigma_{W}(\phi^-)(\gl)\psi(\gl)\go(\gl)
    \end{align}
    where $\phi^-\in\cP^\fR(V_\bbC)$ be given by $\phi^-(\gl)=\overline{\phi(\overline{\gl})}$\index{$\phi^-$}, 
    and 
    \begin{equation}
    \go(\gl):=\frac{d\lambda}{c(-\lambda)}\index{$\gO_r(\lambda)$!$\go(\gl)$}
    \end{equation}
    and where
    \begin{equation}
    d\gl=\frac{dy_1\wedge dy_2\wedge\dots\wedge dy_r }{(-2\pi\sqrt{-1})^r}
    \end{equation}
    with $(y_1,\dots,y_r)$ a positively oriented basis of $X_*(T)\cap V^\perp \subset V^*$. 
    
    If we replace $\rho=\gL^{reg}_F$ by an entire function $\rho$ 
    which has at most moderate growth in 
    vertical strips and \emph{without zeroes}, then rewriting (\ref{eq:Lang}) as (\ref{eq:+sym}) is a straightforward application 
    of the unique residue distributions for the rational $r$-form $\go$ as discussed in \cite{HO1}. In the presence of zeroes of $\rho$ 
    the singularities of $\Sigma_W(\phi^-)$ (caused by 
    the zeroes of $\rho$), the equality of (\ref{eq:Lang}) and (\ref{eq:+sym}) is highly nontrivial. It expresses that the singularities 
    of $\Sigma_W(\phi^-)$ do not make contributions. Recently a breakthrough geometric explanation for this surprising 
    behaviour has been given by David Kazhdan and Andrei Okounkov \cite{KO}.   
    
    We will prove (in \cite{DMHO} and the present paper) the equality of (\ref{eq:Lang}) and (\ref{eq:+sym}) with a more 
    classical approach, building on the 
    contributions of Langlands \cite{La1}, \cite{La2}, Jacquet \cite{J}, Kim \cite{K1}, \cite{K2},  Moeglin \cite{M1}, \cite{M2}, 
    Moeglin-Waldspurger \cite{MW1}, \cite{MW2} and others. At certain points, our arguments are aided by computer algebra for the 
    exceptional groups. Codes and output files for the computer-assisted manipulations will be available at {\tt https://github.com/mgdemartino/SphAutSpec}. Let us outline the steps of our proof, and explain the role of the present paper therein. 
    
    \subsubsection{Pole spaces and iterated residues}
    Let $G'\subset G$ be a maximal standard proper Levi subgroup with 
    root system $R'=R'(G',T)\subset R$, a maximal proper standard Levi subsystem. 
    
    We define a rational function $c'$ on $V_\bbC$ by: 
    \begin{equation}\label{eq:c'}\index{$c'$}
    c'(\gl)=\prod_{\ga\in R'_+}\frac{\hat{\ga}(\lambda)+1}{\hat{\ga}(\lambda)}
    \end{equation}
    and a rational $(r,0)$-form on $V_\bbC$ by 
    \begin{equation}\label{eq:Omega}
    \gO(\lambda):=\frac{d\lambda}{c'(\lambda)c(-\lambda)}\index{$\gO_r(\lambda)$!$\gO(\lambda)$}.
    \end{equation}
    Observe that $\gO$ is $W'$-invariant, in fact it is the $W'$-average of $\go$. 
    \begin{defn}\label{defn:pole}
    An affine subspace $L\subset V$ is called a pole space \index{$L$ pole space} for $\gO$ 
    if the set of $\gO$-poles $\hat{P}_L(\gO):=\hat{P}_L\cap (\hat{R}_+\cup\hat{R}'_-)$ 
    and $\gO$-zeroes $\hat{Z}_L(\gO):=\hat{Z}_L\cap (\hat{R}_+\cup\hat{R}'_-)$ satisfy $|\hat{P}_L(\gO)|\geq |\hat{Z}_L(\gO)|+\textup{codim}(L)$, 
    and $L=\cap_{\ga\in\hat{P}_L(\gO)}\{\hat\ga=1\}$.
    We denote by $\textup{Ord}_L(\gO):=|\hat{P}_L(\gO)|-|\hat{Z}_L(\gO)|-\textup{codim}(L)$\index{$\textup{Ord}_L(\Omega)$, the order of $\Omega$ along $L$} 
    the \emph{order} of $\gO$ along $L$. Let $\cL(\gO)$\index{$\cL(\go)$!$\cL(\gO)$} denote set of $\gO$-pole spaces (a $W'$-invariant set).
    We similarly define the sets $\cL(\go)$\index{$\cL(\go)$}, $\cL({}^w\go)$ (with $w\in W$), $\cL(\gO_r)$ etc. of poles spaces 
    for $\go$, ${}^w\go$, $\gO_r$ etc. respectively. 
    
    We say that $L$ is a pole space \index{$L$ pole space} (without reference to a specific rational form) if 
    $L\in \cup_{w\in W} \cL({}^w\gO)=\cup_{w\in W} w\cL(\gO)$. 
    A residual pole space is an $\gO_r$-pole space (see text immediately below (\ref{eq:c}) in p.2). A residual $\gO$- (or $\go$-, ${}^w\go$-) pole space is a $\gO$- ($\go$-, ${}^w\go$-, respectively) pole space which is also a $\gO_r$-pole space.
    \end{defn}
    
    For $L\in\cL(\gO)$, since  $\gO$ only symmetrizes $\go$ partially, we may have $\textup{Ord}_L(\gO)>0$.  
    
    Suppose given a flag $\cF$ of $\gO$ pole spaces of the form $L=L_k\subset L_{k-1}\subset \dots\subset L_0=V$, 
    starting at $V=L_0$ (with $\textup{dim}(V)=r$) and ending with $L$. Let $F$ be an entire function on $V_\bbC$. 
    There is a canonically defined meromorphic $(r-k=\textup{dim}(L),0)$-form on $L_\bbC$ denoted by 
    $\textup{Res}_{L,\cF}(F\gO)$ of the form 
    \begin{equation}\label{eq:intitres}
    \textup{Res}_{L,\cF}(F\gO)=\sum_i D_{L,\cF,i}(F)|_{L_\bbC}\gO_{L,\cF,i} 
    \index{$\textup{Res}_{L,\cF}(F\gO)$}
    \end{equation}
    with constant coefficient differential operators $D_{L,\cF,i}$ of degree $\text{Ord}_L(\gO)$
    in $S(V_L)$, the symmetric algebra of the normal space $V_L$ to $L$, 
    and rational $(\textup{dim}(L),0)$-forms $\gO_{L,\cF,i}$. We call $\textup{Res}_{L,\cF}(F\gO)$ the iterated residue of $F\gO$ along $L$.
    In the favourable case where in every step of the chain of pole spaces the pole 
    is a simple pole (i.e. the order $0$ case) then 
    the formula simplifies to an expression of the form  
    \begin{equation}\label{eq:resdatCst}
    \textup{Res}_{L,\cF}(F\gO)=(F|_{L_\bbC})\textup{Res}_{L,\cF}(\gO)=\fc_{L,\cF}(F|_{L_\bbC})\gO^L
    \end{equation} 
    which, up to a nonzero rational constant factor $\fc_{L,\cF}$, is independent of the chain $\cF$ of pole spaces.  
    Here $\gO^L(\gl)$\index{$\gO_r(\lambda)$!$\gO^L(\gl)$} denotes, with $\gl=c_L+\gl^L\in L_\bbC$:
    \begin{equation}\label{eq:rewrO}
    \gO^L(\gl)=\frac{d\gl^L}{(c')^L(\gl)c^L(-\gl)}
    \end{equation}
    where $c^L$, $(c')^L$ denote the products of the $c_{\hat{\ga}}$ over the coroots ${\hat{R}_+\backslash\hat{R}_{L,+}}$
    and ${\hat{R}'_+\backslash\hat{R}'_{L,+}}$ respectively, i.e. the coroots 
    which are nonconstant on $L$. Note that $\gO^L$ is rational on $L$. 
    In this case, we often write $C_{L,\cF}\gO^L$ instead of $\textup{Res}_{L,\cF}(\gO)$. 
    In particular, in the order $0$ case the singular set of $\textup{Res}_{L,\cF}(\gO)$ is well defined in $L_\bbC$ 
    independent of the flag of pole spaces $\cF$.
    \subsubsection{Organization of the proof of (\ref{eq:+sym}); the induction hypothesis}
    We will show that there exists a sequence of steps consisting of contour shifts and partial symmetrizations avoiding the singularities of $\Sigma_W(\phi^-)$ in (\ref{eq:Lang}),
    using truncation of integrals as designed by Langlands, whereby we replace sets of the form $p+iV^M$ by 
    compact sets of the form $(p+iV^M)_{\leq \fT}=\{\gl=p+iv\in p+iV^M\mid \vert v\vert^2\leq \vert p\vert^2+\fT\}$, for some sufficiently large $\fT$. The partial symmetrizations are organized in an inductive argument, using induction to the semisimple rank of $G$, following \cite{MW1} and \cite{M1}.
    We will need a case by case analysis in order to be able to do this and for the exceptional cases we will also need computer assistance. 
    The contour shifts are encoded in a combinatorial object, called \emph{cascade} and denoted $C$, which is described in Section \ref{s:MainResultRes} and 
    defined in Appendix \ref{a:Cascade}.
    
    Assume by induction that (\ref{eq:+sym}) is true for all connected reductive groups over $F$ of semisimple rank 
    strictly less than the semisimple rank $r$ of $G$. 
    Let the meromorphic function $\Sigma'_{W'}(\psi)$ on $V_\bbC$ be given by: 
    \begin{equation}
    \Sigma'_{W'}(\psi)(\gl)=\sum_{u\in W'}c'(u\gl)\frac{r(u\gl)}{r(\gl)}\psi(u\gl)\index{$\Sigma_{W}(\phi)$!$\Sigma'_{W'}(\psi)$}.
    \end{equation}
    Let 
    \begin{equation}
    {\gO'_r}(\lambda):=\frac{d\lambda}{c'(\lambda)c'(-\lambda)}\index{$\gO_r(\lambda)$!${\gO'_r}(\lambda)$} 
    \index{$\gO_r(\lambda)$!${\gO'_r}(\lambda)$}
    \end{equation}
    (which is $W'$-invariant and rational).
    For each $L'\in \cL(\gO'_r)$\index{$\cL(\go)$!$\cL(\gO'_r)$} we write $L=L'\oplus \bbR\fw'\in \cL(\gO_r)$, where
    $\fw'$\index{$\mathfrak{w}'$} is the fundamental weight of $\hat{R}$ that is orthogonal to $\hat{R'}$. 
    We define an initial base point $p_{L,\infty}$\index{$p_{L,\infty}$} ``at infinity'' of the form $p_{L,\infty}=c_L+t\fw'$ with $t\gg0$. 
    
    Let $d\nu'_L$ be the measure of (\ref{eq:nuL}), but with respect to the lower rank situation of $G',R'$ and $\Omega_r'$. We then assign the measure $d\nu'_L(dt/2\pi \sqrt{-1})$ to $L=L'\oplus \bbR\fw'$ and for any entire function $\phi$ on $V_\bbC$ we have that, with $\gl=\gl'+t\fw'\in L$,
    $F_L(\gl):=\frac{(c')^L(-\gl')}{c^L(-\gl)}$\index{$F_L(\gl)$}
    \begin{align}\label{eq:ResOmg}
    (F_L\phi|_{L_\bbC})(\gl)d{\nu'}_L(\gl')\frac{dt}{-2\pi\sqrt{-1}}&=C_L(\phi|_{L_\bbC})\gO^L
    \end{align}
    where $\Omega^L$ is as in (\ref{eq:rewrO}) and $C_L$
    is a nonzero constant. It follows in particular that $\gO^L$ pole spaces are $\gO$-pole spaces which are contained in $L_\bbC$. 
    
    Using the induction hypothesis it is now easy to see that: 
    \begin{equation}\label{eq:ind}
    (\theta_\phi,\theta_\psi)=|W'|^{-1}\sum_{L'\in\cL({\gO'_r})_+}\int_{p_{L,\infty}+iV^L}C_L\Sigma_W(\phi^-)(\gl)\Sigma'_{W'}(\psi)(\gl)\Omega^L(\gl).
    \index{$\cL(\go)$!$\cL({\gO'_r})_+$}
    \end{equation}
    We need to truncate the contours in order to obtain compact contours which allow us to make a generic deformation of the base points 
    as in \cite{MW2}.  
    For each $L=L'+\bbR\fw'$, we define a generic deformation 
    $p_{L,\infty,\fT'}=c_L+t\mathfrak{w}'+v_\epsilon=p_{L,\infty}+v_\epsilon$\index{$p_{L,\infty}$!$p_{L,\infty,\fT'}$} of 
    $p_{L,\infty}$ (i.e. $t\gg \fT$ and $v_\epsilon\in V^L$ suitably small). 
    The induction hypothesis thus yields (cf. \cite{M1}; we use the notations $=_\fT$ and ${}_{\leq \fT}$ as in \cite{MW2}): 
    \begin{equation}\label{eq:indT}
    (\theta_\phi,\theta_\psi)=_\fT|W'|^{-1}\sum_{L'\in\cL({\gO'_r})_+}\int_{(p_{L,\infty,\fT'}+iV^L)_{\leq \fT}}C_L\Sigma_W(\phi^-)(\gl)\Sigma'_{W'}(\psi)(\gl)\Omega^L(\gl)
    \end{equation}
    for suitable nonzero real constants $C_L$.

    \subsection{Main result and outline}
    Now (\ref{eq:indT}) is the starting point of the cascade $C$\index{$C$, Cascade} of contour shifts as described in Section \ref{s:MainResultRes} and 
    defined in Appendix \ref{a:Cascade}. Given $C$, we denote by $\cCO^C$\index{$C$, Cascade!$\cCO^C$} a set of representatives of the $W'$-orbits 
    of the $\gO$-pole spaces appearing in $C$.
    The goal is to move the base points by $W'$-transformations and the contour shifts (at the cost of iterated residue integrals), 
    until we 
    reach a situation in which every iterated residue integral in the cascade $C$ is over a contour of the form 
    $(z_L+iV^L)_{\leq \fT}$, where $L$ is an $\gO$-space in $C$, and $z_L\in L$ is generic and near the center $c_L$ of $L$.
    The fact that this can be done without picking up residues from the singularities of $\Sigma_W(\phi^-)\Sigma'_{W'}(\psi)$  
    is the main result of the present paper: 
    \begin{thm}\label{thm:mainC}
    Let $G$, $T \subset B \subset G$ be as above.   
    Choose constants $\fR,\fT,\fT'>0$ with $\fT>\fR$ and $\fT'>3\fT+2\fR^2$ as in \cite[V.2.2]{MW2}. 
    Let $\phi,\psi\in\cP^\fR(V_\bbC)$ and let $\rho$ be any entire function on $\bbC$ 
    satisfying the functional equation $\rho(s)=\rho(1-s)$, and such that 
    $\Lambda:=s^{-1}(1-s)^{-1}\rho$ satisfies the properties 
    \cite[Proposition 3.1]{DMHO}. 
    There exists a maximal proper standard Levi subgroup $G'\subset G$ with root systems $R'\subset R$, and a Cascade $C$ 
    for $R'\subset R$ (see Definition \ref{defn:casc}) satisfying the following:
    Let $\cCO^{C,presym}$\index{$\cCO^C$!$\cCO^{C,presym}$} be collection of pole spaces 
    in $\cCO^C$ after removing the $M$ such that $\textup{Ord}_M(\Omega)=0$ and $M^{\textup{temp}}\not\subset\cup_{N\in \cL(\gO_r)}N^{\textup{temp}}$.
    For every pole space $M\in \cCO^{C,presym}$ we can choose a $\fT'$-general point $z_M$ close to the center $c_M$ of $M$ such that   
    \begin{equation}\label{eq:nearcenter}
    (\theta_\phi,\theta_\psi)=_{\fT}\sum_{M\in \cCO^{C,presym}}\int_{(z_M+iV^M)_{\leq \fT}}\sum_{\cF\in \cF(M)}C_{M,\cF}
    \textup{Res}_{M,\cF}((\Sigma_{W}(\phi^-))(\Sigma'_{W'}(\psi))\Omega)
    \end{equation}
    for suitable constants $C_{M,\cF}$. 
    \end{thm}
    Once Theorem \ref{thm:mainC} has been established, the remaining part of the proof consists of \emph{symmetrising}
    the integral over $z_L+iV^L$ in (\ref{eq:nearcenter}) over the normalizer $N_W(L)=\{w\in W\mid wL=L\}$
    of $L$ in $W$, and \emph{comparison} of the result with the kernels of the residue distribution of the functional on $\cP^\fR(V_\bbC)$ given by (with $z_0\gg0$, 
    and $\theta\in\cP^\fR(V_\bbC)$):
    \begin{equation}
    X_{V,\lambda_0}(\theta):=
    \int_{\textup{Re}(\lambda)=\lambda_0\gg0}\theta(\gl)\go(\gl).
    \end{equation}
    These two steps are addressed in \cite{DMHO}.
    \subsubsection{Outline}
    In Section \ref{s:PoleDen}, we will define the set of \emph{denominators} of functions like $\Sigma_W(\phi)|_{L_\bbC}$ and 
    $\Sigma'_{W'}(\psi)|_{L_\bbC}$ on a $\gO$-pole space $L_\bbC\subset V_\bbC$. This notion is ``universal'' in the entire function $\rho$ 
    on $\bbC$ satisfying the functional equation $\rho(s)=\rho(1-s)$. In our main applications, this function will be the regularised completed Dedekind 
    zeta function $\rho(s)=\Lambda^{reg}_F(s)$ of the number field $F$\footnote{Our results will be true 
    more generally for all $\rho$ as above such that the zeroes of $\rho$ are contained in the critical strip $\textup{Re}(s)\in(0,1)$}.
    We study their behaviour under restriction and transformation 
    by $W$, and compute these sets for \emph{good regular} $\gO$-pole spaces. In Section \ref{s:RegEnv}, we use these properties to define the larger 
    (but computable) set 
    of \emph{enveloping denominators} of $\Sigma_W(\phi)|_{L_\bbC}$ and $\Sigma_{W'}(\psi)|_{L_\bbC}$ on $L_\bbC$ by looking at the denominator sets 
    of all good regular $\gO$ pole spaces $H\supset L$. We establish the relations between the sets of enveloping denominators of two standard pairs 
    $(L_0,x)$ and $(L_1,y)$ such that $x(L_0)=y(L_1)$, and between $\Sigma'_{W'}(\psi)|_{L_\bbC}$ and $\Sigma_{W}(\phi)|_{L_\bbC}$ 
    (Theorem \ref{eq:tau}). 
    In Section \ref{s:Adm}, we define a closed convex subset $\textup{Adm}(L)\subset L$ for every $\gO$-pole space $L$ such that 
    $\Sigma_W(\phi)|_{L_\bbC}$ and $\Sigma'_{W'}(\psi)|_{L_\bbC}$ are regular at the points of $\textup{Adm}(L)+iV^L\subset L_\bbC$.
    We derive transformation properties of $\textup{Adm}(L)$ under $W$ and for intersection with a $\gO$ pole subspace $M\subset L$. 
    
    For every irreducible root system $R$ we find a maximal proper Levi subsystem $R'\subset R$ and construct a Cascade $C$  
    (see Appendix \ref{a:Cascade} for a definition) such that for every $\gO$-pole space $L$ in $C$ which is also an $\gO_r$-pole space,
    we have $c_L\in\textup{Adm}(L)$ (this is Theorem \ref{thm:main}, in the language of Section \ref{s:Adm}). Recall that all the $\gO_r$ spaces are in fact \emph{residual} and their $W$-orbits are naturally in bijection with the set nilpotent orbits, as in (\ref{eq:BC}).
    The construction of $C$ and the proof of Theorem \ref{thm:main} is computer assisted for the exceptional cases.
     In Section \ref{s:MainResultRes}, we finally show that the constructed cascades $C$ allow us to rewrite ``$=_\fT$'' in the inner product 
     of pseudo Eisenstein series using only contours with generic base points near the centers of $\gO$-pole spaces,  
     which the main result Theorem \ref{thm:mainC} of this paper. 
     The proof of Theorem \ref{thm:mainC} is computer assisted for the exceptional cases as well, in two ways. First of all, in Theorem 
     \ref{thm:casc_ord_int} the $W'$-orbits of $\gO$-pole spaces $L$ in the Cascade $C$ with $\textup{Ord}_L(\gO)>0$ and which are \emph{not}  
     met at the center $c_L$ have been classified with computer assistance. In only one case, the ``special line $L=L_{sp}$'' of type 
    $\textup{E}_7(a4)$ for $R$ of type $\textup{E}_8$, there is a potential occurrence of a troublesome factor $\rho(X)^{-1}$ in a relevant normal 
    derivative $\partial_n(\Sigma_W(\phi))|_{L_\bbC}$ associated to one specific flag of pole spaces in $C$ ending in $L$. This denominator $X$ is troublesome in the sense that $X$ maps the segment $\gs_L$ corresponding to that flag to a segment in $\bbC$ intersecting the critical strip nontrivially (in fact, $X(\sigma_L)=[-1/2,1]$). We ruled out the actual occurrence of 
    $X$ as a denominator of $\partial_n(\Sigma_W(\phi))|_{L_\bbC}$ by computer assisted computation, see Corollary 
    \ref{cor:Cadmis}, Theorem \ref{thm:mainspec}. 
    
    In Appendix \ref{a:Cascade} we define the concept of a Cascade $C$ of a pair $(R,R')$ of a root system $R$ and a maximal proper 
    Levi subsystem $R'\subset R$. In Appendix \ref{a:class} we discuss such cascades $C$ in the case of classical root systems. In 
    Appendix \ref{a:EnvDen} we discuss the pseudocode for the computation of the sets of enveloping denominators. 
    In Appendix \ref{a:Special} we analyse the denominators of the potentially problematic normal derivative $\partial_n(\Sigma_W(\phi))|_{L_\bbC}$
    for the ``special line $L=L_{sp}$'' of type $\textup{E}_7(a4)$ for $R$ of type $\textup{E}_8$, and finally, in Appendix \ref{a:SpecialCode} we 
    discuss the pseudocode for the computation of the these denominators based on the results of Appendix \ref{a:Special}. 
    
    \section{Pole spaces and denominators}\label{s:PoleDen}
    Let $L\subset V$ be an affine subspace. We denote by $\textup{Aff}(L)$\index{$\textup{Aff}(L)$} the complex vector space of affine linear 
    functions on $L_\bbC$.
    Let $\textup{Aff}^*(L)\subset \textup{Aff}(L)$\index{$\textup{Aff}(L)$!$\textup{Aff}^*(L)$} denote the subset of non-constant affine functions on $L$.  
    We call $\mu,\nu\in \textup{Aff}(L)$ equivalent, denoted by $\mu\sim\nu$\index{$\mu\sim\nu$ meaning $\mu\in\{\nu,1-\nu\}$}, if $\mu=\{\nu,1-\nu\}$ on $L$. 
    Let $\textup{Aff}^\sim(L)\subset \textup{Aff}^*(L)$\index{$\textup{Aff}(L)$!$\textup{Aff}^\sim(L)$} 
    denote a complete set of representatives of the set of equivalence classes in $\textup{Aff}^*(L)$.
    Let ${\hat{R}}_L\subset {\hat{R}}$\index{$R$!${\hat{R}}_L$} be the Levi subsystem of coroots which are 
    constant on $L$. 
    \begin{thm}\label{thm:denbase}
    Let $A(\rho;\phi_1,\dots,\phi_k)$ be a meromorphic function on $V_\bbC$ depending on  
    an entire function $\rho$ on $\bbC$ satisfying the functional equation 
    \begin{equation}\label{FE}
    \rho(z)=\rho(1-z),
    \end{equation} 
    and on a list of entire functions $\phi_1,\dots,\phi_k$ on $V_\bbC$. Assume that $A$ satisfies the following properties:
    \begin{enumerate}
    \item[(1)]  $A(\rho;\phi_1,\dots,\phi_k)$ is multilinear in $(\phi_1,\dots,\phi_k)$.
    \item[(2)] $A(1;\phi_1,\dots,\phi_k)$ is entire for all tuples $(\phi_1,\dots,\phi_k)$ of entire functions on $V_\bbC$. 
    If all $\phi_i$ are polynomial then $A(1;\phi_1,\dots,\phi_k)$ is a polynomial of degree 
    $\leq \sum_i\textup{deg}(\phi_i)$.
    \item[(3)] There exists $m_{\hat{\ga},i}\in\bbZ_{\geq 0}$ ($\hat\ga\in \hat{R}$, $i=1,\dots,k$) such that for any entire functions $\rho_0$, $\rho_1$ on $\bbC$ satisfying (\ref{FE}), we have $A(\rho_0\rho_1;\phi_1,\dots,\phi_k)=
    (r^A_0)^{-1}A(\rho_1;r^A_{0,1}\phi_1,\dots,r^A_{0,k}\phi_k)$, 
    where $r_{0,i}^A(\gl):=\prod_{\ga\in R}\rho_0(\hat{\ga}(\gl))^{m_{\hat{\ga},i}}$
    and where $r_0^A(\gl):=\prod_{\ga\in R}\rho_0(\hat{\ga}(\gl))^{m_{\hat{\ga}}}$ with $m_{\hat{\ga}}=\sum_i m_{\hat{\ga},i}$. 
    \item[(4)]  For each  $i\in\{1,\dots,k\}$ there exists $N_i\in \bbN$, a subgroup $W_i\subset W$, and an element $d_i\in W$ such that if 
    $\phi_i$ vanishes of order $N\geq 0$ at an orbit $W_id_i(\gl)\subset V_\bbC$ 
    then $A(1;\phi_1,\dots,\phi_k)$ vanishes of order $N-N_i$ at $\gl$. 
    \end{enumerate}
    Let $L\subset V$ be an affine subspace. There exists a unique function $m^A_L:\textup{Aff}^\sim(L)\to \bbZ_{\geq 0}$ such that: 
    \begin{enumerate}
    \item[(i)] $m^A_L$ has finite support. 
    \item[(ii)] 
    Let $\textup{Con}^A(L):=\{\hat{\ga}(L),1-\hat{\ga}(L)\in\bbC\mid \ga\in {\hat{R}}_L\mathrm{\ and\ } m_{\hat{\ga}}\not=0\}$\index{$\textup{Con}^A(L)$}.
    For all entire functions $\rho$ on $\bbC$ which are nonvanishing on $\textup{Con}^A(L)$ and which 
    satisfy (\ref{FE}), and all entire functions $\phi_i,\dots,\phi_k$ on $V_\bbC$, 
    the expression: 
    \begin{equation}\label{eq:m}
    \prod_{\nu\in \textup{Aff}^\sim(L)}\rho(\nu(\lambda))^{m^A_L(\nu)} A(\rho;\phi_1,\dots,\phi_k)|_{L_\bbC}
    \end{equation}
    is entire on $L_\bbC$.
    \item[(iii)] Minimality with respect to $(ii)$: For all function $n_L: \textup{Aff}^\sim(L)\to \mathbb{Z}_{\geq 0}$ with finite 
    support such that $(ii)$ holds with $m^A_L$ replaced by $n_L$, we have $m^A_L\leq n_L$
    (i.e. $m^A_L(\nu)\leq n_L(\nu)$ for all $\nu\in \textup{Aff}^\sim(L)$).
    \end{enumerate}
    \end{thm}
    \begin{proof}
    We first remark 
    that there exists a function $m^A_V:\textup{Aff}^\sim(V)\to \mathbb{Z}_{\geq 0}$ with finite support such that the 
    expression 
    \begin{equation}\label{eq:onV}
    \prod_{\nu\in \textup{Aff}^\sim(V)}\rho(\nu(\gl))^{m^A_V(\nu)}A(\rho;\phi_1,\dots,\phi_k)
    \end{equation}
    is entire on $V_\bbC$. Indeed, if we take $m_V^A(\nu)=0$ unless $\nu\in {\hat{R}}$, 
    and $m_V^A(\hat{\ga})=m_{\hat{\ga}}$ for $\hat{\ga}\in {\hat{R}}$ (observe that if $\hat{\ga},\hat{\gb}\in {\hat{R}}$ then $\hat{\ga}\sim\hat{\gb}$
    iff $\hat{\ga}=\hat{\gb}$) then with $\rho=\rho.1$ we put
    \begin{equation}\label{eq:ronV}
    r^A(\gl)=r^A_0(\gl)=\prod_{\nu\in \textup{Aff}^\sim(V)}\rho(\nu(\gl))^{m^A_V(\nu)}
    \end{equation}
    and the assertion follows from $(2)$ and $(3)$. 
    By assumption $\rho$ has no zeroes in the finite set $\textup{Con}^A(L)$, hence $r^A|_{L_\bbC}$ is not vanishing 
    identically. Since (\ref{eq:onV}) is entire on $V_\bbC$ we may conclude that $A(\rho;\phi_1,\dots,\phi_k)|_{L_\bbC}$ is well defined as 
    meromorphic function on $L_\bbC$, and that $r^A|_{L_\bbC}A(\rho;\phi_1,\dots,\phi_k)|_{L_\bbC}$ is entire on $L_\bbC$. 
    
    We claim that there exists an $m:\textup{Aff}^\sim(L)\to \mathbb{Z}_{\geq 0}$ with the properties:
    \begin{enumerate}
    \item[(a)] For all entire function $\rho$ satisfying (\ref{FE}) and nonvanishing on $\textup{Con}^A(L)$, and entire functions $\phi_1,\dots,\phi_k$, 
    the function 
    \begin{equation}\label{eq:onL}
    \prod_{\nu\in \textup{Aff}^\sim(L)}\rho(\nu(\gl))^{m(\nu)} A(\rho;\phi_1,\dots,\phi_k)|_{L_\bbC}
    \end{equation}
    is entire on $L_\bbC$, 
    \item[(b)] For all $m':\textup{Aff}^\sim(L)\to \mathbb{Z}_{\geq 0}$ with $m'<m$ there exist an entire function $\rho$ on $\mathbb{C}$, 
    (nonvanishing on $\textup{Con}^A(L)$ and satisfying (\ref{FE})), and entire functions $\phi_1,\dots,\phi_k$ on $V_\bbC$, such that 
    \begin{equation}\label{eq:sing}
    \prod_{\nu\in \textup{Aff}^\sim(L)}\rho(\nu(\gl))^{m'(\nu)} A(\rho;\phi_1,\dots,\phi_k)|_{L_\bbC}
    \end{equation}
    is not entire on $L_\bbC$.
    \end{enumerate}
    Indeed, we can construct such an $m$ by starting from an $\tilde{m}$ with finite support which makes (\ref{eq:m}) entire on $L$ for all 
    $\rho$ and $\phi_1,\dots,\phi_k$ as above. Such initial multiplicity function $\tilde{m}$ exists because (\ref{eq:onV}) is entire on $L_\bbC$.
    If this property still holds after subtracting from $\tilde{m}$ an indicator function $\delta_{\nu_0}$ of a singleton subset $\{\nu_0\}$ in the support of 
    $\tilde{m}$, replace $\tilde{m}$ by $\tilde{m}-\delta_{\nu_0}$. Continue like this until no further subtraction is possible, then 
    the resulting function $m$ satisfies the requirements. 
    
    Next we show that the above minimality properties define $m$ uniquely. First we show:
    \begin{lem}\label{lem:restopol}
    If $m$ satisfies $(a)$ and $(b)$ as above, then $m$ also satisfies $(a)_{pol}$ and $(b)_{pol}$, which are like $(a)$ and $(b)$ but with 
    $\rho\in\bbC[z]$ (satisfying (\ref{FE}) and nonzero at $\textup{Con}^A(L)$) and with $\phi\in \bbC[V]$. 
    \end{lem}
    \begin{proof}
    Of course $(a)_{pol}$ is trivially satisfied by $m$. 
    Now suppose that $m'<m$. By the construction of $m$ from $r^A|_{L_\bbC}$ there exist multiplicity parameters 
    $0\leq m'_{\hat{\ga}},m''_{\hat{\ga}}\leq m_{\hat{\ga}}=m_V^A(\hat{\ga})$ for $\hat{\ga}\in {\hat{R}}$ such that 
    \begin{equation}
    \prod_{\nu\in \textup{Aff}^\sim(L)}\rho(\nu)^{m(\nu)}=\prod_{\hat{\ga}\in {\hat{R}}}\rho(\hat{\ga}|_{L_\bbC})^{m'_{\hat{\ga}}}
    \prod_{\hat{\ga}\in {\hat{R}}}\rho(\hat{\ga}|_{L_\bbC})^{m''_{\hat{\ga}}}
    \end{equation}
    while 
    \begin{equation}\label{eq:m''}
    \prod_{\nu\in \textup{Aff}^\sim(L)}\rho(\nu)^{m'(\nu)}=\prod_{\hat{\ga}\in {\hat{R}}}\rho(\hat{\ga}|_{L_\bbC})^{m'_{\hat{\ga}}}
    \end{equation}
    By definition there exist entire functions $\rho$ (satisfying (\ref{FE}) and nonzero on $\textup{Con}^A(L)$) and entire functions 
    $\phi_1,\dots,\phi_k$ such that (\ref{eq:sing}) is singular on $L_\bbC$, while by (\ref{eq:onL}), (\ref{eq:m''}):
    \begin{equation}
    (\prod_{\hat{\ga}\in {\hat{R}}}\rho(\hat{\ga}|_{L_\bbC})^{m''_{\hat{\ga}}})(\prod_{\nu\in \textup{Aff}^\sim(L)}\rho(\nu)^{m'(\nu)} )A(\rho;\phi_1,\dots,\phi_k)|_{L_\bbC}
    \end{equation}
    is entire on $L_\bbC$. We conclude that the singular 
    locus of (\ref{eq:sing}) is a union of hyperplanes of the form  $\hat{\ga}=z_0$ on $L_\bbC$, 
    with $\hat{\ga}\in {\hat{R}}\backslash {\hat{R}}_L$. Moreover 
    \begin{equation}\label{eq:withm'onV}
    \prod_{\hat{\ga}\in {\hat{R}}}\rho(\hat{\ga})^{m'_{\hat{\ga}}}A(\rho;\phi_1,\dots,\phi_k)
    \end{equation}
    is meromorphic in $V_\bbC$ and restricts to (\ref{eq:sing}) on $L_\bbC$.
    Choose one such a hyperplane $H_L\subset L_\bbC$ which is an irreducible component of the 
    singular locus of (\ref{eq:sing}). 
    There are finitely many zeroes $z_1,\dots,z_s$ of $\rho$, and positive coroots 
    $\hat{\ga}_1,\dots,\hat{\ga}_s$ such that for all $i=1,\dots,s$ the hyperplanes $H_i:=\{\gl\in V_\bbC\mid \hat{\ga}_i=z_i\}$ are precisely the components of 
    the singular set of (\ref{eq:withm'onV}) such that  $H_i\cap L_\bbC=H_L$.
    
    Take a polynomial $p$ on $\bbC$ such that $\rho=p\rho_0$ with $\rho_0$ entire, nonzero 
    on the set $\{z_i:i=1,\dots,s\}$, and satisfying (\ref{FE}). 
    Let $r^A=r_p^Ar_0^A$ be the corresponding factorization of $r^A$. We can now rewrite (\ref{eq:withm'onV}) as 
    \begin{equation}
    \frac{1}{r^A_0(\gl)}\prod_{\hat{\ga}\in {\hat{R}}}\rho_0(\hat{\ga}(\gl))^{m'_{\hat{\ga}}}p(\hat{\ga}(\gl))^{m'_{\hat{\ga}}}
    A(p; r^A_{0,1}\phi_1,\dots,r^A_{0,k}\phi_k)|_{L_\bbC}
    \end{equation}
    By our choice of $p$ it follows that $\frac{1}{r^A_0(\gl)}|_{L_\bbC}$ is nonsingular along $H_L$.  It follows that 
    \begin{align*}
    F&:=\prod_{\nu\in \textup{Aff}^\sim(L)}p(\nu(\gl))^{m'(\nu)} 
    A(p;r^A_{0,1}\phi_1,\dots,r^A_{0,k}\phi_k)|_{L_\bbC}\\
    &=\prod_{\nu\in \textup{Aff}^\sim(L)}p(\nu(\gl))^{m'(\nu)} 
    (r^A_p)^{-1}(\gl)A(1;r^A_1\phi_1,\dots,r^A_k\phi_k)|_{L_\bbC}
    \end{align*}
    is singular along $H_L$, in other words there exists $M>0$ such that $(\hat{\ga}_1|_{L_\bbC}-z_1)^MF$ is nonsingular, 
    nonvanishing meromorphic along $H_L$.
    Hence if $\lambda_0\in H_L$ is a generic point, then $(\hat{\ga}_1|_{L_\bbC}-z_1)^MF$ is holomorphic 
    in a neighborhood of $\gl_0\in L_\bbC$ and nonvanishing in $\gl_0$.
    
    Let $N\gg0$ and choose polynomials $f_i\in\bbC[V_\bbC]$ such that up to terms of order $N$, the polynomial 
    equals $r^A_0\phi_i$ at all points of the orbit $W\gl_0$. This is possible. Indeed, take a polynomial $\tilde{e}$ such 
    that $\tilde{e}(\gl_0)=1$ and $\tilde{e}(\mu_0)=0$ for all $\mu_0\in W\gl_0$ with $\mu_0\not=\gl_0$. Then 
    $\tilde{e}^N$ vanishes of order $N$ at the points $\mu_0\in W\gl_0$ unequal to $\gl_0$, and equals $1$ at $\gl_0$.
    In particular $\tilde{e}^N$ is invertible as a formal power series at $\gl_0$, so there exists a polynomial $P_i$ such 
    that the polynomial $e_{i,\gl_0}=P_i\tilde{e}^N$ equals the power series of $\phi_i$ at $\gl_0$ up to order $N$,  
    and vanishes of order $N$ at the other points of the orbit $W\gl_0$. Similarly 
    we can find such polynomials $e_{i,\mu_0}$ approximating the power series of $\phi_i$ at any other point $\mu_0\in W\gl_0$, 
    while vanishing of order $N$ at all $\mu_0'\in W\gl_0$ with $\mu_0'\not=\mu_0$. 
    Then $f_i:=\sum_{\mu_0\in W\gl_0}e_{i,\mu_0}$
    is a polynomial that fits the bill. By (1) and (4) it follows that if $N$ is sufficiently large, the expression 
    \begin{equation}
    (\hat{\ga}_1-z_1)^M\prod_{\nu\in \textup{Aff}^\sim(L)}p(\nu(\gl))^{m'(\nu)} 
    A(p;f_1,\dots,f_k)|_{L_\bbC}
    \end{equation}
    has the same limit as $(\hat{\ga}_1|_{L_\bbC}-z_1)^MF$ for $\gl\to\gl_0$ along any line in $L_\bbC$ not contained in $H_L$ 
    (i.e. not contained in any hyperplane of the form 
    $\{\gl\mid \hat{\ga}(\gl)=\hat{\ga}(z_0)\}\cap L_\bbC$ for a $\hat{\ga}\in {\hat{R}}\backslash {\hat{R}}_L$).  
    This proves that such limit is nonzero, and thus that 
    \begin{equation}
    \prod_{\nu\in \textup{Aff}^\sim(L)}p(\nu(\gl))^{m'(\nu)} 
    A(p;f_1,\dots,f_k)|_{L_\bbC}
    \end{equation}
    is singular at $\gl_0$ since $M>0$. Hence $m$ 
    satisfies $(b)_{pol}$ as desired. 
    \end{proof}
    Let us now study the rational function $A_L(p;f_1,\dots,f_k)|_{|_\bbC}$ for $p\in \bbC[z]$ satisfying (\ref{FE}) and 
    nonvanishing on $\textup{Con}^A(L)$, and $f_1,\dots,f_k\in \bbC[V]$.
    Observe that for $c\in\bbC^\times$ we have $A_L(cp;f_1,\dots,f_k))=c^nA_L(p;f_1,\dots,f_k))$ for some $n\in\bbZ$, thus we may restrict ourselves 
    to $p$ being monic. 
    
    Fix natural numbers $N$ and $M$ with $N$ even, and let $\bbC^N_{FE}[z]$ denote the space of monic polynomials of degree $N$ 
    satisfying (\ref{FE}), and $\bbC_{M,k}[V_\bbC]$ the space of polynomials of degree at most $M$ on $V_\bbC\times\dots\times V_\bbC$. 
    We will now consider the function (with $f=(f_1,\dots,f_k)\in \bbC_{M,k}[V_\bbC]$) $A_L(p;f)(\gl)$ 
    as a rational function on $\bbC^N_{FE}[z]\times \bbC_{M,k}[V_\bbC]\times L_\bbC$, denoted by $\cA_L(p;f;\gl) = A_L(p;f)(\gl)$. 
    We remark that 
    $\bbC^N_{FE}[z]\times \bbC_{M,k}[V_\bbC]\times L_\bbC$ 
    is an affine space; indeed the space $\bbC^N_{FE}[z]$ can be identified with the categorical quotient of $\bbC^n$ by the action of the 
    Weyl group $W(\textup{B}_n)$ of the classical root system of type $\textup{B}_n$ where $N=2n$, which is affine by 
    the invariant theory of real reflection groups. The identification 
    is obtained by sending (in the even degree case) a sequence $(a_1,\dots,a_n)$ to $p=\prod_{i=1}^n(z-\frac{1}{2}+a_i)(z-\frac{1}{2}-a_i)$. 
    Clearly two sequences yield the same 
    image iff they are connected by the action of the $W(\textup{B}_n)$.
    
    Now observe that for any affine linear function $\nu\in\textup{Aff}(L)$ which is not equal to the constant $\frac{1}{2}$, 
    the polynomial $\phi^{(N)}_\nu\in\bbC[\bbC^N_{FE}[z]\times \bbC_{M,k}[V_\bbC]\times L_\bbC]$ 
    given by $\phi^{(N)}_\nu(p,f,\gl):=p(\nu(\gl))$(which only depends on the equivalence class of $\nu$) is irreducible. 
    Indeed, the zero locus of $\phi^{(N)}_\nu$ 
    is given by the union of the hyperplanes $a_i=\frac{1}{2}-\nu(\lambda)$ and $a_i=-\frac{1}{2}+\nu(\lambda)$, which is an orbit under 
    $W(\textup{B}_n)$ of an affine hyperplane. Clearly any $W(\textup{B}_n)$-invariant polynomial which vanishes on any of these affine 
    hyperplanes is divisible by $\phi^{(N)}_\nu$ (by our assumption on $\nu$), 
    so in particular $\phi^{(N)}_\nu$ is irreducible. 
    
    We already know that for any $p$ the singularities of $\cA_L(p;f_1,\dots,f_k)$, with multiplicities, are bounded by 
    the zero set of $r^A_p(\gl)|_{L_\bbC}$. Let $\hat{R}^{L,\sim}$ denote a set of restrictions 
    $\underline{\hat{\ga}}:=\hat{\ga}|_{L_\bbC}$ to $L_\bbC$ of coroots $\ga\in {\hat{R}}$
    representing all equivalence classes $[\underline{\hat{\ga}}]:=\{\underline{\hat{\ga}},1-\underline{\hat{\ga}}\}$ of such restricted coroots. 
    For each $\underline{\hat{\ga}}\in \hat{R}^{L,\sim}_+$, let $m^L(\underline{\hat{\ga}})\in\mathbb{N}$ denote the number 
    of coroots $\gb\in {\hat{R}}$ counted with multiplicity $m_{\hat{\gb}}$ such that $[\underline{\hat{\gb}}]=[\underline{\hat{\ga}}]$. Then 
    \begin{equation}\label{eq:ronL}
    r^A_p(\gl)|_{L_\bbC}=\prod_{\underline{\hat{\ga}}\in \hat{R}^{L,\sim}}p(\underline{\hat{\ga}}(\gl))^{m^L(\underline{\hat{\ga}})}=
    \prod_{\underline{\hat{\ga}}\in \hat{R}^{L,\sim}}(\phi^{(N)}_{\underline{\hat{\ga}}})^{m^L(\underline{\hat{\ga}})}(p,f,\gl)
    \end{equation}
    
    On the other hand we have a unique decomposition of the rational function 
    $\cA_L(p,f,\gl)$ in irreducible factors in $\bbC[\bbC^N_{FE}[z]\times \bbC_{M,k}[V_\bbC]\times L_\bbC]$. 
    The irreducible factors 
    with negative exponents must be factors with positive multiplicity in the above expression of $r^A_p(\gl)|_{L_\bbC}$, 
    so $\cA_L$ decomposes as a product of the form 
    \begin{equation}
    \cA_L|_{(\bbC^N_{FE}[z]\times \bbC_{M,k}[V_\bbC]\times L_\bbC)}
    =\prod_{\underline{\hat{\ga}}\in \hat{R}^{L,\sim}}(\phi^{(N)}_{\underline{\hat{\ga}}})^{-m^e_{M,N}(\underline{\hat{\ga}})}P_{M,N}
    \end{equation}
    where $P_{M,N}$ is a polynomial relatively prime to the $\phi^{(N)}_{\underline{\hat{\ga}}}$. 
    Since $\bbC_{M,k}[V_\bbC]\subset \bbC_{M',k}[V_\bbC]$ if $M\leq M'$, these exponent functions are weakly increasing 
    as a function of $M$, and since for all $N,M$ we have (by (\ref{eq:onV}) and (\ref{eq:ronL})):
    \begin{equation}
    0\leq m^e_{M,N}(\hat{\ga})\leq m^L(\underline{\hat{\ga}})
    \end{equation} 
    they will have limiting values denoted by $m^e_{N}(\underline{\hat{\ga}})$. 
    Combining this with the definition of $m^A_L$ and Lemma \ref{lem:restopol} we see that (by (\ref{eq:onL})) $m^A_L$ is indeed unique. In fact it is now 
    clear that $m$ can be expressed in terms of the 
    above multiplicities as follows.  If $\nu\in\textup{Aff}^\sim(L)$ is not of the form $\nu\in [\underline{\hat{\ga}}]$ for some 
    $\hat{\ga}\in {\hat{R}}$ then $m^A_L(\nu)=0$, while for $\nu\in [\underline{\hat{\ga}}]$ we have:   
    \begin{equation}
    m^A_L(\nu)=\textup{max}(\{m^e_{N}(\hat{\ga})\mid N\textrm{\ even}\}
    \end{equation}
    This implies in particular the uniqueness of $m^A_L$ we wanted to prove. 
    \end{proof}
    \begin{rem}\label{rem:zaris}
    The rational function $\cA_L$ on $\bbC^N_{FE}[z]\times \bbC_{M,k}[V_\bbC]\times L_\bbC$ is completely determined 
    by its restriction to any Zariski open subset of $\bbC^N_{FE}[z]\times \bbC_{M,k}[V_\bbC]\times L_\bbC$. It is then easy to see from the above 
    proof given for any finite set $X\subset \bbC$ such that $\textup{Con}^A(L)\subset X$, the multiplicity function $m^A_L$ on $\textup{Aff}^\sim(L)$ is 
    already determined by (i), (ii) and (iii) where  $\rho$ varies in the space of 
    entire functions satisfying (\ref{FE}) and nonzero on $X$. 
    \end{rem}
    \begin{defn}\label{defn:mult}
    For any affine subspace $L\subset V$, the multiplicity function $m^A_L$ is defined as in Theorem \ref{thm:denbase}
    relative to the set of entire 
    functions $\rho$ which are nonvanishing on $\textup{Con}^A(L)$, or any finite subsets 
    $X\subset \bbC$ with $\textup{Con}^A(L)\subset X$. 
    \end{defn}
    \begin{defn}\label{defn:densets}
    Let $L\subset V$ be an affine subspace. 
    Consider the meromorphic functions $A_1(\rho;\phi)=\Sigma_{W}(\phi)$, $A_2(\rho;\psi)=\Sigma_{u\in W'x}(\psi)$
    and $A_3(\rho;\phi,\psi)=\Sigma_{W}(\phi)\Sigma_{u\in W'x}(\psi)$. These functions satisfy the properties $(i)$-$(iv)$ of 
    Theorem \ref{thm:denbase} with $r^{A_1}(\gl)=r(-\gl)$, $r^{A_2}(\gl)=r(\gl)$ and $r^{A_3}(\gl)=r(-\gl)r(\gl)$ respectively.
    
    For each  $i=1,2$ or $3$, the multiplicity function $m^{A_i}_L$ is thus well defined on the set of $\rho$ 
    nonvanishing on $\textup{Con}^{A_i}(L)$ (or any larger finite set) by Definition \ref{defn:mult}.
    Given $x\in W$ we define the following multiplicity functions:  
    \begin{enumerate}
    \item[(a)] $m^{nc}_{(L,x,\Sigma)}=m^{A_1}_L$\index{$m_{(L,x)}$!$m^{nc}_{(L,x,\Sigma)}$} and 
    $m_{(L,x,\Sigma)}:=\textup{min}(m^{A_1}_L,m^{A_3}_{L})$\index{$m_{(L,x)}$!$m_{(L,x,\Sigma)}$} (the minimum of 
    $m^{A_1}_L$ and $m^{A_3}_{L}$). We write 
    $\textup{Den}_{nc}(L,x,\Sigma)$\index{$\textup{Den}(L)$!$\textup{Den}_{nc}(L,x,\Sigma)$} and  
    $\textup{Den}(L,x,\Sigma)$\index{$\textup{Den}(L)$!$\textup{Den}(L,x,\Sigma)$} for their respective supports 
    (``nc'' stands for ``no cancellation''). 
    Note that $\textup{Den}_{nc}(L,x,\Sigma)$ is independent of $x\in W$ (we shall sometimes simply write 
    $\textup{Den}_{nc}(L,\Sigma)$\index{$\textup{Den}(L)$!$\textup{Den}_{nc}(L,\Sigma)$} accordingly).  
    \item[(b)] $m^{nc}_{(L,x,\Sigma')}=m^{A_2}_{L}$\index{$m_{(L,x)}$!$m^{nc}_{(L,x,\Sigma')}$} 
    and $m_{(L,x,\Sigma')}:=\textup{min}(m^{A_2}_{L},m^{A_3}_{L})$\index{$m_{(L,x)}$!$m_{(L,x,\Sigma')}$}, with their respective supports 
    $\textup{Den}_{nc}(L,x,\Sigma')$\index{$\textup{Den}(L)$!$\textup{Den}_{nc}(L,x,\Sigma')$} and 
    $\textup{Den}(L,x,\Sigma')$\index{$\textup{Den}(L)$!$\textup{Den}(L,x,\Sigma')$}. 
    \item[(c)] Finally we define 
    $m_{(L,x)}:=m^{A_3}_{L}$\index{$m_{(L,x)}$}, and write  $\textup{Den}(L,x)$\index{$\textup{Den}(L)$!$\textup{Den}(L,x)$}
    for its support.
    \end{enumerate}
    All supports are counted with multiplicities. 
    Note that $m_{(L,x)}=m_{(L,x,\Sigma)}+m_{(L,x,\Sigma')}\leq m^{nc}_{(L,x,\Sigma)}+m^{nc}_{(L,x,\Sigma')}$, thus 
    $\textup{Den}(L,x)=\textup{Den}(L,x,\Sigma)\cup \textup{Den}(L,x,\Sigma')\subset \textup{Den}_{nc}(L,x,\Sigma)\cup \textup{Den}_{nc}(L,x,\Sigma')$.
    We write $\textup{Den}(L)=\textup{Den}(L,1)$\index{$\textup{Den}(L)$}.
    
    We lift the multiplicity functions $m_{(L,x,\Sigma)}$ etc. to functions on $\textup{Aff}^*(L)$ which are invariant under 
    the action $Y\to 1-Y$, and denote these lifts by the same notations. 
    Accordingly we lift $\textup{Den}(L,x,\Sigma)$ etc. from $\textup{Aff}^*(L)/\sim$ to a subset of $\textup{Aff}^*(L)$ invariant for $Y\to 1-Y$.  
    \end{defn}
    Our notion of denominators has an application to the problem of iterated restrictions of the meromorphic functions $A(\rho,\phi)$ in Theorem \ref{thm:denbase}. 
    \begin{prop}\label{prop:rest}
    Let $A(\rho,\phi)$ be as in Theorem \ref{thm:denbase}. 
    Let $L\subset M_i\subset V$ be affine subspaces and put 
    $\textup{Con}^A_{M_i}(L):=\{\nu(L),1-\nu(L)\in\bbC\mid m^A_{M_i}(\nu)\not=0\mathrm{\ and\ }\nu|_L\mathrm{\ is\ constant}\}$
    \index{$\textup{Con}^A(L)$! $\textup{Con}^A_{M_i}(L)$} for $i=1,2$. If $\rho$ is nonvanishing in 
    $\textup{Con}^A(M_1)\cup \textup{Con}^A(M_2)\cup \textup{Con}^A_{M_1}(L)\cup \textup{Con}^A_{M_2}(L)$
    then the iterated restrictions $(A(\rho,\phi)|_{M_{1,\bbC}})|_{L_{\bbC}}$ and $(A(\rho,\phi)|_{M_{2,\bbC}})|_{L_{\bbC}}$ exist and are equal as meromorphic 
    functions on $L_{\bbC}$.
    \end{prop}
    \begin{proof}
    The existence of both expressions as meromorphic functions on $L_\bbC$ is clear from the definitions. 
    We have a unique decomposition $\rho=\rho_0.p_0$ where $\rho_0$ is entire and nonvanishing on $\textup{Con}^A(L)$, and 
    $p_0$ is a monic polynomial satisfying (\ref{FE}) which has zeroes only in $\textup{Con}^A(L)$. Say $p_0\in \bbC^N_{FE}[z]$ 
    Define $\cU\subset \bbC^N_{FE}[z]\times L_{\bbC}$ by 
    \begin{multline*}
    \cU=\{(p,\gl)\in \bbC^N_{FE}[z]\times L_{\bbC}\mid p(\mu)\not=0, \\\forall \mu\in 
    \textup{Con}^A(M_1)\cup \textup{Con}^A(M_2)\cup \textup{Con}^A_{M_1}(L)\cup \textup{Con}^A_{M_2}(L)\}
    \end{multline*}
    and $\cU_L:=\{(p,\gl)\in \bbC^N_{FE}[z]\times L_{\bbC}\mid p \mathrm{\ is\ nonvanishing\ in\ } 
    \textup{Con}^A(L)\}$, then clearly $\cU_L\subset \cU\subset  \bbC^N_{FE}[z]\times L_{\bbC}$ are Zariski open subsets. 
    Consider the expressions 
    \begin{equation*}
    \frac{\prod_{\nu}(p\rho_0)(\nu)^{m_{M_1}(\nu)}A(p\rho_0,\phi)|_{M_{1,\bbC}}}{\prod_{\nu}(p\rho_0)(\nu)^{m_{M_1}(\nu)}}|_{L_{\bbC}}\mathrm{\ and\ }
    \frac{\prod_{\nu}(p\rho_0)(\nu)^{m_{M_2}(\nu)}A(p\rho_0,\phi)|_{M_{2,\bbC}}}{\prod_{\nu}(p\rho_0)(\nu)^{m_{M_2}(\nu)}}|_{L_{\bbC}}.
    \end{equation*}
    Both are quotients of entire functions on $\cU$, with denominators which do not vanish 
    identically. Both represent the same meromorphic function $A(p\rho_0,\phi)|_{L_{\bbC}}$ on $\cU_L\subset \cU$, 
    hence on $\cU$. Since the denominators do not vanish identically on $\{p_0\}\times L_{\bbC}\subset \cU$ 
    we can evaluate on this subvariety to obtain the result. 
    \end{proof}
    The following ``heritablity property'' of denominator sets plays an important role.
    \begin{prop}\label{prop:herit}
    Suppose that $L\subset M$ are nested affine subspaces, and $x\in W$.
    Then $\textup{Den}_{nc}(L,\Sigma)\subset \textup{Den}_{nc}(M,\Sigma)|_L$, 
    $\textup{Den}_{nc}(L,x,\Sigma')\subset \textup{Den}_{nc}(M,x,\Sigma')|_L$, and 
    $\textup{Den}(L,x)\subset \textup{Den}(M,x)|_L$. 
    \end{prop}
    \begin{proof}
    Consider the surjective map $\pi^M_L:\textup{Aff}(M)/\sim\to\textup{Aff}(L)/\sim$ obtained by restriction. 
    Let $m^{A_i,M}_L:\textup{Aff}(L)/\sim\to \bbZ_{\geq 0}$ denote the integrals over the fibers of $\pi^M_L$ 
    of the functions $m^{A_i}_M$. 
    By Theorem \ref{thm:denbase} we have $m^{A_i}_L\leq m^{A_i,M}_L$ for all $i$, whence the result. 
    \end{proof}
    \begin{lem}\label{lem:trans}
    Let $L\subset V$ be an affine subsapce, let $x\in W$, and $u\in W'$. 
    Let $W_L\subset W$ denote the pointwise fixator of $L$, and $v\in W_L$. 
    Then  
    $\textup{Den}(L,x,\Sigma)=\textup{Den}(L,ux,\Sigma)=\textup{Den}(L,xv,\Sigma)$ and 
    $\textup{Den}(L,x,\Sigma')=\textup{Den}(L,ux,\Sigma')=\textup{Den}(L,xv,\Sigma')$. Also note:  
    $\textup{Den}_{nc}(L,x,\Sigma)=\textup{Den}_{nc}(L,\Sigma)$ is independent of $x$, and 
    $\textup{Den}_{nc}(L,x,\Sigma')=\textup{Den}_{nc}(L,ux,\Sigma')=\textup{Den}_{nc}(L,xv,\Sigma')$.
    \end{lem}
    \begin{proof}
    These transformation rules are obvious. 
    \end{proof}
    Hence we can replace $x$ in Lemma \ref{lem:trans} by 
    any representative $d$ of the double coset $W'xW_0$. 
    \begin{lem}\label{lem:transmove}
    If $L\subset V$ is an affine subspace, and $L=wL_0$ for some $w\in W$,
    then (viewing $w:L_0\to L$ as an isomorphism) for all $x\in W$ we have 
    $w^*(\textup{Den}(L,x))=\textup{Den}(L_0,xw)$. Here, $w^*$\index{$w^*$ pull back via $w$} denotes the pull back via $w|_{L_0}$.
    \end{lem}
    
    \begin{proof}
    This is clear from the transformation formula (with $\lambda\in L_0$ and $\phi,\psi\in\cP^\fR(V_{\bbC})$):
    \[
    w^*((\Sigma_W(\phi)\Sigma_{W'}(\psi))|_L)(\lambda)=
    (\Sigma_W(\phi)\Sigma_{W'w}(\psi))|_{L_0}(\lambda).\]
    \end{proof}
    \begin{rem}\label{rem:warn}
    This transformation rule is true for $\textup{Den}(L,w)$, but not in general 
    for the subsets $\textup{Den}(L,w,\Sigma)$ and $\textup{Den}(L,w,\Sigma')$ seperately.   
    Indeed, in the situation of Lemma (\ref{lem:trans}) we have the transformation rules (cocycle 
    relations), where $\phi,\psi$ are in $\cP^\fR(V_\bbC)$:
    \begin{equation}\label{eq:trans}
    \Sigma_W(\phi)|_{L}=(w^{-1})^*\left(\frac{r(w\lambda)}{r(\lambda)} \Sigma_W(\phi)|_{L_0}\right)
    \end{equation}
    and 
    \begin{equation}\label{eq:trans'}
    \Sigma_{W'}(\psi)|_{L}=(w^{-1})^*\left(\frac{r(\lambda)}{r(w\lambda)} \Sigma_{W'w}(\psi)|_{L_0}\right).
    \end{equation}
    However, in the special case that $w\in W({\hat{R}}_{L_0})$ then 
    ${\hat{R}}(w)$ consists of coroots which are constant on $L_0$ (thence on $L$). In that case,  
    $w^*(\textup{Den}(L,x,\Sigma))=\textup{Den}(L_0,xw,\Sigma)$ and $w^*(\textup{Den}(L,x,\Sigma'))=\textup{Den}(L_0,xw,\Sigma')$.
    (Indeed $\frac{r(w\lambda)}{r(\lambda)}|_{L_0}$ is a nonzero constant now, provided $\rho$ does not vanish on $\textup{Con}^A(L)$).
    \end{rem}
    \begin{defn}\label{defn:noncst}
    Let $L\subset M$ be nested affine subspaces. Let ${\hat{\fz}}_L$ be the underlying vector space of $L$. We denote by 
    $\textup{Den}^{{\hat{\fz}}_L}_{nc}(M,w)\subset \textup{Den}_{nc}(M,w)$\index{$\textup{Den}(L)$!$\textup{Den}^{{\hat{\fz}}_L}_{nc}(M,w)$} 
    the subset of $\textup{Den}_{nc}(M,w)$ which are 
    not constant on $L$ (and similar notations for other sets of denominators). 
    \end{defn}
    \begin{rem}\label{rem:warn_herit}
    We can combine Proposition \ref{prop:herit} with the formula in Remark \ref{rem:warn} (with similar arguments): Let $L\subset M$, 
    let $w_M\in W({\hat{R}}_L)$ and $M=w_M(M_0)$. Then 
    $\textup{Den}_{nc}(L,w,\Sigma)\subset \textup{Den}^{{\hat{\fz}}_L}_{nc}(M,w,\Sigma)|_L=(w_M^{-1})^*(\textup{Den}^{{\hat{\fz}}_L}_{nc}(M_0,ww_M,\Sigma)|_{w_M^{-1}(L)})$
    and also: 
    $\textup{Den}_{nc}(L,w,\Sigma')\subset \textup{Den}_{nc}^{{\hat{\fz}}_L}(M,w,\Sigma')|_L=(w_M^{-1})^*(\textup{Den}_{nc}^{{\hat{\fz}}_L}(M_0,ww_M,\Sigma')|_{w_M^{-1}(L)})$. 
    \end{rem}
    \subsection{Standard and essentially standard pole spaces}
    Let ${\hat{R}}_+\subset {\hat{R}}$\index{$R$! ${\hat{R}}_+$} be the set of positive coroots, with basis $\hat{\Delta}$\index{$\hat{\Delta}$}.
    We have an action of $W'\times W$ on the set of pairs $(L,x)$ with $L$ a pole space and $w\in W$ given by $(u,w)(L,x)=(wL,uxw^{-1})$. 
    
    We call a pole space $L$ \emph{essentially standard}\index{$L$ pole space! $L$ essentialy standard} if $\hat{R}_L\subset {\hat{R}}$ is a standard Levi subsystem, 
    and \emph{standard}\index{$L$ pole space! $L$ standard} if $L$ is essentially standard and $\hat{\ga}(L)\geq 0$ for all $\ga\in {\hat{R}}_{L,+}$. 
    If $L$ is a pole space then there exist standard pairs $(L_0,w)$ \index{$L$ pole space! $(L_0,w)$ standard pair} representing $L$, i.e. $L_0$ is a standard 
    pole space, $w\in W$ is such that $L=w(L_0)$. 
    If $L$ is an essentially standard pole space then there exists a $w\in W(\hat{\Delta}(L))$\index{$W(\hat{\Delta}(L))$} (with 
    $\hat{\Delta}(L)\subset {\hat{R}}_{L,+}$\index{$\hat{\Delta}$! $\hat{\Delta}(L)$}
    the basis of simple coroots of ${\hat{R}}_L$), and a standard pole space $L_0$ such that the standard pair $(L_0,w)$ represents $L$ 
    (and in this case $\hat{\Delta}(L)=\hat{\Delta}(L_0)$).
    More generally:
    \begin{prop}\label{prop:essstdhull}
    Let $L$ be an essentially standard pole space, and $M\supset L$ a 
    pole space enveloping $L$. Then there exists a $w_M\in W(\hat{\Delta}(L))$ and 
    a standard pole space $M_0$ such that $M=w_M(M_0)$. 
    \end{prop}
    \begin{proof}
    The set ${\hat{R}}(M)$ of coroots which are constant on $M$ form a Levi subsystem 
    inside ${\hat{R}}(L)$. Hence there exists a $u\in W(\hat{\Delta}(L))$ such that 
    $u({\hat{R}}(M))={\hat{R}}(u(M))\subset {\hat{R}}(L)$ is standard, and a 
    $v\in  W(\hat{\Delta}(u(M)))\subset W(\hat{\Delta}(L))$ such that alle 
    coroots in $\hat{\Delta}(u(M))$ are nonnegative 
    on $vu(M)$. Take $M_0=vu(M)$ and $w_M=(vu)^{-1}$.
    \end{proof}
    \begin{prop}\label{prop:tau_L}
    If $L$ is a standard residual pole space we define $\tau_L:=-w_L$\index{$\tau_L$} where $w_L\in W(\hat{\Delta}(L))$ denotes 
    the longest element. Then $\tau_L(L)=L$ and if $v=c_L+x\in L$ then $\tau_L(v)=c_L-x$. Put 
    $\Sigma_{W,\tau}(\phi):=\Sigma_{W}(\phi)\circ(-\textup{id}_V)$\index{$\Sigma_{W}(\phi)$!$\Sigma_{W,\tau}(\phi)$}, 
    in other words, 
    \begin{equation*}
    \Sigma_{W,\tau}(\phi)(\gl)=\sum_{w\in W}c(w\gl)\frac{r(w\gl)}{r(\gl)}\phi(w\gl)
    \end{equation*}
    This function satisfies the requirements of Theorem \ref{thm:denbase} and thus we can define 
    its set of denominators $\textup{Den}_{nc}(L,\Sigma_\tau)$\index{$\textup{Den}(L)$!$\textup{Den}_{nc}(L,\Sigma_\tau)$}
    in a similar fashion as $\textup{Den}_{nc}(L,\Sigma)$. 
    Then we have $\tau_L(\textup{Den}_{nc}(L,\Sigma))=\textup{Den}_{nc}(L,\Sigma_\tau)$.
    \end{prop}
    \begin{proof}
    Note $\Sigma_{W,\tau}|_L=((\frac{r(\gl))}{r(w_L(\gl))}\circ\tau_L))|_L(\Sigma_{W}\circ\tau_L))|_L=
    (\frac{r(\gl)}{r(w_L(\gl))})|_L(\Sigma_{W}\circ\tau_L))|_L$. The identity of denominator sets follows trivially
    since the cocycle factor is a nonzero constant on $L$.
    \end{proof}
    \subsection{Positive denominators of essentially standard pole spaces}
    Let $L$ be an affine subspace, and $X\in \textup{Aff}^*(L)$. We call the unique element $\nabla_X^L\in \textup{Aff}^*(L)$\index{$\nabla_X^L$}
    such that $\nabla_X^L(c_L)=0$ and such that $X-\nabla_X^L$ is constant on $L$ the \emph{gradient} of $X$ on $L$. 
    
    Let $L$ be an essentially standard pole space. We define $L^+\subset L$ by 
    \begin{equation*}
    L^+:=\{x\in L\mid \forall
    \hat{\ga}\in {\hat{R}}_+\backslash {\hat{R}}_+(L):\ \nabla^L_{\hat{\ga}}(x)>0\}
    \end{equation*}
    \index{$L$ pole space!$L_+$ positive cone}
    and we define $\textup{Aff}_+^*(L)\subset \textup{Aff}^*(L)$\index{$\textup{Aff}(L)$!$\textup{Aff}_+^*(L)$}, 
    the subset of positive affine 
    functions, by:
    \begin{equation*}
    \textup{Aff}_+^*(L)=\{X\in \textup{Aff}^*(L)\mid \nabla_X^L \mathrm{\ is\ positive\ on\ } L^+\}
    \end{equation*}
    The nonconstant restictions of coroots to some essentially standard pole space $L$ either 
    have a positive gradient or a negative gradient. 
    \begin{defn} Let $L$ be an essentially standard pole space. 
    Define $\textup{Den}(L,w)_+=\{Y\in \textup{Den}(L,w)\mid Y\in \textup{Aff}_+^*(L)\}$ and 
    $\textup{Den}(L,w)_-=\{Y\in \textup{Den}(L,w)\mid -Y\in\textup{Aff}_+^*(L)\}$. 
    Observe that every equivalence class of denominators in $\textup{Den}(L,w)$ on $L$ has a unique 
    representative in  $\textup{Den}(L,w)_+$ and a unique representative in $\textup{Den}(L,w)_-$.
    We define two sets of affine linear functions on $L$:
    \begin{align*}
    D_L&:=\{Y\in\textup{Aff}_+^*(L)\mid Y(c_L)\geq 1\}\index{$D_L$}\\
    D'_L&:=\{Y\in\textup{Aff}_+^*(L)\mid Y(c_L)\leq 0\}\index{$D_L$!$D'_L$}
    \end{align*}
    Observe that the set of equivalence classes $[D_L]$ of elements in $D_L$ 
    is disjoint from the set of equivalence classes $[D'_L]$ of elements in $D'_L$.
    \end{defn}
    We have the following useful refinement of Remark \ref{rem:warn_herit}:
    \begin{prop}\label{prop:IS}
    Let $L\subset M$ be pole spaces.  
    Choose $w\in W$ such that $(L,1)=w(L_0,w)$ with $L_0$ standard. Put $M^0=w^{-1}(M)\supset L_0$.
    Let  $w_{M^0}\in W(\hat{\Delta}(L_0))$ and $M_0$ a standard pole space such that $M^0=w_{M^0}(M_0)$.
    We have: 
    \begin{align*}
    \textup{Den}^{{\hat{\fz}}_{L_0}}_{nc}(M^0,w,\Sigma)_+&=(w_{M^0}^{-1})^*(\textup{Den}^{{\hat{\fz}}_{L_0}}_{nc}(M_0,ww_{M^0},\Sigma)_+) \mathrm{\ and}\\
    \textup{Den}^{{\hat{\fz}}_{L_0}}_{nc}(M^0,w,\Sigma')_+&=(w_{M^0}^{-1})^*(\textup{Den}^{{\hat{\fz}}_{L_0}}_{nc}(M_0,ww_{M^0},\Sigma')_+)
    \end{align*}
    so that in particular: 
    \begin{align*}
    \textup{Den}_{nc}(L_0,w,\Sigma)_+&\subset (w_{M^0}^{-1})^*(\textup{Den}^{{\hat{\fz}}_{L_0}}_{nc}(M_0,ww_{M^0},\Sigma')_+)|_{L_0}\mathrm{\ and}\\
    \textup{Den}_{nc}(L_0,w,\Sigma')_+&\subset (w_{M^0}^{-1})^*(\textup{Den}^{{\hat{\fz}}_{L_0}}_{nc}(M_0,ww_{M^0},\Sigma')_+)|_{L_0}
    \end{align*}
    \end{prop}
    \index{$\textup{Den}(L)$!{$\textup{Den}_{nc}(L_0,x,\Sigma)_+$}}
    \begin{proof}
    Given $L_0\subset M^0$ Proposition \ref{prop:essstdhull} implies the existence of a $w_{M^0}\in W(\hat{\Delta}(L_0))$ as stated. 
    This element satisfies $w_{M^0}^{-1}({\hat{\fz}}_{L_0})={\hat{\fz}}_{w_{M^0}^{-1}(L_0)}={\hat{\fz}}_{L_0}$ and $w_{M^0}(w_{M^0}^{-1}(L_0)^+)=L_0^+$. 
    In combination with Remark \ref{rem:warn_herit} this implies the result.   
    \end{proof}
    \subsection{The denominators of good regular pole spaces}\label{sub:reg}
    A pole space $L$ is called \emph{regular} if no coroots vanish identically on $L$. 
    \begin{lem}
    A pole space is regular iff $\hat{P}_L$ is a base of the Levi subsystem ${\hat{R}}_L\subset {\hat{R}}$ of coroots which are constant on $L$.
    In particular a regular pole space is residual. 
    \end{lem}
    \begin{proof}
    We may assume without loss of generality that $L$ is standard.
    It is clear that $L$ is residual, since there are no zero hyperplanes of $\Omega_r$ vanishing identically on $L$, 
    while the number of $\Omega_r$ pole hyperplanes (even those of $\Omega$, which is a priori a smaller set) is at least the codimension of $L$. 
    Hence $c_L$ is integral on $\hat{R}_L$ (since it is well known that $2c_L$ is in the $W(\hat{R}_L)$ orbit of the weighted Dynkin diagram of 
    a \emph{distinguished} nilpotent orbit of the standard Levi subalgebra $\hat{g}^L=\hat{\fg}(\hat{R}_L)$\index{$\hat{\fg}^L$} 
    whose root system is $\hat{R}_L$, 
    in combination with Jantzen's result on the evenness of distinguished nilpotent orbits). 
    Now $\hat{P}_L\subset \hat{R}_{L,+}$ spans $\bbR {\hat{R}}_L$, we have $\hat{\gb}(c_L)=1$ for all $\hat{\gb}\in\hat{P}_{L}\subset \hat{R}_{L,+}$, 
    while $\hat{\ga}(c_L)\in\bbZ_+$ for all $\ga\in \hat{\Delta}(L)$ by the above integrality and regularity of $L$. This implies that 
    $\hat{P}_L= \hat{\Delta}(L)$.
    \end{proof}
    \begin{defn}\label{defn:good} 
    Let $L$ be regular pole space and $w\in W$. We call the pair $(L,w)$ \emph{good}\index{$L$ pole space!$(L,w)$ good regular pair} 
    if $w(\hat{P}_L)\subset {\hat{R}}_+\cup \hat{R}'$. We say $L$ is good\index{$L$ pole space!$L$ good regular pole space} 
    if $(L,e)$ is good, and more generally, 
    we say $L$ is $w$-good if $(L,w)$ is good\index{$L$ pole space!$L$ $w$ good regular pole space}. 
    \end{defn}
    \begin{lem}\label{lem:pospole} Let $L$ be a regular pole space, and $(L_0,w)$ a standard pair representing $L$.  
    Let $w=ud$ with $u\in W'$ and $d$ minimal in $W'd$.Then $L$ is good iff $d(\hat{P}_{L_0})\subset {\hat{R}}_+$. 
    \end{lem}
    \begin{proof}
    We prove $\implies$ (the converse is obvious). 
    Suppose that  $\hat{\ga}\in \hat{P}_{L_0}\subset  {\hat{R}}_+$ and that $d(\hat{\ga})=-\hat{\gb}$ is a negative root. 
    Then $\hat{\gb}\in {\hat{R}}(d^{-1})\subset \hat{R}_+\backslash \hat{R}'_+$.
    Therefore $w(\hat{\ga})=-u(\hat{\gb})\in -(\hat{R}_+\backslash \hat{R}'_+)$, contradicting the assumption.
    \end{proof}
    Let ${\hat{\fg}}$\index{${\hat{\fg}}$} 
    be the Lie algebra of $\hat{G}$\index{$\hat{G}$}, equipped with a choice of maximal torus $V_\bbC={\hat{\ft}}$ and Borel 
    $\hat{\fb}\supset{\hat{\ft}}$\index{${\hat{\fg}}$!$\hat{\fb}$}\index{${\hat{\fg}}$!${\hat{\ft}}$}, 
    and $L\subset V$ denote a standard residual pole space. 
    Write ${\hat{\fp}}^L\subset\hat{\fg}$\index{${\hat{\fp}}^L$} for the standard parabolic subalgebra with 
    Levi subalgebra ${\hat{\fg}}^L$, and put 
    ${\hat{\fz}}^L=Z({\hat{\fg}}^L)$\index{${\hat{\fz}}^L$}. 
    Then ${\hat{\ft}}={\hat{\fz}}^L+{\hat{\ft}}^L$ where ${\hat{\ft}}^L$ is a maximal toral subalgebra of ${\hat{\fg}}^L_{ss}$, and  
    $L=c_L+{\hat{\fz}}^L$ with $c_L\in {\hat{\ft}}^L$ the Weyl vector of $\hat{R}_{L,+}$. 
    Fix an $\mathfrak{sl}_2$-triple $\langle e,f,h\rangle\subset {\hat{\fg}}^L$ with $c_L=\frac{h}{2}$ and put 
    ${\hat{\fq}}^L:=\langle e,f,h\rangle\oplus {\hat{\fz}}^L\subset {\hat{\fg}}^L$\index{${\hat{\fq}}^L$}. 
    Then ${\hat{\fg}}^L$ and  thus ${\hat{\fg}}/{\hat{\fg}}^L$ are modules over the reductive algebra ${\hat{\fq}}^L$ (independent of the choice of ${\hat{\fq}}^L$).   
    Choose ${\hat{\fp}}^L$, a semi-standard parabolic subalgebra with Levi decomposition ${\hat{\fp}}^L={\hat{\fg}}^L\oplus {\hat{\fu}}^L$.  
    The unipotent radical ${\hat{\fu}}^L$ of ${\hat{\fp}}^L$ is an ${\hat{\fq}}^L$-module (which does depend on the choice of  ${\hat{\fp}}^L$).
    Call two coroot spaces in ${\hat{\fg}}$ $L$-equivalent if 
    the restrictions of these coroots to ${\hat{\fz}}^L$ differ by an integer. Given an $L$-equivalence class of coroots,  
    the subspace of ${\hat{\fu}}^L$ spanned by the corresponding coroot spaces is a ${\hat{\fq}}^L\subset {\hat{\fg}}^L$ submodule. 
    Let $\tau^L:{\hat{\fg}}^L\to {\hat{\fg}}^L$\index{$\tau_L$!$\tau^L$} be the Cartan involution. Let $\gc:=\textup{Ad}(g)\in \textup{Inn}(\hat{G}^L)$ be such that 
    $\tau_L:=\gc\tau^L$ restricts to the identity on $\langle e,f,h\rangle$ (and to $-\textup{id}$ on ${\hat{\fz}}^L$). 
    Then $\tau_L:{\hat{\fg}}/{\hat{\fg}}^L\to {\hat{\fg}}/{\hat{\fg}}^L$ is an equivalence from $\textup{ad}\circ \tau_L|_{{\hat{\fq}}^L}$ to $\textup{ad}|_{{\hat{\fq}}^L}$, 
    sending ${\hat{\fu}}^L\to \overline{{\hat{\fu}}}^L$. So if 
    \begin{equation}\label{eq:phwp1}
    {\hat{\fu}}^L:=\bigoplus_{i=1}^k\textup{Sym}^{m_i}(\bbC^2)\otimes \bbC_{\nabla^L_{\hat{\ga}_i}}\index{${\hat{\fu}}^L$}
    \end{equation}
    is the irreducible decomposition of ${\hat{\fu}}^L$ for $\textup{ad}|_{{\hat{\fq}}^L}$ (where $\bbC_{\nabla^L_{\hat{\ga}_i}}$ denotes 
    the restriction of a coroot space ${\hat{\fg}}_{\hat{\ga}_i}\subset {\hat{\fu}}^L$ to ${\hat{\fz}}^L$), then 
    \begin{equation}\label{eq:nhwp1}
    \overline{{\hat{\fu}}}^L:=\bigoplus_{i=1}^k\textup{Sym}^{m_i}(\bbC^2)\otimes\bbC_{-\nabla^L_{\hat{\ga}_i}}\index{$\overline{{\hat{\fu}}}^L$}
    \end{equation}
    is the irreducible decomposition of $\overline{{\hat{\fu}}}^L$ for $\textup{ad}|_{{\hat{\fq}}^L}$
    Observe that we can choose $\hat{\ga}_i$ such that $\hat{\ga}_i(c_L)=\frac{m_i}{2}$ in (\ref{eq:phwp1}) and (\ref{eq:nhwp1}), 
    in other words such that $\hat{\ga}_i|_L=\nabla^L_{\hat{\ga}_i}+\frac{m_i}{2}$ (in case (\ref{eq:phwp1})) and  
    $\tau_L(\hat{\ga}_i)|_L=-\nabla^L_{\hat{\ga}_i}+\frac{m_i}{2}$ (in case (\ref{eq:nhwp1})). We say that $\hat{\ga}_i|_L=\nabla^L_{\hat{\ga}_i}+\frac{m_i}{2}$
    is the highest weight of the $i$-th irreducible summand of the ${\hat{\fq}}^L$-module $(\ref{eq:phwp1})$, and denote the corresponding highest 
    weight module of ${\hat{\fq}}^L$ by $V^L(\hat{\ga}_i|_L)\subset {\hat{\fu}}^L$. Similarly we have $V^L(\tau_L(\hat{\ga}_i)|_L)\subset \overline{{\hat{\fu}}}^L$. 
    The main result of this section relates the denominator sets $\textup{Den}_{nc}(L,w,\Sigma)=\textup{Den}_{nc}(L,\Sigma)$ and  
    $\textup{Den}_{nc}(L,w,\Sigma')$ to this $\mathfrak{sl}_2$-
    module in the case of regular pole spaces:  
    \begin{prop}\label{lem:incl}
    Let $L$ be a standard \emph{regular} pole space. Using the above notations, write $L=c_L+\fz^L$, with $c_L=\frac{1}{2}h$, 
    and choose $\hat{\fp}^L$ such that $\hat{\fu}^L\subset \hat{\fu}$.
    Let $w\in W$.  Then $\textup{Den}_{nc}(L,w,\Sigma) 
    =\{ \{\nabla^L_{\hat{\ga}_i}+\frac{m_i}{2}+1,-\nabla^L_{\hat{\ga}_i}-\frac{m_i}{2}\}\mid i=1,\ldots,k\}\subset D_L$. 
    In other words, the $\nu\in\textup{Aff}^\sim(L)$ appears in $\textup{Den}_{nc}(L,w,\Sigma)$ only if $\nu$ is of the form $\nu=\hat{\ga}|_L+1$ 
    for some $\hat{\ga}\in\hat{R}_+$, with multiplicity $m_{L}(\hat{\ga}+1)|_{L})$ equal to the multiplicity 
    $\textup{mult}_{\hat{\fu}^{L}}^{\hat{\fq}^{L}}(V^{L}(\hat{\ga}|_{L}))$  
    of the $\hat{\fq}^L$-highest weight module $V^L(\hat{\ga}|_L)$ in $\hat{\fu}^L$.
    
    Write $w=ud$ with $u\in W'$ and $d\in W$ and assume $(L,w)$ is good. 
    Then $\textup{Den}_{nc}(L,w,\Sigma')\subset \textup{Den}_{nc}(L,\Sigma_\tau)$ where 
    $\textup{Den}_{nc}(L,\Sigma_\tau)=\{\{\nabla^L_{\hat{\ga}_i}-\frac{m_i}{2},-\nabla^L_{\hat{\ga}_i}+\frac{m_i}{2}+1\}\mid i=1,\ldots,k\}\subset D_L'$. 
    Thus $\textup{Den}_{nc}(L,\Sigma)_+$ consists of the highest $\hat{\fq}^L$-weights of $\hat{\fu}^L$ plus $1$, 
    $\textup{Den}_{nc}(L,\Sigma_\tau)_-$ is the set of highest 
    $\hat{\fq}^L$-weights of $\overline{\hat{\fu}}^L$ plus $1$, and 
    $\textup{Den}_{nc}(L,w,\Sigma')_-\subset \textup{Den}_{nc}(L,\Sigma_\tau)_-$.
    
    \end{prop}
    \begin{proof}
    Since $L$ is regular the individual summands of $\Sigma_{W}(\phi), \Sigma'_{W'w}$ and $\Sigma_{W,\tau}$ (cf. Proposition \ref{prop:tau_L}) 
    can be restricted to $L$ as meromorphic functions.
    The individual summands of these have a nice structure due to 
    \begin{equation}\label{eq:Sigmacr}
    c(-w\gl)\frac{r(\gl)}{r(w\gl)}=\prod_{\ga\in R(w)}\frac{\hat{\ga}(\gl)+1}{\hat{\ga}(\gl)}\prod_{\gb\in \hat{R}_+\backslash R(w)}\frac{\hat{\gb}(\gl)-1}{\hat{\gb}(\gl)}
    \prod_{\gc\in \hat{R}(w)}\frac{\rho(\hat{\gc}(\gl))}{\rho(\hat{\gc}(\gl)+1)}
    \end{equation}
    and 
    \begin{equation}
    c(w\gl)\frac{r(w\gl)}{r(\gl)}=\prod_{\ga\in R(w)}\frac{\hat{\ga}(\gl)-1}{\hat{\ga}(\gl)}\prod_{\gb\in \hat{R}_+\backslash R(w)}\frac{\hat{\gb}(\gl)+1}{\hat{\gb}(\gl)}
    \prod_{\gc\in \hat{R}(w)}\frac{\rho(\hat{\gc}(\gl)+1)}{\rho(\hat{\gc}(\gl))}
    \end{equation}
    and finally for $w=ud$ with $u\in W'$: 
    \begin{equation}\label{eq:cp}
    c'(ud\gl)\frac{r(ud\gl)}{r(\gl)}
    =\prod_{\hat{\ga}\in d^{-1}(\hat{R}(u))}\frac{\hat{\ga}(\gl)-1}{\hat{\ga}(\gl)}\prod_{\hat{\gb}\in d^{-1}(\hat{R}'_+\backslash \hat{R}(u))}\frac{\hat{\gb}(\gl)+1}{\hat{\gb}(\gl)}
    \prod_{\gc\in \hat{R}(ud)}\frac{\rho(\hat{\gc}(\gl)+1)}{\rho(\hat{\gc}(\gl))}
    \end{equation} 
    The first equation implies that this summand of $\Sigma_{W}|_L$ is relevant on $L$ iff $\hat{P}_L\subset \hat{R}(w)$.
    The second equation implies that this summand of $\Sigma_{W,\tau}|_L$ is relevant on $L$ iff $\hat{P}_L\cap \hat{R}(w)=\emptyset$.
    The third equation implies that if $L$ is $w$-good then the term for $w=ud$ is relevant in $\Sigma'_{W'd}|_L$ only if it is also relevant for $\Sigma_{W,\tau}$,  
    since $\hat{R}(w)=\hat{R}(d)\sqcup d^{-1}\hat{R}(u)$ and $\hat{P}_L\cap \hat{R}(d)=\emptyset$ by $w$-goodness.
    Therefore it is enough to treat the cases $\Sigma_{W}|_L$ and $\Sigma_{W,\tau}|_L$. 
    By Proposition \ref{prop:tau_L} it is enough to treat the case of $\Sigma_{W}$. Observe that if $w$ is relevant, 
    then $\hat{\fu}^L(w):=\sum_{\ga\in R(w)}\hat{\fg}_{\hat{\ga}}\cap \hat{\fu}^L$ is a $\langle e,h\rangle\times \fz^L\subset \hat{\fq}^L$-submodule of $\hat{\fu}^L$. 
    Hence the $\langle e,h\rangle\times \fz^L$-highest weights of $\hat{\fu}^L(w)$ are also $\hat{\fq}^L$-highest weights of $\hat{\fu}^L$. 
    The obvious cancellation in the third factor of (\ref{eq:Sigmacr}) leaves as denominator a product of the form $\prod\rho(\hat{\ga}+1)$ where 
    the product runs over the $\langle e,h\rangle\times \fz^L$-highest weights of $\hat{\fu}^L(w)$, which by the above is a subset of the set of 
    $\hat{\fq}^L$-highest weights of $\hat{\fu}^L$ proving the result. 
    \end{proof}
    \begin{prop}\label{prop:essregst}
    Let $L$ be a regular essentially standard pole space and let $w\in W$.
    Then $\textup{Den}(L,w,\Sigma)_+\subset \textup{Den}_{nc}(L,w,\Sigma)_+\subset  D_L$. 
    If $(L,w)$ is good then we have $\textup{Den}(L,w,\Sigma')_+\subset \textup{Den}_{nc}(L,w,\Sigma')_+\subset D'_L$, and 
    moreover, if $X-c\in \textup{Den}(L,w,\Sigma')_+$ then $X+c+1\in \textup{Den}(L,1,\Sigma)_+$.
    \end{prop}
    \begin{proof}
    Write $L=w_L(L_0)$ with $L_0$ standard and $w_L\in W(\hat{\Delta}(L_0))$. Since the elements of $\textup{Den}(L_0,ww_H,\Sigma)_+$
    and $\textup{Den}(L_0,ww_H,\Sigma')_+$ can (by the above result) be represented by elements of the set $\hat{R}_+\backslash \hat{R}(L_0)$
    (which is permuted by $w_L=(w_L^{-1})^*$)  restricted to $L_0$, and since $\frac{r(w_L\gl)}{r(\gl}|_{L_0}$ is a nonzero constant, the formula's 
    \begin{equation}
    \Sigma_{W}|_L:=(w_L^{-1})^*(\frac{r(w_L\gl)}{r(\gl)}\Sigma_{W}|_{L_0})\textrm{\ and\ } \Sigma'_{W'w}|_L:=(w_L^{-1})^*(\frac{r(\gl)}{r(w_L\gl)}\Sigma'_{W'ww_L}|_{L_0})
    \end{equation}
    imply 
    \begin{equation}\label{eq:transS}
    \textup{Den}(L,w,\Sigma)_+=(w_L^{-1})^*(\textup{Den}(L_0,ww_L,\Sigma)_+)=(w_L^{-1})^*(\textup{Den}(L_0,d,\Sigma)_+)
    \end{equation}
    and 
    \begin{equation}\label{eq:transSp}
    \textup{Den}(L,w,\Sigma')_+=(w_L^{-1})^*(\textup{Den}(L_0,ww_L,\Sigma')_+)=(w_L^{-1})^*(\textup{Den}(L_0,d,\Sigma')_+)
    \end{equation}
    Since $w_L$ permutes the coroots in $\hat{R}_+\backslash \hat{R}(L_0)$, Proposition \ref{lem:incl} implies: 
    \begin{equation}
    \textup{Den}(L,w,\Sigma)_+=(w_L^{-1})^*(\textup{Den}(L_0,ww_L,\Sigma)_+)\subset D_L
    \end{equation} 
    If $w(\hat{P}_L)\subset \hat{R}_+\cup R'$ we have $ww_L(P_{L_0})\subset \hat{R}_+\cup \hat{R}'$, 
    and Lemma \ref{lem:pospole} thus implies that the minimal length representative $d$ in $W'ww_L$ satisfies 
    $d(P_{L_0})\subset \hat{R}_+$. Then we have by $W'$-invariance and Remark \ref{rem:warn} that: 
    \begin{equation}
    \textup{Den}(L,w,\Sigma')_+=(w_L^{-1})^*(\textup{Den}(L_0,ww_L,\Sigma')_+)=(w_L^{-1})^*(\textup{Den}(L_0,d,\Sigma')_+)
    \end{equation}
    and thus Proposition \ref{lem:incl} yields: 
    \begin{equation}
    \textup{Den}(L,w,\Sigma')_+=(w_L^{-1})^*(\textup{Den}(L_0,ww_L,\Sigma')_+)\subset D'_L
    \end{equation} 
    We similarly $\textup{Den}(L,w,\Sigma)_+\subset \textup{Den}(L,1,\Sigma)_+\subset D_L$ and if 
    $X-c\in \textup{Den}(L,w,D'_L)_+$ then $X+c+1\in \textup{Den}(L,1,\Sigma)_+$, by transfer to $L_0$ via $w^*_L$ and 
    Proposition \ref{lem:incl}.
    \end{proof}
    \section{Regular envelopes, the enveloping denominators}\label{s:RegEnv} 
    \begin{defn}\label{def:hull}
    Let $L$ be an essentially standard pole space, and $w\in W$. A regular envelope $H$ of $L$ is 
    a regular pole space $H$ such that $H\supset L$. We say that a regular envelope $H$ is a $w$-good regular envelope if $H$ is $w$-good in the sense of Definition \ref{defn:good}. 
    We denote by $\textup{Den}^{\hat{\fz}_L}(H,w)\subset \textup{Den}(H,w)$ the subset of denominators 
    of $H$ which are not constant on $L$.
    We define the sets of \emph{enveloping denominators} by:
    \begin{align*}\textup{Den}^\sim(L,w,\Sigma)&:= \cap_{\{H\supset L\ w-\mathrm{good\ regular\ envelope}\}} \textup{Den}_{nc}^{\hat{\fz}_L}(H,w,\Sigma)|_L\mathrm{\ and\ }\\ 
    \textup{Den}^\sim(L,w,\Sigma')&:= \cap_{\{H\supset L\ w-\mathrm{good\ regular\ envelope}\}} \textup{Den}_{nc}^{\hat{\fz}_L}(H,w,\Sigma')|_L
    \end{align*}
    We denote by $\textup{Den}^\sim(L,w,\Sigma)_+\subset \textup{Den}^\sim(L,w,\Sigma)$ the set of representatives 
    with a positive $L$-gradient of the elements belonging to $\textup{Den}^\sim(L,w,\Sigma)$ (and similarly for $\textup{Den}^\sim(L,w,\Sigma')$).
    \end{defn}
    \begin{cor}\label{cor:cont} Let $L$ be a pole space and let $(L_0,w)$ be a standard pair such that  $L=w(L_0)$. Then 
    $\textup{Den}(L_0,w,\Sigma)\subset \textup{Den}_{nc}(L_0,w,\Sigma)\subset\textup{Den}^\sim(L_0,w,\Sigma)$ 
    as well as $\textup{Den}(L_0,w,\Sigma')\subset \textup{Den}_{nc}(L_0,w,\Sigma')\subset \textup{Den}^\sim(L_0,w,\Sigma')$.
    Finally we have $\textup{Den}(L)\subset (w^{-1})^*(\textup{Den}^\sim(L_0,w,\Sigma))\cup (w^{-1})^*(\textup{Den}^\sim(L_0,w,\Sigma'))$.
    \end{cor}
    \begin{proof}
    Suppose that $L_0, H$ are pole spaces and $L_0\subset H$. By Proposition \ref{prop:herit} (and its proof) 
    it follows that $M_{L_0,x}\leq M_{L_0,x}^H$, 
    $m_{L_0}\leq m_{L_0}^H$ (hence $m_{(L_0,x,\Sigma)}\leq m^H_{L_0}$) and  $m'_{L_0,x}\leq m'^H_{L_0,x}$ (hence $m_{(L_0,x,\Sigma')}\leq m'^H_{L_0,x}$). 
    This implies the non-obvious steps in the sequences of inclusions, by taking the minimum of the right hand sides over the set of $w$-good envelopes of $L_0$ . 
    
    The final inclusion follows from $\textup{Den}(L)=\textup{Den}(L,1)=(w^{-1})^*(\textup{Den}(L_0,w))=
    (w^{-1})^*(\textup{Den}(L_0,w,\Sigma))\cup (w^{-1})^*(\textup{Den}(L_0,w,\Sigma'))$.
    \end{proof}
    \begin{cor}\label{rem:std}
    Let $L=wL_0$ be a pole space, with $L_0$ essentially standard. 
    If $H\supset L_0$ is a regular envelope we can choose a standard regular pole space $H_0$ and 
    a $w_H\in W(\hat{\Delta}(L_0))$ such that $H=w_H(H_0)$
    (see Proposition \ref{prop:essstdhull}). 
    So by Proposition \ref{prop:essregst}: If $L=w(L_0)$ with $L_0$ essentially standard then: 
    \begin{align*}\textup{Den}^{\sim}(L_0,w,\Sigma)_+&:=
    \cap_{\{H\supset L_0\  \textup{$w$-good\ regular\ envelope}\}}(w_H^{-1})^*(\textup{Den}_{nc}^{\fz_{L_0}}(H_0,w_H,\Sigma)_+)|_{L_0},\\
    \textup{Den}^{\sim}(L_0,w,\Sigma')_+&:=
    \cap_{\{H\supset L_0\  \textup{$w$-good\ regular\ envelope}\}} (w_H^{-1})^*(\textup{Den}_{nc}^{\fz_{L_0}}(H_0,ww_H,\Sigma')_+)|_{L_0}.
    \end{align*}
    \end{cor}
    The cascade $C$ (see Appendix \ref{a:Cascade}) consists of a collection of pairs $(L,\sigma)$ where $L$ is an $\Omega$-pole space and 
    $\sigma\subset L$ a segment along which the base points of the contours of integration are moved. For each $W'$-orbit of pole spaces 
    which appears in $C$ a standard pair $(L_0,d)$ is chosen with $d\in W$ minimal in $W'd$, and for each pair $(L,\sigma)$ in $C$ with 
    $L\in W'd(L_0)$ a Weyl group element $w=ud\in W'd$ is chosen such that $L=w(L_0)$. For each standard pair $(L_0,d)$ 
    and $(L,\sigma)\in C$ such that $L\in W'd(L_0)$, the union of segments $\sigma_0:=w^{-1}(\sigma)\subset L_0$
    forms a tree $C_{(L_0,d)}$ in $L_0$, which is a star with center $c_{L_0}$ in case there exists a residual pole space 
    $M$ such that $L_0^{\textup{temp}}\subset M^{\textup{temp}}$ (in this case $c_{L_0}=c_M$). The set of representing standard pairs $(L_0,d)$ with 
    a nonempty tree $C_{(L_0,d)}$ (and the further data that are attached to the segments in $C_{(L_0,d)}$) is called the set of 
    standard data $\textup{Std}(C)$ of the cascade $C$.
    
    The cascade $C$ and the set of standard data $\textup{Std}(C)$ satisfying the properties listed in Definition \ref{defn:casc} have been 
    constructed recursively using computer algebra (see Appendix \ref{a:Cascade}). For the residual members $(L_0,d)$ in $\textup{Std}(C)$ 
    we have also computed by computer algebra the sets of enveloping denominators (see Appendix \ref{a:EnvDen}). 
    
    Theorem \ref{thm:tauadm} below gives an important transformation rule for the sets of enveloping denominators for residual 
    standard pairs representing elements in the cascade $C$ (but not necessarily in $\textup{Std}(C)$). This result depends on 
    an invariance property for the Cartan involution form members of $\textup{Std}(C)$ which is based on computer verifications 
    for exceptional cases, and which we will formulate separately. Let $(L_0,x)\in \textup{Std}(C)$ be a standard residual pair.
    \begin{defn}
    Let $\tau_0=\tau_{L_0}: L_0\to L_0$ be as in Proposition \ref{prop:tau_L}. We define 
    \begin{equation*}
    \textup{Den}^{\sim}(L_0,x,\Sigma_\tau)_+:=
    \bigcap_{\{H\supset L_0 \ \textup{$x$-good\ regular\ envelope}\}}(w_H^{-1})^*(\textup{Den}_{nc}^{\fz_{L_0}}(H_0,w_H,\Sigma_\tau)_+)|_{L_0}
    \end{equation*}
    \index{$\textup{Den}(L)$!$\textup{Den}^{\sim}(L_0,x,\Sigma_\tau)_+$}
    Here $H=w_H(H_0)$ where $H\supset L_0$ is a regular pole space containing $L$.
    We define an affine linear map $(1-\tau_0):\textup{Aff}(L)\to\textup{Aff}(L)$ by 
    $(1-\tau_0)(X)=1-\tau_0(X)$\index{$(1-\tau_0)$}.
    \end{defn}
    The following Lemma was verified by computer for the exceptional cases:
    \begin{lem}\label{lem:tauadmpart_a} 
    Let $\hat{R}$ be an irreducible root system, with base $\hat{\Delta}$. 
    There exists proper maximal standard Levi $\hat{R}'$ and a cascade $C$ for $(\hat{R},\hat{R}')$ such that 
    for every residual pole space $(L_0,x)\in\textup{Std}(C)$ we have: 
    \begin{equation}\label{eq:tau}
    \textup{Den}^\sim(L_0,x,\Sigma_\tau)_+=(1-\tau_0)(\textup{Den}^\sim(L_0,x,\Sigma)_+)
    \end{equation}
    \end{lem}
    \begin{proof}
    For regular pole space pairs $(H_0,x)$ in $\textup{Std}(C)$ follows directly from Proposition \ref{prop:essregst} and \ref{lem:incl}. 
    For classical groups, the cascades discussed in Appendix \ref{a:class} have the property that for all standard pairs $(L_0,x)$ in $\textup{Std}(C)$ 
    we have $d(\hat{P}_{L_0})\subset \hat{R}_+$(with $x=ud$ with $u\in W'$ and $d$ minimal in $W'd$). Hence all regular envlopes of $L_0$ are $x$-good 
    (See Appendix \ref{a:class}), and thus $\tau_0$ preserves the set of $x$-good enveloping regular pole spaces of $L_0$. 
    For the remaining non-regular standard residual pairs in the standard data $\textup{Std}(C)$ of the cascades constructed for the exceptional cases 
    $E_n$ and $F_4$, part $(a)$ was shown by computer verifications (see appendices \ref{a:Cascade}, \ref{a:EnvDen}).
    \end{proof}
    \begin{thm}\label{thm:tauadm} 
    Let $C$ be a cascade satisfying (\ref{eq:tau}). 
    \begin{enumerate}
    \item[(a)] Let $L$ be a pole space appearing $C$ represented by $(L_0,x)\in\textup{SP}(C)$,  
    and let $(L_1,y)$ be any standard pole pair representing $L$. Let $w=x^{-1}y\in W$, so that $w(L_1)=L_0$. 
    For  $X=\nabla^1_{{\hat{\ga}}}+c\in\textup{Aff}(L_1)$ with positive gradient $\nabla^1_{{\hat{\ga}}}$, 
    we define $|w|(X):=|w|(\nabla^1_{{\hat{\ga}}})+c\in\textup{Aff}(L_0)$\index{$|w|$} where $|w|(\nabla^1_{{\hat{\ga}}})=w(\nabla^1_{{\hat{\ga}}})
    =\nabla^0_{w({\hat{\ga}})}$ if $w({\hat{\ga}})$ is positive, and $|w|(\nabla^1_{{\hat{\ga}}})=-\nabla^0_{w({\hat{\ga}})}$
    otherwise.Then 
    \begin{equation*}
    \textup{Den}^{\sim}(L_0,x,\Sigma)_+=|w|\textup{Den}^{\sim}(L_1,y,\Sigma)_+
    \end{equation*}
     \item[(b)] Let $L$ and $(L_1,y)$ be as in $(a)$. Then 
    \begin{equation*}
    \textup{Den}^\sim(L_1,y,\Sigma_\tau)_+=(1-\tau_1)(\textup{Den}^\sim(L_1,y,\Sigma)_+)
    \end{equation*}
    \item[(c)] Let $L$ be a residual pole space appearing in the cascade, $(L_1,x)$ a standard pair 
    such that $L=x(L_1)$. Then 
    $\textup{Den}^{\sim}(L_1,x,\Sigma')_+\subset (1-\tau_1)(\textup{Den}^{\sim}(L_1,x,\Sigma)_+)$ 
    \end{enumerate}
    \end{thm}
    \begin{proof}
    For $(a)$: 
    First note that $w=x^{-1}y$ maps the set of $y$-good regular envelopes of $L_1$ to the set of $x$-good regular envelopes of $L_0$. 
    Using Lemma  \ref{lem:trans} we may and will assume that $w(\Delta_{L_1})=\Delta_{L_0}$ (by adapting $x$ on the right by a suitable 
    element of the pointwise fixator of $L_0$); 
    we fix an inner automorphism $\gk_w=\textup{Ad}(g)\in \textup{Inn}(\hat{\fg})$ normalizing $\hat{\ft}$ and realizing $w$ on $\hat{\ft}$.
    Let $H^1$ be an $y$-good regular envelope of $L_1$, and let $\hat{\fq}_{H^1}\subset \hat{\fg}^{H^1}\subset \hat{\fg}^{L_1}$ be the 
    reductive subalgebra associated to $H^1$ as in Subsection \ref{sub:reg}. Choose a standard regular pole space $H_1$ 
    such that $H^1=w_{H^1}(H_1)$ with $w_{H^1}\in W(\Delta_{L_1})$
    (hence $w_{H^1}$ permutes ${\hat{R}}_+\backslash {\hat{R}}_{L_1}$). Let $\hat{\fp}^{H^1}=\gk_{w_{H^1}}(\hat{\fp}^{H_1})$ where 
    $\gk_{w_{H^1}}\in \textup{Inn}(\hat{G}^{L_1})$ realizes $w_{H^1}$. Let $\hat{\fp}^{H_1}$ be the standard parabolic subalgebra 
    as in Proposition \ref{lem:incl}, and put $\hat{\fp}^{H^1}=\gk_{w_{H^1}}(\hat{\fp}^{H_1})$ where $\gk_{w_{H^1}}\in \textup{Inn}(G^{L_1})$
    realizes $w_{H^1}$. 
    By Proposition \ref{prop:IS} we have 
    \begin{equation}
    \textup{Den}_{nc}^{\fz_{L_1}}(H^1,1,\Sigma)_+=(w_{H^1}^{-1})^*(\textup{Den}_{nc}^{\fz_{L_1}}(H_1,w_{H^1},\Sigma)_+)
    \end{equation} 
    Let $\hat{\ga}$ be mapped to $\hat{\gb}=w_{H^1}(\hat{\ga})$. Hence $m_{H^1}((\hat{\gb}+1)|_{H^1})$ equals 
    $m_{H_1}(({\hat{\ga}}+1)|_{H_1})$, which is equal to (by Proposition \ref{lem:incl}) the multiplicity 
    $\textup{mult}_{\hat{\fu}^{H_1}}^{\hat{\fq}^{H_1}}(V^{H_1}({\hat{\ga}}|_{H_1}))$. In view of the above definition of $\hat{\fp}^{H^1}$ 
    this equals $\textup{mult}_{\hat{\fu}^{H^1}}^{\hat{\fq}^{H^1}}(V^{H^1}({\hat{\gb}}|_{H^1}))$. 
    Hence from the definition of $\textup{Den}^\sim(L_1,y,\Sigma)$ we see that for $\nu\in \textup{Aff}(L_1)^*$ the 
    multiplicity of $\nu+1$ in $\textup{Den}^\sim(L_1,y,\Sigma)$ is:
    \begin{align*}
    m&^\sim_{(L_1,y,\Sigma)}(\nu+1))\\
    &=\textup{min}\{\sum_{\mu\in\textup{Aff}(H^1):\mu|_{L_1}=\nu}
    \textup{mult}_{\hat{\fu}^{H^1}}^{\hat{\fq}^{H^1}}(V^{H^1}(\mu))\mid H^1\supset L_1 \textrm{\ $y$-good regular envelope}\}\\
    &=\textup{min}\{\sum_{\mu\in\textup{Aff}(H^1):\mu|_{L_1}=\nu}
    \textup{mult}_{\hat{\fg}/\hat{\fg}^{H^1}}^{\hat{\fq}^{H^1}}(V^{H^1}(\mu))\mid H^1\supset L_1 \textrm{\ $y$-good regular envelope}\}\\
    &=\textup{min}\{\sum_{\mu\in\textup{Aff}(H^1):\mu|_{L_1}=\nu}
    \textup{mult}_{\hat{\fg}/\hat{\fg}^{L_1}}^{\hat{\fq}^{H^1}}(V^{H^1}(\mu))\mid H^1\supset L_1 \textrm{\ $y$-good regular envelope}\}
    \end{align*}
    The second equality follows by that fact that $\hat{\fu}^{H^1}$ and $\overline{\hat{\fu}}^{H^1}$ do not share 
    modules of the form $V^{H^1}(\mu)$ in common if $\mu|_{L_1}$ is non-constant. Indeed, we know  
    that the highest weights $\mu$ of the components in $\hat{\fu}^{H^1}$ which are nonconstant on $L_1$ 
    are of the form $\hat{\ga}|_{H^1}$ with $\hat{\ga}\in {\hat{R}}_+\backslash {\hat{R}}_{L_1}$, and those of 
    $\overline{\hat{\fu}}^{H^1}$ are of the form $\hat{\gb}|_{H^1}$ with $\hat{\gb}\in {\hat{R}}_-\backslash {\hat{R}}_{L_1}$.
    These sets are disjoint (e.g. the first collection is positive on $\delta-\delta_{L_1}\in\hat{\fz}_{H^1}$, 
    while the second is negative on that vector). the second equality follows from the fact that we are only 
    considering highest weights which are not constant on $L_1$.

    Now $\gk_w$ defines a bijection $H^1\to H^0=\gk_w(H^1)$ from the set of $y$-good regular envelopes of $L_1$ to the set 
    of $x$-regular envelopes of $L_0$. Then $\gk_w(\hat{\fq}^{H^1})=\hat{\fq}^{H^0}$ and 
    $\gk_w:\hat{\fg}/\hat{\fg}^{L_1}\to \hat{\fg}/\hat{\fg}^{L_0}$ intertwines the action of $\hat{\fq}^{H^1}$ on $\hat{\fg}/\hat{\fg}^{L_1}$ and 
    $\hat{\fq}^{H^0}$ on  $\hat{\fg}/\hat{\fg}^{L_0}$. Hence 
    \begin{align*}
    m&^\sim_{(L_1,y,\Sigma)}(\nu+1))\\
    &=\textup{min}\{\sum_{\mu\in\textup{Aff}(H^1):\mu|_{L_1}=\nu}
    \textup{mult}_{\hat{\fg}/\hat{\fg}^{L_1}}^{\hat{\fq}^{H^1}}(V^{H^1}(\mu))\mid H^1\supset L_1 \textrm{\ $y$-good regular envelope}\}\\
    &=\textup{min}\{\sum_{\mu\in\textup{Aff}(H^0):\mu|_{L_0}=w\nu}
    \textup{mult}_{\hat{\fg}/\hat{\fg}^{L_0}}^{\hat{\fq}^{H^0}}(V^{H^0}(\mu))\mid H^0\supset L_0 \textrm{\ $x$-good regular envelope}\}.
    \end{align*}
    By the above and Proposition \ref{lem:incl} the support of $\mu\to \textup{mult}_{\hat{\fg}/\hat{\fg}^{L_1}}^{\hat{\fq}^{H^1}}(V^{H^1}(\mu))$  
    consists of weights $\mu\in\textup{Aff}^\sim(H^1)$ 
    of the form $\mu={\hat{\gb}}|_{H^1}$ for some positive coroot ${\hat{\gb}}$ which is not constant on $L_1$. 
    Hence the support of $m^\sim_{(L_1,y,\Sigma)}$ cosists of $\nu+1\in\textup{Aff}^\sim(L_1)$ with $\nu$ of the form 
    $\hat{\gb}|_{L_1}$ for some $\hat{\gb}\in {\hat{R}}_+\backslash {\hat{R}}_{L_1}$. 
    In the case that $w\nu=w{\hat{\gb}}$ is a positive coroot, we see immediately from the above formula that 
    $m^\sim_{(L_0,x,\Sigma)}(w{\hat{\gb}}+1)|_{L_0})=m^\sim_{(L_1,y,\Sigma)}({\hat{\gb}}+1)|_{L_1})$. 
    Otherwise if $w{\hat{\gb}}$ is a negative coroot, we see that the outcome equals (using Proposition 
    \ref{lem:incl}) $m^\sim_{(L_0,x,\Sigma_\tau)}((w{\hat{\gb}}+1)|_{L_0})$. 
    Using Proposition \ref{prop:tau_L} for $(L_0,x)$, we see that this equals 
    $m^\sim_{(L_0,x,\Sigma)}((\tau_0(w{\hat{\gb}})+1)|_{L_0})$. This proves $(a)$.
    
    By Proposition \ref{prop:tau_L} it follows easily that $(b)$ is equivalent to the statement that 
    $\textup{Den}^\sim(L_1,y,\Sigma)=\textup{Den}^\sim_\tau(L_1,y,\Sigma)$
    where 
    \begin{equation}
    \textup{Den}^\sim_\tau(L_1,y,\Sigma)_+=\bigcap_{\{\tau_1(H) \supset L_1  
    \textup{\ $y$-good\ regular\ envelope}\}}(w_H^{-1})^*(\textup{Den}_{nc}^{\fz_{L_1}}(H_1,w_H,\Sigma)_+)|{L_1}
    \end{equation}
    Following the steps of the proof of $(a)$ we see that this is true for $(L_1,y,\Sigma)$ if and only if 
    it is true for $(L_0,x,\Sigma)$. For $(L_0,x,\Sigma)$ it follows by (\ref{eq:tau}), whence the result. 
    
    For Part $(c)$ note that $\textup{Den}^\sim(L_1,y,\Sigma')\subset \textup{Den}^\sim(L_1,y,\Sigma_\tau)$ by 
    Proposition \ref{lem:incl}. Now use $(b)$.
    \end{proof}
    \begin{thm}\label{thm:main} Let $\hat{R}$ be an irreducible root system, with base $\hat{\Delta}$. 
    There exists proper maximal standard Levi $\hat{R}'$ and a cascade $C$ for $(\hat{R},\hat{R}')$ such that 
    for every residual pole space $L$ in $C$, and every standard pair $(L_0,w)$ with $L=w(L_0)$, we have  
    $\textup{Den}^{\sim}(L_0,w,\Sigma)_+\subset D_{L_0}$ and 
     $\textup{Den}^{\sim}(L_0,w,\Sigma')_+\subset (1-\tau_1)(\textup{Den}^{\sim}(L_0,w,\Sigma)_+)\subset D'_{L_0}$.  
    \end{thm}
    \begin{proof}
    By Theorem \ref{thm:tauadm} it is enough to find $C$ satisfying (\ref{eq:tau}) and 
    such that for all residual 
    pole pairs $(L_0,w)\in \textup{Std}(C)$ we have $\textup{Den}^{\sim}(L_0,w,\Sigma)_+\subset D_{L_0}$. 
    For classical root systems see Appendix \ref{a:class}. For the exceptional cases 
    this has been proved by computer calculations of the denominator sets of the residual standard pairs $(L_0,w)\in \textup{Std}(C)$ 
    of an explicit cascade $C$ constructed by computer (see Appendix 
    \ref{a:Cascade} and Appendix \ref{a:EnvDen}).  
    \end{proof}
    \section{Admissible sets for order $0$ pole spaces}\label{s:Adm}
    \begin{defn}\label{def:Adm}
    For each $\Omega$-pole space $L$ of order $0$ we define $\textup{PoleFree}(L)\subset L$\index{$\textup{PoleFree}(L)\subset L$ for $L$ of $\Omega$-order $0$}
    as the complement of the set $\cup_{Y\in \textup{Den}(L)}Y^{-1}(0,1)$. Given a pair $(L_0,w)$ with 
    $L_0$ a standard pole space and $w\in W$ such that $L=w(L_0)$, we define:
    \begin{align}\label{eq:ineqw}
    \textup{Adm}^{(L_0,w)}(L):=
    \{p\in L\mid 
    &\ Y(w^{-1}p)\geq 1,\,\forall Y\in \textup{Den}(L_0,w,\Sigma)_+ \mathrm{\ and\ } \\ 
    \nonumber&\ Z(w^{-1}p)\leq 0,\,\forall Z\in \textup{Den}(L_0,w,\Sigma')_+\ \ \ \ \ \ \ \ \}
    \end{align}
    (an intersection of halfspaces of $L$).
    Finally we define the closed convex set $\textup{Adm}(L)$\index{$\textup{Adm}(L)\subset L$ for $L$ of $\Omega$-order $0$} by: 
    \begin{equation}
    \textup{Adm}(L):=\bigcap_{(L_0,w) \mathrm{\ standard\ pair\ such\ that\ }L=w(L_0)}\textup{Adm}^{(L_0,w)}(L)
    \end{equation}
    \end{defn}
    \begin{lem}\label{lem:indepw}
    If $\textup{Adm}(L)$ is non-empty then 
    it is a component of $\textup{PoleFree}(L)\subset L$, and $\textup{Adm}(L)=\textup{Adm}^{(L_0,w)}(L)$ for 
    any standard pair $(L_0,w)$ such that $L=w(L_0)$. Let $F$ be a number field, and $\rho(s)=\Lambda_F(s)$ its 
    regularized Dedekind zeta function, and let $\phi,\psi$ be holomorphic functions on $B\subset V_\bbC$. 
    Then $\Sigma_W(\phi)\Sigma'_{W'}(\psi)|_{L_\bbC}$ is holomorphic at all points of $B\cap (\textup{Adm}(L)+iV^L)$.
    \end{lem}
    \begin{proof}
    By Definition \ref{defn:densets} and by Lemma \ref{lem:trans} we have, for every standard pair $(L_0,w)$ such that $L=w(L_0)$: 
    \begin{equation}\label{eq:parti}
    \textup{Den}(L)= (w^{-1})^*(\textup{Den}(L_0,w,\Sigma))\cup   (w^{-1})^*(\textup{Den}(L_0,w,\Sigma'))
    \end{equation}
    Hence for each equivalence class $\{\nu,1-\nu\}\in \textup{Den}(L)$ the inequalities in (\ref{eq:ineqw}) contain either 
    the inequality $\nu\geq 1$ or the inequality $\nu\leq 0$, and both if $\{\nu,1-\nu\}\in (w^{-1})^*(\textup{Den}(L_0,w,\Sigma))\cap  
    (w^{-1})^*(\textup{Den}(L_0,w,\Sigma'))$. Therefore, if $\textup{Adm}^{(L_0,w)}(L)$ is not empty then (\ref{eq:parti}) is a partition 
    of $\textup{Den}(L)$, and $\textup{Adm}^{(L_0,w)}(L)$ is the intersection of a collection of half spaces where each half space is 
    a component of a set of the form $\bbR\backslash \nu^{-1}(0,1)\subset L$ with $\{\nu,1-\nu\}\in \textup{Den}(L)$, and where 
    one such half space is selected for each class of $\textup{Den}(L)$. This shows that if $\textup{Adm}^{(L_0,w)}(L)$ is not empty
    then $\textup{Adm}^{(L_0,w)}(L)$ is a component of $\textup{PoleFree}(L)$. But then it is clear that $\textup{Adm}(L)$ is non-empty 
    if and only if  $\textup{Adm}^{(L_0,w)}(L)$ is non-empty and independent of the choice of the standard pair $(L_0,w)$. It follows that in that case 
    $\textup{Adm}(L)=\textup{Adm}^{(L_0,w)}(L)$ for any standard pair $(L_0,w)$ such that $L=w(L_0)$. The final claim is obvious.
     \end{proof}
     \begin{lem}\label{lem:herit}
    Let $L$ be an $\Omega$-pole space of order $0$, and $u\in W'$. Then 
    $u(\textup{Adm}(L))=\textup{Adm}(uL)$. If $M\supset L$ is also an order $0$ $\Omega$-pole space, then  
    $\textup{Adm}(M)\cap L\subset \textup{Adm}(L)$
    (heritability of admissibility).  
    \end{lem}
    \begin{proof}
    The first assertion follows easily using Definition \ref{def:Adm}, Lemma \ref{lem:indepw}, together with  
    Lemma \ref{lem:trans}, by noting that $(L_0,w)$ is a standard pair for $L$ if and only if $(L_0,uw)$ is a standard pair for $uL$.
    To prove the heritability we may assume that $\textup{Adm}(M)\cap L\not=\emptyset$. Assume that $p\in \textup{Adm}(M)\cap L$. 
    Let $(L_0,w)$ be a standard pair for $L$, and put $M^0=(w^{-1})(M)$. Then we are in the situation of Proposition \ref{prop:IS}, 
    and we will use the notations in that Proposition. 
    Since $p\in \textup{Adm}(M)\subset \textup{PoleFree}(M)$ we can use the induction hypothesis and the standard pair 
    $(M_0,ww_{M^0})$ for $M$ to see that $p\in\textup{Adm}(M)^{(M_0,ww_{M^0})}(M)$. 
    By Proposition \ref{prop:IS} and Proposition \ref{prop:herit} we see that this implies that also $p\in \textup{Adm}^{(L_0,w)}(L)$. 
    Lemma \ref{lem:indepw} shows that this is true 
    for all choices of standard pairs $(L_0,w)$ and hence $p\in \textup{Adm}(L)$.
    \end{proof} 
    \section{Reduction to contours near residual centers (split case)}\label{s:MainResultRes}
    Iterated residue integrals in a pole space $L$ in $C$ are of the form 
    \begin{equation}\label{eq:itres}
    \int_{p_{L}+iV^L}\textup{Res}_{L,\cF}(\Sigma_W(\phi)\Sigma'_{W'}(\psi)\gO)
    \end{equation}
    where $\cF$ is the flag of pole spaces $L=L_0\subset L_1\subset L_2\dots$ in $C$ ending in $L$ corresponding to 
    a path of segments in $C$ defining this iterated residue 
    (see Subsection \ref{ss:intro} (or \cite[V.1.3]{MW2})). 
    If $\textup{Ord}_L(\Omega)=0$ the integrand $\textup{Res}_{L,\cF}(\Sigma_W(\phi)\Sigma'_{W'}(\psi)\gO)$ is of the form: 
    \begin{equation}
    \textup{Res}_{L,\cF}(\Sigma_W(\phi)\Sigma'_{W'}(\psi)\gO)=\fc_{L,\cF}(\Sigma_W(\phi)\Sigma'_{W'}(\psi))|_{L_\bbC}\Omega^L
    \end{equation}
    with $\fc_{L,\cF}$ a constant depending on $\cF$ (see (\ref{eq:resdatCst})). More generally, the integrand
    $\textup{Res}_{L,\cF}(\Sigma_W(\phi)\Sigma'_{W'}(\psi)\gO)$ has a unique expansion of the form 
    \begin{equation}\label{eq:resdat}
    \textup{Res}_{L,\cF}(\Sigma_W(\phi)\Sigma'_{W'}(\psi)\gO):=\sum_i (D_{L,\cF,i}(\Sigma_W(\phi)\Sigma'_{W'}(\psi))|_{L_\bbC})\gO^L_{\cF,i}
    \end{equation}
    where $\gO_{L,\cF,i}$ are rational $(\textup{dim}(L),0)$-forms on $L_\bbC$ and where 
    $D_{L,\cF,i}\in \textup{Sym}(V_L)$ is a constant coefficient differential operator of order $\textup{Ord}_L(\Omega)$, 
    all depending on $\cF$.  This makes the study of poles substantially harder in the case $\textup{Ord}_L(\Omega)>0$.
    \begin{thm}\label{thm:casc_ord_int}
    Let $\hat{R}$ be an irreducible based root system which is not of type $\textup{E}_8$. 
    Then there exists a proper maximal standard Levi subsystem $\hat{R}'$, and a cascade $C$ as in Theorem \ref{thm:main}, 
    such that all pole spaces $L$ in $C$ which are intersected in a point other than $c_L$ by a segment in $C$ are of order $0$. 
    
    If $\hat{R}$ has type $\textup{E}_8$, and we take $\hat{R}'\subset \hat{R}$ equal to the standard Levi subsystem of type $\textup{E}_7$, 
    then, we have constructed a cascade $C$ (by computer) which satisfies Theorem \ref{thm:main}, but having three $W'$-orbits of pole spaces $L$ 
    with $\textup{Ord}_L(\Omega)>0$ which are not intersected at $c_L$. 
    These three cases are all residual and of $\Omega$-order $1$. Moreover, if $L$ is a pole space in one of these $W'$-orbits 
    and if $(\sigma_M,M)\in C$ such that $L\subset M$ is a pole hyperplane with $\sigma_M\cap L=\{p_{(L,\sigma_M)}\}$,  
    then $\textup{Ord}_M(\Omega)=0$ and $c_L\in \textup{Adm}(M)$ \emph{except} for one 
    line $L_{sp}=s_8s_7(L_{sp,0})$\index{$L_{sp,0}=\textup{E}_7(a4)$!$L_{sp}$ the special pole line of $E_8$}
    with $L_{sp,0}=\textup{E}_7(a4)$\index{$L_{sp,0}=\textup{E}_7(a4)$} and $M_{sp}=\textup{E}_6(a3)\supset L_{sp}$. If $x$ is 
    the coordinate on $L_{sp,0}$ centered at $c_{L_{sp,0}}$ and dual to $\go_8\in V^{L_{sp,0}}$ then $x(s_7s_8(p_{(L_{sp},\sigma_{M_{sp}})}))>3/2$ 
    while $\textup{Den}^{\hat{\fz}_{L_{sp,0}}}((s_7s_8)(M),s_7s_8,\Sigma)|_{L_{sp,0}}\backslash D_{L_{sp,0}}= \{(x-1/2)\}$ and 
    $\textup{Den}^{\hat{\fz}_{L_{sp,0}}}((s_7s_8)(M),s_7s_8,\Sigma')|_{L_{sp,0}}\subset D'_{L_{sp,0}}$. 
    In particular, for all pairs $L\subset M$ with $\textup{Ord}_L(\Omega)>0$ and where $L$ is not intersected at its center $c_L$, 
    we have $\textup{Ord}_M(\gO)=0$ and $L^{\textup{temp}}\subset \textup{PoleFree}(M)+iV^M$. 
    \end{thm}
    \begin{proof}
    For classical cases all spaces in $C$ have $\Omega$-order $0$, see Appendix \ref{a:class}. 
    For exceptional cases the above was verified by checking the classification of the pole spaces in the cascades $C$  
    which we have constructed by computer algebra. 
    \end{proof}
    \begin{thm}
    Suppose that the cascade $C$ satisfies the properties in Theorem \ref{thm:main} and in 
    Theorem \ref{thm:casc_ord_int}.
    Let $L$ be a pole space in $C$ of order $0$. Then all initial points $p \in L$ 
    (such $p$ is either of the form $p_{L,\infty}$ with $p_{L,\infty}$ the initial point ``at infinity'' at the start of a branch 
    of $C$ (so in this case $L$ is residual and induced from $G'$), or $p=p_{(L,\sigma_M)}=\sigma_M\cap L$
    where $(\sigma_M,M)\in C$ is a parent of $L$) belong to $\textup{Adm}(L)$. 
    When $L^{\textup{temp}}$ is a subspace of the tempered form $M^{\textup{temp}}$ of a residual pole space $M\in C$ then 
    $c_L\in \textup{Adm}(L)$. In particular, in the latter cases we have $w(\textup{Adm}(L))=\textup{Adm}(wL)$
    for all $w\in W$. 
    \end{thm}
    \begin{proof}
    The proof is by induction on the codimension $k$ of the order $0$ pole spaces $M$ such that 
    $(\sigma_M,M)\in C$ for some segment $\sigma_M\subset M$. We will use throughout in the proof the obvious fact that 
    if $L$ is a pole space in $C$ with $\textup{Ord}_L(\Omega)=0$ and  $(\sigma_M,M)\in C$ such that $L\subset M$ is a pole hyperplane 
    in $M$, then also $\textup{Ord}_M(\Omega)=0$. 
    
    Assume by induction that for all such pairs $(\sigma_M,M)$ with $M$ of codimension $k-1$, we have 
    $\sigma_M\subset \textup{Adm}(M)$. That is, for all $M$ with $\textup{Ord}_M(\gO)=0$ and all standard pairs 
    $(M_0,w)$ such that $M=w(M_0)$, all points $p\in \sigma_M$ satisfy: 
    \begin{align}\label{eq:ineq}
    &Y(w^{-1}p)\geq 1,\ \forall Y\in \textup{Den}(M_0,w,\Sigma)_+ \mathrm{\ and\ }\\ 
    \nonumber&Z(w^{-1}p)\leq 0,\ \forall Z\in \textup{Den}(M_0,w,\Sigma')_+
    \end{align}
    The induction basis is $M=V=\hat{\ft}^*$, with $p=p_{V,\infty}$ and $\sigma_V=[p_{V,\infty},0]$. 
    Let $w=ud\in W$, with $d$ minimal in $W'd$. 
    It follows easily from the definitions that $\textup{Den}(V,w)_+
    =(1+\hat{R}_+\backslash \hat{R}(d))\cup (\hat{R}(d) \cup d^{-1}(\hat{R}'_+))$. 
    Indeed, the factor $\Sigma'_{W'w}$ can be written as
    \[
    \Sigma'_{W'w}=\frac{r(d\gl)}{r(\gl)}d^*(\Sigma'_{W'})
    \]
    Since the denominators of $d^*(\Sigma'_{W'})$ are clearly in $d^{-1}(\hat{R}'_+)\subset \hat{R}_+$ none of the 
    numerators $1+\hat{R}(d)$ of $\frac{r(d\gl)}{r(\gl)}$ cancels, and so $1+\hat{R}(d)$ is exactly the set of common numerators 
    of $\Sigma'_{W'w}$. This reduces, according to the definitions, the set $\textup{Den}(V,w,\Sigma)_+$ to $(1+\hat{R}_+\backslash \hat{R}(d))$
    (by cancelling these numerators from the set $1+\hat{R}_+$) and shows that $\textup{Den}(V,w,\Sigma')_+=
    \hat{R}(d)\cup d^{-1}(\hat{R}'_+)$. This defines a partitioning of $\textup{Den}(V,w)_+$. Therefore, 
    since $(w^{-1})^*(\textup{Den}(V,w))=\textup{Den}(V)=(1+\hat{R}_+)\sqcup \hat{R}'_+$ and since the inequalities (\ref{eq:ineq}) 
    at least have $c_V=0$ as solution,  (\ref{eq:ineq}) describes the unique component of $\textup{PoleFree}(V)$ which containes $0$.
    In particular the solution set of (\ref{eq:ineq}) is independent of $w$. 
    We see that the solution set is $\textup{Adm}(V)=\bbR_{\geq0}\fw'$, and thus that $\sigma_V\subset \textup{Adm}(V)$.
    
    Now let us look at the induction step. Let $L$ be a pole space in $C$ with $\textup{Ord}_L(\Omega)=0$ and codimension $k$. 
    There are two cases to consider. First assume that $L$ is induced from a standard residual pole space of $\hat{G'}$ and 
    $\sigma_L=[p_{L,\infty},c_L]$. This is similar to the proof of the induction base. 
    By Theorem \ref{thm:main} and Corollary \ref{cor:cont} we see that for any standard pair $(L_0,w)$ such that $L=w(L_0)$
    we have that $(w^{-1})^*( \textup{Den}(L_0,w,\Sigma)_+)\subset (w^{-1})^*(\textup{Den}(L_0,w)\cap \textup{Den}(L_0,w,\Sigma)_+^{\sim})$ consists of 
    $L$-gradients of coroots plus a constant 
    $\geq 1$, while $(w^{-1})^*(\textup{Den}(L_0,w,\Sigma')_+)\subset (w^{-1})^*(\textup{Den}(L_0,w)\cap \textup{Den}(L_0,w,\Sigma')_+^{\sim})$ 
    consists of $L$-gradients of coroots minus a constant $\geq 0$. 
    Hence $c_L$ satisfies (\ref{eq:ineq}) for $M=L$ in this case. Theorem \ref{thm:main}, Definition \ref{def:Adm} and Lemma \ref{lem:indepw} show that 
    $c_L\in\textup{Adm}(L)$, and that $\textup{Adm}(L)=\textup{Adm}^{(L_0,w)}(L)$ for any standard pair $(L_0,w)$ such that $L=w(L_0)$. 
    Since $L$ is itself standard and in fact induced from a standard pole space of $G'$ it follows that in particular $\textup{Adm}(L)=\textup{Adm}^{(L,1)}(L)$.
    But $\textup{Den}(L,1,\Sigma')^\sim$ is contained in $\hat{R}'_+|_L$, while  $\textup{Den}(L,1,\Sigma)^\sim$ is contained in $\hat{R}_+|_L+1$. 
    Since $p_{L,\infty}=c_L+t\go_0$ with $t\gg0$ and $c_L\in \textup{Adm}(L)$ it follows 
    that also $p_{L,\infty}\in \textup{Adm}(L)$.  By convexity we have $\sigma_L\subset  \textup{Adm}(L)$.
    
    The second case to consider is when $L$ has parents in the cascade. Suppose that 
    $(\sigma_M,M)\in C$ is a parent of $L$ of codimension $k-1$, and put $p_M=L\cap\sigma_M$. Then  $\textup{Ord}_M(\Omega)=0$
    and the induction hypotheses 
    and Lemma \ref{lem:herit} show that  $p_M\in \textup{Adm}(L)$. 
    Next suppose that $L\subset M$ where $M$ is residual and in $C$, and $L$ is such that $L^{\textup{temp}}=M^{\textup{temp}}$. 
    Then obviously $c_L=c_M$. Now observe that Theorem \ref{thm:main} 
    (as above) shows that $c_M\in \textup{Adm}(M)$, and then Lemma \ref{lem:herit} implies that 
    $c_L=c_M\in \textup{Adm}(M)\cap L\subset  \textup{Adm}(L)$. 
    The cascade $C$ is defined such that all the segments $\sigma_L\subset L$ have as end points either points of the form 
    $u^{-1}(p)$ with $u\in W'$, where $p\in uL$ is the intersection point of $uL$ and a segment of a parent of $uL$ in the cascade $C$ or, if 
    $L$ is such that $L^{\textup{temp}}\subset M^{\textup{temp}}$ for some residual $M$ in the cascade $C$, possibly $p=c_L$.
    By the above and Lemma \ref{lem:herit} we see that such end points always belongs to $\textup{Adm}(L)$, and 
    hence by convexity, $\sigma_L\subset \textup{Adm}(L)$.
    
    Note that when $L$ is such that $L^{\textup{temp}}\subset M^{\textup{temp}}$ for some residual $M$ in the cascade $C$ we 
    can describe $\textup{Adm}(L)$ as the unique component of $\textup{PoleFree}(L)$ such that $c_L\in \textup{Adm}(L)$.
    It follows that if also $w(L)$ is in the cascade, for some $w\in W$ and $L^{\textup{temp}}\subset M^{\textup{temp}}$ for some 
    residual pole space, then $w(\textup{Adm}(L))=\textup{Adm}(wL)$. 
    \end{proof}
    \begin{cor}\label{cor:Cadmis}
    Let $F$ be a number field, and $\rho(s)=\Lambda_F(s)$ its regularized Dedekind zeta function. Let $\fR\gg0$, and let 
    $\phi,\psi$ be holomorphic functions on $B_{\fR}=\{v=x+iy\mid x,y\in V,\ \Vert x\Vert\leq \fR\}  
    \subset V_\bbC$. 
    Let $C$ be a cascade as in Theorem \ref{thm:casc_ord_int}. Let $(\sigma_j,L_j)\in C$ for $j=0,1,2,\ldots$ be a connected 
    path of segments in $C$ with $L_0=L$ and $\textup{dim}(L_j)=\textup{dim}(L_0)+j$. 
    We denote by $\cF=\{\dots \supset L_2\supset L_1\supset L_0=L\}$ the corresponding 
    flag of pole spaces in $C$. Put $\cF_j=\{\dots \supset L_{j+2}\supset L_{j+1}\supset L_j\}$ for the partial 
    flag in $\cF$ ending at $L_j$. Then for all $j=0,1,2,\ldots$ the terms 
    \begin{equation}\label{eq:derres}
    D_{L,\cF_j,i}(\Sigma_W(\phi)\Sigma'_{W'}(\psi))|_{L_{j,\bbC}}
    \end{equation}
    in the expansion (\ref{eq:resdat}) of $\textup{Res}_{L_j,\cF_j}((\Sigma_{W}(\phi))(\Sigma'_{W'}(\psi))\Omega)$ on $L_{j,\bbC}$ 
    for the flag $\cF_j$ are holomorphic at all points of $B_{\fR}\cap (\sigma_j+iV^{L_j})$.
    \end{cor}
    \begin{proof}
    We first treat the case $\textup{Ord}_L(\Omega)=0$. 
    We note that if $L$ has order $0$ then $\textup{Ord}_{L_j}(\Omega)=0$ for all $j\geq 0$.
    (For any flag $\cF$ in $C$ it follows directly from the definition $C$ that $\textup{Ord}_{L_j}(\Omega)$ is weakly 
    increasing as a function of the codimension.)  
    It follows from Definition \ref{defn:casc}, Definition \ref{def:Adm} (the convexity and closedness of $\textup{Adm}(M)$), 
    Theorem \ref{thm:casc_ord_int}, the $W'$-equivariance of $\textup{Adm}(M)$, and Lemma \ref{lem:herit} (the heritability 
    of admissibility) that for all $j\geq 0$ we have $\sigma_j\subset \textup{Adm}(L_j)$. By Definition \ref{def:Adm} the 
    assertion follows.
    
    The next case we consider is the case where $\textup{Ord}_L(\Omega)>0$ and $\sigma_1\cap L=\{i_{\sigma_0}\}$
    with $i_{\sigma_0}\not=c_L$.
    By Theorem \ref{thm:casc_ord_int} we know $L$ is residual, $\textup{Ord}_{L}(\Omega)=1$, 
    $\textup{Ord}_{L_1}(\Omega)=0$ and 
    in all situations except if  $L=u(L_{sp})$ and $L_1=u(M_{sp})$ for some $u\in W'$ we have that $\sigma_0\subset 
    \textup{Adm}(L_1)$. This again implies the asserted result. In the special situation $L=L_{sp}$ and 
    $M=M_{sp}$, Theorem \ref{thm:mainspec} implies that the restriction to $L_{sp,\bbC}$ 
    \begin{equation}
    D_{L,\cF_0,1}(\Sigma_W(\phi)\Sigma'_{W'}(\psi)|_{M_{sp,\bbC}})|_{L_{sp,\bbC}}
    \end{equation}
    of the normal derivative $D_{L,\cF_0,1}(\Sigma_W(\phi)\Sigma'_{W'}(\psi)|_{M_{sp,\bbC}})$ 
    of $\Sigma_W(\phi)\Sigma'_{W'}(\psi)|_{M_{sp,\bbC}}$ does \emph{not} 
    have poles in the critical strip $\textup{Re}(x)\in [1/2,3/2]$ of the denominator $x-1/2\in \textup{Den}^{\hat{\fz}_{L_{sp,0}}}((s_7s_8)(M),s_7s_8,\Sigma)|_{L_{sp,0}}$.   
    From Theorem \ref{thm:casc_ord_int},
    the remaining denominators in $\textup{Den}^{\hat{\fz}_{L_{sp,0}}}((s_7s_8)(M),s_7s_8,\Sigma)|_{L_{sp,0}}$ are all inside $D_{L_{sp,0}}$  
    and, similarly, $\textup{Den}^{\hat{\fz}_{L_{sp,0}}}((s_7s_8)(M),s_7s_8,\Sigma')|_{L_{sp,0}}$ is a subset of $D'_{L_{sp,0}}$, showing 
    again that the assertion holds for $\sigma_0$ (while for $\sigma_j$ with $j>0$ it follows from the order $0$ case).
    
    Finally we have the case where $\textup{Ord}_L(\Omega)>0$ and $\sigma_1\cap L=\{c_L\}$. 
    If $L$ is residual then clearly $\sigma_0=[c_L,c_L]$. Otherwise  
    Theorem \ref{thm:casc_ord_int} implies that $L$ is intersected at $c_L$ in \emph{any} flag in $C$ in which $W'L$ 
    participates, which also implies that $\sigma_0=[c_L,c_L]$. Hence we do not need to move inside $L$ except in an arbitrary 
    small neighbourhood of $c_L$. Therefore (\ref{eq:derres}) for $(\sigma_0,L_0)$ follows in this case follows from 
    (\ref{eq:derres}) for $(\sigma_1,L_1)$, since the latter implies that $i$,  
    \begin{equation}\label{eq:hol}
    D_{L_1,\cF_1,i}(\Sigma_W(\phi)\Sigma'_{W'}(\psi))|_{L_{1,\bbC}}
    \end{equation}
    is holomorphic at all points of $B_R\cap (\sigma_1+iV^{L_1})$. 
    If $L_1$ is of order $0$ or is of positive order such that $\sigma_2\cap L_1=\{i_{\sigma_1}\}$ with 
    $i_{\sigma_1}\not=c_{L_1}$, then we have already established (\ref{eq:derres}) for $(\sigma_1,L_1)$ by the above. 
    If not, then $\sigma_1=[c_{L_1},c_{L_1}]=[c_L,c_L]$, and we consider likewise $L_1\subset L_2$.  
    Continuing like this we will find a $j\geq 0$ with $\textup{Ord}_{L_j}(\Omega)>0$, $\sigma_j=[c_{L_j},c_{L_j}]$
    and either $\textup{Ord}_{L_{j+1}}(\Omega)=0$ or $\textup{Ord}_{L_{j+1}}(\Omega)>0$ and $\sigma_{j+1}\not=[c_{L_j},c_{L_j}]$. 
    Then as above 
    \begin{equation}\label{eq:holi}
    D_{L_{j+1},\cF_{j+1},i}(\Sigma_W(\phi)\Sigma'_{W'}(\psi))|_{\sigma_{j+1}+iV^{L_{j+1}}}
    \end{equation}
    is holomorphic at all points of $B_R\cap (\sigma_{j+1}+iV^{L_{j+1}})$. 
    Now (\ref{eq:derres}) for $(\sigma_k,L_k)$ with $k\leq j$ is obtained from (\ref{eq:holi}) by taking suitable constant coefficient 
    derivatives in $L_{j+1}$ and restricting to $c_{L}+iV^{L_k}$. This proves the assertion.
    \end{proof}
    \begin{cor}\label{cor:nonsp}
    Assume that $F$,$\rho$, $\phi$ and $\psi$ and $C$ are as in Corollary \ref{cor:Cadmis}. 
    Choose $\fT>\fR$ and $\fT'>3\fT+2\fR^2$ as in \cite[V.2.2]{MW2}. We call a point $b\in L$ (with $L$ 
    an $\Omega$-pole space) $\fT'$-generic if for every flag $\cF=\{\dots \supset L_2\supset L_1\supset L_0=L \}$ 
    of pole spaces in $C$ ending in $L_\bbC$, 
    \begin{equation*}
    \textup{Res}_{L,\cF}((\Sigma_{W}(\phi^-)(\gl))(\Sigma'_{W'}(\psi)(\gl))\Omega)
    \end{equation*}
    is holomorphic on all points of $(b+iV^L)_{\leq \fT'}$. We say (following \cite{MW2}) that 
    $b$ is $\fT'$-generic near $c_L$ if the open ball in $L$ with center $c_L$ and radius $\Vert c_L-b\Vert$
    does not intersect any pole spaces in $L$ other than those which contain $c_L$.
    
    We can choose for every $L\in \cCO^C$ a base point $b_L\in L$ satisfying: 
    \begin{enumerate}
    \item[(i)] $b_L$ is $\fT'$-general. 
    \item[(ii)] $b_L$ is $\fT'$-general near $c_L$ whenever 
    $L^{\textup{temp}}\subset M^{\textup{temp}}$ for some residual pole space $M$, 
    \item[(iii)] We can 
    rewrite (\ref{eq:ind}) up to compact contours outside $V_{\leq \fT}$ as: 
    \begin{equation}\label{eq:atcenters}
    (\theta_\phi,\theta_\psi)=_\fT\sum_{L\in \cCO^C}\int_{(b_L+iV^L)_{\leq \fT}}\sum_{\cF\in \cF(L)}C_{L,\cF}
    \textup{Res}_{L,\cF}((\Sigma_{W}(\phi^-))(\Sigma'_{W'}(\psi))\Omega)
    \end{equation}
    for suitable constants $C_{L,\cF}$ as in (\ref{eq:nearcenter}). 
    \end{enumerate}
    \end{cor}
    \begin{proof}
    We begin with the $\fT$-truncated version of the initial induction formula for $\hat{R}'\subset \hat{R}$, cf. (\ref{eq:indT}): For 
    suitable nonzero real constants $C_L$ we have 
    \begin{align}\label{eq:trun} 
    (\theta_\phi,\theta_\psi)
    &=X_{V,p_V}(\psi\Sigma_W(\phi^-)\\
    \nonumber&=_\fT|W'|^{-1}\sum_{L'\in{\cL({\gO'_r})_+}}\int_{(p_{L,\infty,\fT'}+iV^L)_{\leq \fT}}C_L\Sigma'_{W'}(\psi)(\gl)\Sigma_{W}(\phi)(\gl)\Omega^L(\gl)
    \end{align}
    Recall that  $\cCO^C$ denotes a set of representatives of the $W'$-orbits of $\Omega$-pole spaces which appear in $C$, 
    and $\cCO^{C,res}\subset \cCO^C$ the subcollection of residual pole spaces in this set. 
    
    We define a collection $\gC$ of $\fT'$-general deformations of the paths of segments in $C$. 
    Let $(\sigma_M,M)\in C$ be a segment in $C$. A point $p\in \gs_M$ which lies on a
    $\Omega$ pole hyperplanes $L\subset M$ is also considered as a vertex of $C_L$ in $L$, in addition to 
    being a point of $M$, even if $p$ is an extremal point of $\gs_M$. We will replace $\gs_M$ by a ``broken line'' 
    $\gc_M$ in the terminology of \cite[V.1.5(c)]{MW2} which is $\fT'$-general (cf. \cite[V.1.5(b);V.2]{MW2}) and 
    goes from a $\fT'$-general point $i_{\gc_M}\in M$ which is near $i_{\gs_M}$ to a $\fT'$-general point $f_{\gc_M}\in M$ 
    which is near $f_{\gs_M}$.
    
    The intersection point of $\gs_M$ with a pole hyperplane $L\subset M$ separating 
    $i_{\gs_M}$ and $f_{\gs_M}$ (or $i_{\gc_M}$ and $f_{\gc_M}$, this amounts to the same thing) is a $\fT'$-general 
    point of $L$ near the original intersection point.   
    If there exists a segment $(\gs_N,N)\in C$ such that $M\subset N$ is a hyperplane and $i_{\gs_M}=u(\gs_N)\cap M$ or 
     $f_{\gs_M}=u(\gs_N)\cap M$, then its general deformation is determined by the $\fT'$-general broken line $\gc_N$ deforming 
     $\gs_N$. Otherwise $i_{\gs_M}$ equals a point of the form $p_{L,\infty}$, which we can freely deform to a nearby $\fT'$-general 
     point $p'_{L,\infty}$ by deforming $c_{L'}=c_L$ to a $\fT'$-general point $z_{L'}\in L'$ (and retaining the formula 
     $p'_{L,\infty}=z_{L'}+w\mathfrak{w}'$), or $f_{\gs_M}=c_M$, in which case we are allowed to choose a $\fT'$-general deformation 
     $z_M$ near $c_M$. Multiplicities are allowed in all of this, so if $i_{\gs_M}$ and $f_{\gs_M}$ are in multiple 
     ways obtained as intersections of the form $u_i(\gs_{N_i})\cap M$ then we will have several $\fT'$-general 
     deformations of this point near the original point. In this case we choose one of these $\fT'$-generial points
     and connect all others to this point by a ``tiny'' $\fT'$-general broken line, creating the $\fT'$-general curve 
     of these other deformations as well. 
     
     Note that some of the hyperplanes $L\subset M$ on which $i_{\gs_M}$ or $f_{\gs_M}$ lie may no longer be 
     intersected by $\gc_M$. If that is the case we anyway choose a $\fT'$-general point on $L$ near 
     $i_{\gs_M}$ or $f_{\gs_M}$ (whichever the case may be). This is natural, it only means that the 
     residue integral on $L$ with this base point is $0$. Nevertheless for the book keeping it is easier to 
    include these initial points on $L$, which enables us to define a $\fT'$-general deformation of all 
    segments $(\sigma_M,M)$ in $C$. We now define $\gC$ as the collection of all $\fT'$-general broken lines  
    $(\gc_M,M)$ obtained in this way. By construction $\gC$ is in canonical bijection with the segments of $C$ (counting with 
    multiplicity), a segment $(\gc_M,M)$ corresponding to the unique segment  $(\gs_M,M)$ it deforms. It follows 
    from our description of $\gC$ and this bijection that for all pole spaces $M\in\cCO^C$, the union of the collection 
     $\gC^M:=\{\gc\mid (\gc,M)\in\gC\}$ is connected (as was true for $C^M$), and empty iff $C^M$ is empty.
     In particular, we can take $\cCO^\gC=\cCO^C$ as a set of representatives of the $W'$-orbits of $\Omega$-pole 
     spaces which have a nontrivial intersection with $\gC$.
     
    Consider a path $\ldots,(\gc_2,L_2),(\gc_1,L_1), (\gc_0,L_0)=(\gc,L)$ in $\gC$ consisting 
    of a connected union of $\fT'$-generic piecewise linear paths in $\gC$ corresponding to a path of segments in $C$. Suppose that 
    $p\in\gc_j$ is a $\fT'$-general point in this path, with $(\gc_j,L_j)\in \gC$. Let 
    $\cF$ denote a flag of pole spaces in $C$ associated to the path. 
    Then the integral associated with $(\gc_j,L_j)$ at base point $p$ is:  
    \begin{equation}\label{eq:Mint}
    \int_{(p+iV^{L_j})_{\leq \fT}}\textup{Res}_{L_j,\cF_j}((\Sigma_{W}(\phi^-)(\Sigma'_{W'}(\psi))\Omega))
    \end{equation}
    By Corollary \ref{cor:Cadmis}, all the coefficient functions 
    \begin{equation}\label{eq:der}
    D_{L,\cF_j,i}(\Sigma_W(\phi)\Sigma'_{W'}(\psi))|_{L_{j\bbC}}
    \end{equation}
    in the expansion (\ref{eq:resdat}) are holomorphic at all points of $\gc_j$.
    
    The iterated residue data $\textup{Res}_{L,\cF}$ are meromorphic $(\textup{dim}(L),0)$-forms on $L_\bbC$.
    There are two ways at our disposal to transform the integral (\ref{eq:Mint}): By the $W'$-invariance of the integrand 
    we may, for any $u\in W'$, transform the integral to 
    \begin{equation}
    \int_{(u(p)+iV^{u(L_j)})_{\leq \fT}}\textup{Res}_{uL_j,u\cF_j}((\Sigma_{W}(\phi^-)(\Sigma'_{W'}(\psi))\Omega))
    \end{equation}
    or we may use the residue theorem of \cite[V.1.5(c)]{MW2} to move $p$ to another 
    generic point $q$ of $\gc$, at the cost of the residues picked up at $\gc\cap L_j$ 
    for hyperplanes $L\subset L_j$ separating $p$ and $q$, and compact contours 
    which lie entirely outside ${L_{j,\leq \fT}}$ in the sense of \cite[V.1.5(a)]{MW2}.  
     
    Using these two steps, (\ref{eq:trun}) and Corollary \ref{cor:tree} (the closedness of $C$ for intersections with $\gO$-poles up 
    to the $W'$-action, and the connectedness of $T_M\subset M$) for each $M\in\cCO^C$, 
    we may choose a single representing $\fT'$-generic base point $b_M\in \cup_{\gc\in \gC^M}\gc\subset M$ and then, starting at the 
    largest dimension and going down in dimension step by step, we can move the base points of the integrals in any $M\in\cCO^C$ 
    to $b_M$, using the residue theorem and the $W'$-action, to represent the residues of integrals in the representative pole 
    spaces $M\in\cCO^C$. In this way we 
    rewrite (\ref{eq:trun}) up to compact contours outside $V_{\leq \fT}$ as
    in (\ref{eq:atcenters}): 
    \begin{equation}
    (\theta_\phi,\theta_\psi)=_\fT\sum_{M\in \cCO^C}\int_{(b_M+iV^M)_{\leq \fT}}\sum_{\cF\in \cF(M)}C_{M,\cF}
    \textup{Res}_{M,\cF}((\Sigma_{W}(\phi^-))(\Sigma'_{W'}(\psi))\Omega)
    \end{equation}
    where $\cF(M)$ denotes the collection of flags of pole spaces in $C$ ending in $M$, and $C_{M,\cF}\in\bbR$ a constant which 
    takes into account the number of integrals in the $W'$-orbit which are being represented by this integral and the constants $C_L$
    in (\ref{eq:indT}).
    We choose $b_M=z_M$, a $\fT'$-generic point near $c_M$, if $M^{\textup{temp}}\subset\cup_{N\in \cCO^{C,res}}W'N^{\textup{temp}}$, which 
    is possible by the definition $C$.
    \end{proof}
    \begin{lem}\label{lem:canc0}
    Assume that $M$ is a pole space with $\textup{Ord}_M(\Omega)=0$. 
    Then there exists a constant $\fc_{M,\cF}\in\bbR$ (as in (\ref{eq:resdatCst})) such that: 
    \begin{equation}
    \textup{Res}_{M,\cF}((\Sigma_{W}(\phi^-))
    (\Sigma'_{W'}(\psi))\Omega)=
    \fc_{M,\cF}(\Sigma_{W}(\phi^-)\Sigma'_{W'}(\psi))|_{M_\bbC}\Omega^M.
    \end{equation}
    Assume that $C$ is as in Theorem \ref{thm:casc_ord_int}, 
    that $M\in \cCO^C$ has $\gO$-order $0$, and is such that 
    $M^{\textup{temp}}\not\subset\cup_{N\in \cL(\gO)}N^{\textup{temp}}$ where $\cL$ is the set of $\gO_r$-residual pole spaces. 
    Then  
    \begin{align}\label{eq:vanish}
    \sum_{\cF\in \cF(M)}C_{M,\cF}&\textup{Res}_{M,\cF}((\Sigma_{W}(\phi^-))
    (\Sigma'_{W'}(\psi))\Omega)=\\
    \nonumber&(\sum_{\cF\in \cF(M)}\fc_{M,\cF} C_{M,\cF})(\Sigma_{W}(\phi^-))(\Sigma'_{W'}(\psi))|_{L_\bbC}\Omega^L=0
    \end{align}
    for all $\phi,\psi\in \cP^\fR(V_\bbC)$ and for suitable real constants $C_{M,\cF}$ as in (\ref{eq:nearcenter}) and (\ref{eq:atcenters}).
    \end{lem}
    \begin{proof}
    Recall that (\ref{eq:atcenters}) is equal modulo $=_\fT$ to $X_{V,p_V}(\psi\Sigma_W(\phi^-))$ 
    (cf. equation (\ref{eq:trun})). 
    We will fix $\phi$ and view the right hand side of (\ref{eq:atcenters}) as a function of $\psi\in\cP^\fR(V_\bbC)$. 
    
    We replace the argument $\psi\Sigma_W(\phi^-)$ of $X_{V,p_V}$ in (\ref{eq:trun}) by a function $\theta\in\cP^\fR(V_\bbC)$. 
    Using 
    the iterated residue computation defined by $C$ to compute $X_{V,p_V}(\theta)$, and and using 
     \begin{equation}\label{eq:SigmaSigmaR}
    \Sigma'_{W'}(\psi)(\gl)\Sigma_W(\phi^-)(\gl)=|W'|A'_0(\psi\Sigma_{W}(\phi^-))
    \end{equation}
    to compare with the computation in the proof of (\ref{eq:atcenters}), we see that this yields   
    \begin{equation}\label{eq:standPW}
    X_{V,p_V}(\theta)=\sum_{M\in \cCO^C}\int_{(b_M+iV^M)}\sum_{\cF\in \cF(M)}C_{M,\cF}\textup{Res}_{M,\cF}(|W'|A'_0(\theta)\Omega)
    \end{equation}
    (since we can use the same contour shifts as before, but now without the truncation).
    
    \emph{Claim}: With the $b_M$ as in Corollary \ref{cor:nonsp}, we have 
    for all $\theta\in\cP^\fR(V_\bbC)$:
    \begin{equation}\label{eq:=0}
    \int_{(b_M+iV^M)}(\sum_{\cF\in \cF(M)}\fc_{M,\cF} C_{M,\cF})|W'|(A'_0(\theta)|_{M_\bbC})\Omega^M=0.
    \end{equation}
    Indeed, first note that modulo lower dimensional iterated residue integrals we can move $b_M$ to a generic point $z_M$ 
    near $c_M$, so the LHS of (\ref{eq:=0}) becomes
    \begin{equation}\label{eq:intzM}
    \int_{(z_M+iV^M)}(\sum_{\cF\in \cF(M)}\fc_{M,\cF} C_{M,\cF})|W'|(A'_0(\theta)|_{M_\bbC})\Omega^M.
    \end{equation}
    Such integral (\ref{eq:intzM}) viewed as a functional in $\theta$ defines a tempered boundary 
    value distribution with support contained in $W'M^{\textup{temp}}$ (here we view the integral as a sum over $u\in W'$ 
    of integrals with base points $u(z_M)$). Hence the union of the supports of the unique tempered local 
    distributions associated with (\ref{eq:=0}) is contained in the union of 
    $W'M^{\textup{temp}}$ together with a union of spaces of the form $N^{\textup{temp}}$ where $N$ runs over 
    a set of $\Omega$-pole spaces with $\textup{dim}(N)<\textup{dim}(M)$. 
    
    Assume that the above claim is false. Then we can choose a pole space $L$ of maximal dimension 
    subject to the conditions $\textup{Ord}_M(\Omega)=0$, $L^{\textup{temp}}\not\subset\cup_{N\in \cL}W'N^{\textup{temp}}$,  
    and that the integral (\ref{eq:=0}) (with $M$ replaced by $L$) defines a nonzero tempered functional on $\cP^\fR(V_\bbC)$. 
    Then, by the choices of the base points $b_M$ and the choice of $L$, if $W'M$ appears in $C$ and 
    $W'L^{\textup{temp}}\subsetneq W'M^{\textup{temp}}$ then (\ref{eq:=0}) does hold for $M$. 
    Indeed, since the order $\textup{Ord}_{M_i}(\Omega)$ 
    is weakly increasing with the codimension in a full flag of pole spaces in $C$, such $M$ would obviously also satisfy 
    $\textup{Ord}_{M}(\Omega)=0$ and $M^{\textup{temp}}\not\subset\cup_{N\in \cL}N^{\textup{temp}}$. The choice of $L$ thus implies 
    that equation (\ref{eq:=0}) has to hold for such $M$, for all $\theta\in \cP^\fR(V_\bbC)$. 
     
    On the other hand, we know by \cite[Proposition 3.6]{HO1} that 
    the unique tempered distributions defined by (\ref{eq:standPW}) are supported in the 
    union of the tempered forms of the residual spaces in $\cL$. 
    It follows then by the previous paragraph that the support of the unique tempered distributions defined by 
    the nontrivial (by our assumption) functional on $\cP^\fR(V_\bbC)$ defined by 
    \begin{equation}\label{eq:intbL}
    \int_{(b_L+iV^L)}(\sum_{\cF\in \cF(L)}\fc_{L,\cF} C_{L,\cF})|W'|(A'_0(\theta)|_{L_\bbC})\Omega^L
    \end{equation}
    is contained in $\cup_{N\in \cL}W'N^{\textup{temp}}\cap W'L^{\textup{temp}}$ (which has 
    dimension strictly smaller than $\textup{dim}(L)$) together with a union of spaces of the form $N^{\textup{temp}}$ 
    where $N$ runs over a set of $\Omega$-pole spaces with $\textup{dim}(N)<\textup{dim}(M)$
    (due to the residues when moving 
    the base point $b_L$ to $z_L$).
    Let $N_{W'}(L)$ be the subgroup of $W'$ which stabilizes $L$, and $W'_L=N_{W'}(L)/C_{W'}(L)$
    where $C_{W'}(L)$ is the subgroup of $W'$ fixing $L$ pointwise. 
    It follows that there exists a nonzero $W'_L$-invariant polynomial $Q$ on $L$ 
    such that  
    \begin{equation}\label{eq:z_L}
    \int_{z_L+iV^L}Q(\sum_{\cF\in \cF(L)}\fc_{L,\cF} C_{L,\cF})|W'|(A'_0(\theta)|_{L_\bbC})\Omega^L=0
    \end{equation}
    for all $\theta\in \cP^\fR(V_\bbC)$. 
    In fact, we can take $Q$ such that all the poles of the integrand, the meromorphic function  
    $(\sum_{\cF\in \cF(L)}\fc_{L,\cF} C_{L,\cF}))|W'|(A'_0(\theta)|_{L_\bbC})\Omega^L$
    on $L$, are lifted, so that (\ref{eq:z_L}) becomes independent of $z_L$. We move $z_L$ to $c_L$, 
    and observe that for any $f\in H^{W',R}$ (the space of $W'$-invariant holomorphic functions on 
    $\Vert\textup{Re}(\gl)\Vert\leq \fR$ of moderate growth on vertical strips) we have $A'_0(f\theta)|_{L_\bbC}=(f|_{L_\bbC})A'_0(\theta)|_{L_\bbC}$.  
    Now Sublemma V.2.11 of \cite{MW1} shows that there exists a nonzero $W'_L$-invariant polynomial $P$ on $L_\bbC$ such that 
    $PH_L^{W'_L,\fR}\subset H^{W',\fR}|_{L_\bbC}$, where $H_L^{W'_L,\fR}$ denotes the set of holomorphic functions on the intersection 
    of $L_\bbC$ with $\{\gl\mid \Vert\textup{Re}(\gl)\Vert\leq \fR\}$ which are invariant for $W'_L$ and of moderate growth on 
    vertical strips. Finally recall that $A'_0(\theta)|_{L_\bbC}$ is 
    $W'_L$-invariant. Hence for any $F\in H_L^{R}$ there exists an $f\in H^{W',\fR}$ such that $(\frac{1}{|W'_L|}\sum_{u\in W'_L}{}^uF)P=f|_{L_\bbC}$.
    Hence for all $F\in H^{\fR}_L$, we have: 
    \begin{align}
    \int_{c_L+iV^L}FQP(\sum_{\cF\in \cF(L)}&\fc_{L,\cF}  C_{L,\cF})|W'|(A'_0(\theta)|_{L_\bbC})\Omega^L=\\
    \nonumber&\int_{c_L+iV^L}Q(\sum_{\cF\in \cF(L)}\fc_{L,\cF}  C_{L,\cF})|W'|(A'_0(f\theta)|_{L_\bbC})\Omega^L=0.
    \end{align}
    This implies that $(\sum_{\cF\in \cF(L)}\fc_{L,\cF}  C_{L,\cF}) PQA'_0(\theta)|_{L_\bbC}=0$, hence for all $\theta\in\cP^\fR(V_\bbC)$ we have 
    \begin{equation}\label{eq:van}
    \sum_{\cF\in \cF(L)}(\fc_{L,\cF} C_{L,\cF})A'_0(\theta)|_{L_\bbC}=0, 
    \end{equation} 
    contradicting the choice of $L_\bbC$. Hence we conclude that (\ref{eq:=0}) holds for all 
    $M$ of $\gO$-order $0$ such that $M^{\textup{temp}}\not\subset\cup_{N\in \cL}N^{\textup{temp}}$. 
    In turn this implies (as above) that (\ref{eq:van}) hold for any such $\Omega$-pole space $M$. This concludes the proof of the claim.
    
    Now, back to the proof of the Lemma, at any $\lambda\in M_\bbC$ the value of the left hand side of (\ref{eq:van}) 
    is a sum $\sum_{\mu\in W'\gl}D_\mu(\theta)(\mu)$ of evaluations where $D_\mu$ is a differential operator 
    with rational functions on $M_\bbC$ as coefficients (but the sum is regular on $M_\bbC$). On the open dense subset 
    of elements $\mu\in W'M_\bbC$ which are regular in the sense that $\mu$ belongs to a single 
    subspace $uM_\bbC$ with $u\in W'$ and whose isotropy group $W_\mu'$ in $W'$ is the pointwise fixator group of 
    $uM$ in $W'$, the rational differential operator $D_\mu$ only depends on the coset $u\textup{Fix}_{W'}(M)$.
    We choose a set of representatives ${W'}^{M}$ for $W'/\textup{Fix}_{W'}(M)$, and write $D_u$ for the corresponding 
    rational differential operator on $uM_\bbC$. 
    If $\sum_{u\in {W'}^{M}}D_u(\theta)(u\gl)$ vanishes for all $\theta\in\cP^\fR(V_\bbC)$ then it is clear 
    that the differential operators $D_u$ on the open set of regular elements have to be all identically $0$. 
    Indeed, first observe that the ring of polynomials 
    $\bbC[V_\bbC]$ is a multiplier 
    ring for $\cP^\fR(V_\bbC)$. Similar to the argument in the proof Theorem (\ref{thm:denbase}), by multiplication with 
    a suitable polynomial we can first reduce the sum over $\mu\in W'\gl$ to a single term 
    by making $\theta$ vanish to a high order at all elements other than $u\gl$ (with $u\in {W'}^{M}$) in the orbit $W'\lambda$,  
    and assume (as we may) that $\theta$ is nonzero at $u\gl$. 
    Then multiply with a polynomial which is the inverse of the power series of $\theta$ at $u\gl$ up to high order. 
    Then we can further multiply with any monomial centered at $u\gl$, to show that  
    all coefficients of $D_u$ vanish at $u\gl\in uM_\bbC$. This applies to any $u\gl\in uM_\bbC$ in 
    the open set of regular elements, concluding the proof that $D_u=0$. 
    
    Now we return to the case where the argument of $X_{V,p_V}$ is $\psi\Sigma_W(\phi^-)$.
    Define $\textup{Adm}_\fT(M)\subset M$ to be an open neighbourhood of $\textup{Adm}(M)\subset M$
    such that $\Sigma_W(\phi^-)\Sigma'_{W'}(\psi)$ is holomorphic at all points of the sets of the form $(p+iV^M)_{\leq \fT}$ with $p\in\textup{Adm}_\fT(M)$. 
    It follows that the $D_u$ vanish in particular on the regular elements in the sets $u(p+iV^M)_{\leq \fT}$ with  $p\in \textup{Adm}_\fT(M)$, and that 
    for $\gl\in \textup{Adm}_\fT(M)$ regular we have 
    \begin{align}\label{eq:vanpure}
    \sum_{\cF\in \cF(M)}C_{M,\cF}&\textup{Res}_{M,\cF}((\Sigma_{W}(\phi^-))(\Sigma'_{W'}(\psi))\Omega)(\gl)=\\
    \nonumber&(\sum_{\cF\in \cF(M)}\fc_{M,\cF} C_{M,\cF})\Sigma'_{W'}(\psi)(\gl)\Sigma_W(\phi^-)(\gl)\gO^M(\gl) =\\
    \nonumber&(\sum_{\cF\in \cF(M)}\fc_{M,\cF} C_{M,\cF})A'_0(\psi\Sigma_W(\phi^-))(\gl)\Omega^M(\gl)=\\
    \nonumber&(\sum_{\cF\in \cF(M)}\fc_{M,\cF} C_{M,\cF})(\sum_{u\in {W'}^{M}}D_u(\psi\Sigma_W(\phi^-))(u\gl))\Omega^M(\gl)=0.
    \end{align}
    Since the expression is a priori meromorphic on $M_\bbC$, this implies the vanishing (\ref{eq:vanish}) as desired.  
    This holds for all $\psi\in \cP^\fR(V_\bbC)$, as was to be shown.
    \end{proof}
    \subsection{Proof of Main Theorem \ref{thm:mainC}}\label{s:proofMT}
    \begin{proof}
    Choose $C$ for $R'\subset R$ as constructed in Theorem \ref{thm:casc_ord_int}. 
    Now the result follows directly from Corollary \ref{cor:nonsp}, Lemma \ref{lem:canc0}
    and Theorem \ref{thm:mainspec}.
    \end{proof} 
    \section{Appendix A: The Cascade $C$ and the pseudocode for the construction of the Cascade}\label{a:Cascade}
    
    Let $R$ be a root system, $R_+\subset R$ a subset of positive roots, and $R'\subset R$ a maximal proper standard Levi subsystem.
    Let $\mathfrak{w}'$ denote the unique fundamental weight orthogonal to $R'$. 
    \begin{defn}\label{defn:casc}
    A cascade $(C,\textup{Std}(C),\textup{SP}(C))$\index{$C$, Cascade!$\textup{Std}(C)$}\index{$C$, Cascade!$\textup{SP}(C)$}\index{$C$, Cascade}
    for $R'_+\subset R_+$ consists of 
    \begin{enumerate}
    \item[(i)] A finite set $C$ of pairs $(\sigma,L)$ with $L$ an $\Omega$-pole space, and $\sigma=[i_\sigma,f_\sigma]\subset L$ 
    a (directed) segment. 
    \item[(ii)] A distinguished set $\textup{Std}(C)$ of pairs $(L_0,d)$ with $L=dL_0\in\cL(\gO)$ where $L_0$ is a standard 
    space and $d\in W$ has minimal length in $W'd$, such that the collection $dL_0\subset \cL(\gO)$ with $(L_0,d)$ in 
    $\textup{Std}(C)$ forms a complete set of representatives of $W'$-orbits in $W'p_2(C)\subset \cL(\gO)$ 
    (where $p_2(\sigma,L)=L$ for $(\sigma,L)\in C$). Let $\cCO^C:=\{dL_0\mid(L_0,d)\in\textup{Std}(C)\}$
    \index{$\cCO^C$}, which is a set of representatives of the $W'$-orbits of pole spaces $M$ in $C$.
    \item[(iii)] For each $(\sigma,L)\in C$, a choice of a standard pair $(L_0,w_L)=(L_0,u_Ld_L)$ with $u_L\in W'$ 
    such that $L=w_L(L_0)$ and $(L_0,d_L)\in\textup{Std}(C)$. 
    Let $\textup{SP}(C)$ denote this set of standard pairs $\{(L_0,w_L)\mid L\in p_2(C)\}$.   
    \end{enumerate}
    We view the union of the segments 
    in $C$ as a graph whose vertices are the endpoints of its segments $\gs_M$ and the intersections of $\sigma_M$ 
    with the pole spaces of $\gO$ contained in $M$ with codimension one, and whose edges are the subsegments of 
    the $\gs_M$ defined by these vertices. Given a pole space $M\in\cL(\gO)$ we write $C_M\subset C$ for the subset of pairs of 
    the form $(\gs,M)\in C$.
    For $(M_0,d)\in\textup{Std}(C)$ put $M=dM_0\in \cCO^C$. 
    We denote by 
    $\cCO^{C,res}=\{M\in\cL(\gO_r)\mathrm{\ (i.e.\ residual)\ }\mid \exists\gs_M: (\gs_M,M)\in C\}$\index{$\cCO^C$!$\cCO^{C,res}$}. 
    
    The collection $(C,\textup{Std}(C),\textup{SP}(C))$ has to satisfy the following requirements:  
    \begin{enumerate}
    \item[(iv)] Let $L$ be a standard residual pole space of the form $L=L'+\bbR\mathfrak{w}'\mathrm{\ with\ }L'\in\cL({\gO'_r})_+$.
    Then $(\gs_{L,0},L):=([p_{L,\infty},c_L],L)\in C$, $(L,1)\in \textup{Std}(C)$, and $w_L=1$.
    \item[(v)] For every $(\gs_N,N)\in C_N$, if $M\subset N$ is a pole space for $\Omega$ of codimension $1$ in $N$ 
    such that $M\cap \gs_N$ is nonempty, then there exists a $u\in W'$ and a 
    $(\sigma_K,K)\in C$ such that $K=u(M)$ and $\{i_{\gs_K}\}=u(\gs_N)\cap K$.
    \item[(vi)] Conversely, if $(\gs_M,M)\in C_M$ and $(\gs_M,M)\not=(\gs_{L,0},L)$ for any $L'\in\cL'$, then there exists a 
    segment $(\gs_N,N)\in C$ and a $u\in W'$ such that $M\subset u(N)$ has codimension 1, and 
    $\{i_{\gs_M}\}=u(\gs_N)\cap M$. 
    \item[(vii)] If $(\sigma_M,M)\in C$ and $M^{\textup{temp}}\subset\cup_{N\in \cCO^{C,res}}W'N^{\textup{temp}}$ then $f_{\sigma_M}=c_M$.
    Otherwise $f_{\gs_M}=i_{\gs'_M}$ for some $u(\gs'_M,M)\in C$ with $u\in W'$. 
    \end{enumerate}
    \end{defn}
    \begin{defn}
    A segment of the form $(\gs_{L,0},L)\in C$ as in Definition \ref{defn:casc}(iv) is called a segment of the first generation. 
    Suppose that recursively we have defined a collection of segments in $C$ of the $n^{th}$ generation with $n> 1$. 
    A segment $(\gs,M)\in C$ with $\gs=[i_\gs,f_\gs]$ which does not belong to the $m^{th}$ generation for some $m\leq n$, and such 
    that there exists a $u\in W'$ and a segment $(\gs_N,N)\in C$ of generation $n$ such that $\{i_\gs\}=u(\gs_N)\cap M$ and 
    $M\subset u(N)$ is a pole space of codimension $1$, then we call $(\gs,M)$ of generation $n+1$. 
    \end{defn}
    \begin{cor}\label{cor:tree}
    $C$ is ``closed for intersections with $\gO$-pole spaces'' in the sense that every intersection $L\cap \sigma_M$ of a segment 
    $(\sigma_M,M)\in C$ with a hyperplane 
    $L\subset M$ with $L\in\cL(\gO)$, is contained in a segment $(\sigma_L,L)\in C$ up to the action of $W'$. $C$ is ``connected'' 
    in the sense that for every pole space $M\subset \cCO^C$, if $T_M$ denotes the collection of segments of the form 
    $(u_L^{-1}(\sigma_L),M)$ in $M$ where $(\sigma_L,L)\in C$ is such that $L=u_L(M)$, then $T_M$ is 
    a connected oriented tree 
    in $M$ (possibly a single vertex), containing $c_M$ if $M^{\textup{temp}}\subset N^{\textup{temp}}$ for some 
    residual pole space $N$. 
    \end{cor}
    \begin{proof}
    This is immediate from the definitions.
    \end{proof}
   
   \subsection{The pseudocode for generating the cascade}
We will now describe the algorithm used to construct $(C,\textup{Std}(C),\textup{SP}(C))$ for $R'_+\subset R_+$. The specific cases where we needed computer assistance were the pairs $(\textup{B}_3,\textup{F}_4),(\textup{D}_5,\textup{E}_6),(\textup{E}_6,\textup{E}_7),(\textup{E}_7,\textup{E}_8)$. At {\tt https://github.com/mgdemartino/SphAutSpec}, one can find the outputs as well as the codes used.

The algorithm has $(r+1)$ phases, with $r$ being the rank of the roots system $R$. Each phase, denoted $k$, correspond to the codimension of the spaces that are crossed in the cascade and the aim is to construct two databases of spaces: 
\[
\Gen = (\Gen_k)_{k=0,\ldots,r}\quad \textup{ and } \quad \Std = (\Std_k)_{k=0,\ldots,r}.
\]
The $\Gen$ database will store the information pertaining to $C$ while the $\Std$ database stores the information needed for $\textup{Std}(C)$ and $\textup{SP}(C)$.
\subsubsection{The $\Gen$ database}
We will construct $\Gen = (\Gen_k)_{k=0,\ldots,r}$ and in each phase $k$, $\Gen_k$, is a 2D-array (a matrix) whose rows correspond to pole spaces of codimension $k$ crossed in the algorithm. To each row, the following information is stored:
\begin{center}
\begin{tabular}{ r c l }
column \#1&:& $\hat{P}_L$, the pole set of $L$,\\
column \#2&:& $p_L$, the initial point of crossing,\\
column \#3&:& $c_L$, the center of $L$,\\
column \#4&:& $\textup{Ord}_L(\Omega)$, the order of $L$ along $\Omega$,\\
column \#5&:& $\textup{Num}^L$, the numerator of $\Omega^L$,\\
column \#6&:& $\textup{Den}^L$, the denominator of $\Omega^L$,\\
column \#7&:& $\mathtt{Cd}_L = (J',\gamma',\rl_L)$, the crossing datum.
\end{tabular}
\end{center}
Here, $\Omega^L$ was defined in (\ref{eq:rewrO}), $\textup{Num}^L$, $\textup{Den}^L$ are lists of coroots that index the numerator and denominator of $\Omega^L$ and the crossing datum $\mathtt{Cd}_L$ traces the flag $\cF$ of $\Omega$ pole spaces that creates $L$: $(J',\gamma')$ encodes the information of an initial pole space $L'\in\cL(\Omega_r')$ where the cascade starts, with $J'$ a subset of the diagram for $R'$ and $\gamma'$ is a weighted Dynkin diagram for $J'$; and $\rl_L$ is an (ordered) list of coroots $\hat{\ga}$ indicating the hyperplanes $\{\hat\ga = 1\}$ whose successive intersection creates a flag of pole spaces ending in $L$. We stress that each flag of $\Omega$ pole spaces appearing in $C$ will be represented by a unique row in $\Gen_k$.

\subsubsection{The $\Std$ database}
We construct $\Std = (\Std_k)_{k=0,\ldots,r}$ which stores data for (choices of) standard pole spaces corresponding to $\Gen$, up to $W'$-conjugation. Specifically, in each phase $k$, $\Std_k$ is a 2D-array whose rows correspond to a $W'$-orbit\footnote{The algorithm crosses few members of a $W'$-orbit, and we track these in the database.} of pole spaces of codimension $k$. For each such row, we store the following information:
\begin{center}
\begin{tabular}{ r c l }
column \#1&:& $\hat{P}_{L_0}$, the pole set of the chosen standard form,\\
column \#2&:& $c_{L_0}$, the center of $L_0$,\\
column \#3&:& $\uw$, the list of $W$-elements, for each $L$ in this $W'$-orbit,\\
column \#4&:& $\usig$, the list of segments $[p_L,q_L]$ in $L_0$, for each $L$ in this $W'$-orbit,\\
column \#5&:& $\textup{Ord}_{L_0}(\Omega)$, the order of $L_0$ along $\Omega$,\\
column \#6&:& $\sub$ tag `$\ttrue$' or `$\tfalse$',\\
column \#7&:& $\mathtt{Pd}_{L_0} = (J_{L_0},\gamma_{L_0})$, the standard parabolic datum of $L_0$,\\
column \#8&:& $\hat{R}^\perp_{L_0}$, the perpendicular coroots $\{\hat\alpha\in\hat{R}_+\mid (\hat\alpha,\hat\beta)=0,\forall\hat\beta\in\hat{R}_{L_0}\}$,\\
column \#9&:& test if there are non-constant parallel coroots to $\usig$, \\
column \#10&:& $\upl$, the list  of the pole sets of each $L$ in this $W'$-orbit.
\end{tabular}
\end{center}
\begin{rem}\label{rem:AlgIniRems}
Some remarks: 
\begin{itemize}
\item Given a pole space $L$, we load a procedure $\mathtt{stdData}$ which computes the function $L\mapsto L_0$ corresponding to the choice of standard space $L_0$. In broad strokes, we compute an element $\lambda_L = c_L + x^L\in L$ with $|x^L|$ large enough so that if $w\lambda_L$ is dominant, then $w(L)$ is standard. The procedure $\mathtt{stdData}$ returns this $w$ and the parabolic datum $\mathtt{Pd}_{L_0}$.
\item $\sub(L_0) = \ttrue$ if $L_0$ is residual or if there is a $w'\in W'$ and a residual space $M$ such that $w'(L_0)\subset M$ and $c_{L_0} = c_M$. It is $\tfalse$ if otherwise.
\item It is useful to store the set of perpendicular roots $\check{P}_{L_0}$ in order to speed some auxiliary routines.
\item For column \#9, the presence of non-constant roots parallel to a segment could have caused to make modifications in the algorithm. Fortunately, such phenomenon was not observed in any of the cases computed.
\end{itemize}
\end{rem}

\subsubsection{Pseudocode for the algorithm}
To initialize the algorithm, we construct $\Gen_0$, which has a single row corresponding to $L=V$, the minimal principal series. Here, $p_{V,\infty} = t_\infty\fw'$ with $t_\infty\gg 0$ and $\textup{Num}^V,\textup{Den}^V$ are, respectively, the numerator and denominator of $\Omega$ defined in (\ref{eq:Omega}). For each $k$ from $0$ to $r$, we assume $\Gen_k$ is constructed and we run the procedure $\mathtt{CascPhase}(k)$, which constructs $\Gen_{k+1}$ and $\Std_k$ if $k<r$ and when $k=r$, $\mathtt{CascPhase}(r)$ just constructs $\Std_r$. In the algorithm, described below, the enumerated comments are the following:

\SetAlFnt{\small} 
\begin{algorithm}[!htp]
$\mathtt{Gen} =\Gen_k$\Comment*[r]{(1)}
$\mathtt{Std} =()$\Comment*[r]{(2)}
$\mathtt{Nex} =\{1^{\textup{st}}$-gen spaces of codimension $k+1\}$\Comment*[r]{(3)}
\For{$L \in \Gen$}{
  \eIf{$\exists(L_0,g)\in\Std$ with $g(L_0) \in W'(L)$}{
    compute $w$ so that $w(L_0) = L$\Comment*[r]{(4)}
    \eIf{$\sub=\ttrue$}{
      \eIf{$w^{-1}([p_L,c_L])\in\usig$}{
        append $\uw,\usig,\upl$ in $\Std$ \Comment*[r]{(5)}
      }{
        compute $\mathtt{LPols} = \mathtt{PolesCrossed}(L,p_L,c_L)$ \Comment*[r]{(6)}
        $\mathtt{Nex} = \mathtt{add2Gen}(\mathtt{LPols},\mathtt{Nex})$\;
        append $\uw,\usig,\upl$ in $\Std$\Comment*[r]{(7)}
      }
    }{
      compute $q_0$, where $\min_{q\in \mathtt{Pts}_{L_0}}\{d(p_0,q)\}$ is attained \Comment*[r]{(8)}
      \eIf{$[p_0,q_0]\in\usig$}{
        append $\uw,\usig,\upl$ in $\Std$ \Comment*[r]{(9)}
      }{
        compute $\mathtt{LPols} = \mathtt{PolesCrossed}(L,p_L,q_L)$ \Comment*[r]{(10)}
        $\mathtt{Nex} = \mathtt{add2Gen}(\mathtt{LPols},\mathtt{Nex})$\;
        append $\uw,\usig,\upl$ in $\Std$\Comment*[r]{(11)}
      }
    }
  }{
    compute $\mathtt{stdData}(L)$ and $\sub = \sub(L_0)$\;
      \eIf{$\sub=\mathtt{true}$}{
        compute $\mathtt{LPols} = \mathtt{PolesCrossed}(L,p_L,c_L)$ \Comment*[r]{(12)}
        $\mathtt{Nex} = \mathtt{add2Gen}(\mathtt{LPols},\mathtt{Nex})$\;
        $\mathtt{Std} =\mathtt{add2Std}(\mathtt{stdData}(L),\mathtt{Std})$\Comment*[r]{(13)}	
      }{
        $\mathtt{Std} =\mathtt{add2Std}(\mathtt{stdData}(L),\mathtt{Std})$\Comment*[r]{(14)}
      }
  }
}
\Return $\mathtt{Nex}, \Std$\;
\end{algorithm}

\begin{enumerate}[(1)]
\item Reads the previously constructed $\Gen_k$ database.
\item Initializes the $\Std_{k}$ database.
\item Initializes the $\Gen_{k+1}$ database by inputing the data for $1^{\textup{st}}$-generation spaces $\{L' = L(J',\gamma')\mid (J',\gamma')\}$. These are given by
\[
L = c_{L'} + t_\infty\fw'.
\]
All columns of $\Gen$ corresponding to this $L$ are constructed using $(J',\gamma')$ as input. The root crossing data $\rl$ in column \#7 of $\Gen$ is empty, in this case.
\item We load an auxiliary procedure $\mathtt{InStdWp}$ which, if $\ttrue$ returns $u\in W'$ such that $g(L_0) = u(L)$. We let $w=u^{-1}g$. This procedure uses the set of perpendicular coroots $\check{R}^\perp_{L_0}$.
\item If the segment $[w^{-1}(p_L),w^{-1}(c_L)]$ is already a member of $\usig$, the possible spaces generated by the contour-shift along this segment were already considered. We just enlarge the lists $\uw,\usig,\upl$ to keep track of the occurrences.
\item The routine $\mathtt{PolesCrossed}(L,p,q)$ computes all spaces crossed while making the contour shift in $L$ along the segment $[p,q]$. It computes all information needed for columns \#1 -- 7 of $\Gen$. If $p = q$, nothing is done. This allows the algorithm to continue in codimension $r$ as all spaces created are points and there is no extra spaces generated and $\mathtt{Nex}=()$ when $k=r$.
\item Since the row corresponding to the $W'$-orbit of $L_0$ already exists, we only update the lists in columns \#3, 4, 9 of $\Std$.
\item Here, $p_0=w^{-1}(p_L)$ and $\mathtt{Pts}_{L_0}$ is the set of initial points of segments in $L_0$. In a $\sub = \tfalse$ spaces, we need not shift to the center, but rather we can stay stay in the admissible, convex region where all the initial points lie.
\item Similar to (5). In this case, we will create a segment that was already considered so there is no need to perform the contour-shift again.
\item Here, $q_L = w(q_0)$. 
\item Same as (7).
\item Same as (6).
\item In this situation, we are adding a new entry to the $\Std$ database. The procedure $\mathtt{stdData}(L)$ was discussed in Remark \ref{rem:AlgIniRems}. At this stage we also test for the existence for non-constant roots parallel to the segment $[p_L,c_L]$.
\item In this particular situation, the pole space in question will not contribute to the spectrum and we do not shift to the center. In the first appearance of a $\sub = \tfalse$ space, the segment considered is $w^{-1}([p_L,p_L])$. 
\end{enumerate}
   
   \section{Appendix B: The classical split groups}\label{a:class}
    The following Theorem is essentially a reformulation of Moeglin's results \cite{M1}, with the exception 
    of part $(vi)$ (Moeglin follows a slightly different approach at this point, and does not define the set $\textup{Adm}(L)$).
    We include it for completeness.
    \begin{thm}\label{thm:classmain}
    Let $R$ be a root system of rank $n$ of classical type $\uX_n$, and let $R'\subset R$ be 
    the standard Levi subsystem of type $\uX_{n-1}$. There exists a cascade $C_\uX$ such that 
    \begin{enumerate}
    \item[(i)] All segments $(\sigma,L)\in C_\uX$ of the third generation and higher are points of the form 
    $\sigma=[c_L,c_L]=\{c_L\}$.
    \item[(ii)] All pole spaces $L$ in the first two generations of $C_\uX$ satisfy $\textup{Ord}_L(\gO)=0$.
    \item[(iii)] For all $(\sigma,L)\in C_\uX$ with $L$ not residual we have $\gs=[c_L,c_L]=\{c_L\}$. 
    \item[(iv)] If $(\gs,L)\in C_\uX$ is in the first two generations of $C_\uX$ then $\hat{P}_L\subset \hat{R}_+$. 
    \item[(v)] If $(L_0,d)\in\textup{Std}(C_\uX)$ is a standard datum from the first two generations of $C_\uX$ then 
     all regular envelopes of $(L_0,d)$ are $d$-good. 
    \item[(vi)] All residual pole spaces $L$ in $C_X$ with $\textup{Ord}_L(\gO)=0$ satisfy $c_L\in\textup{Adm}(L)$.
    \end{enumerate}
    \end{thm}
    \begin{proof}
    Recall that 
    \begin{align}\label{eq:FL}
    \gO=F{\gO'_r}&=
    \frac{c'(-\gl)}{c(-\gl)}\frac{d\gl}{c'(\gl)c'(-\gl)}\\
    &\nonumber\prod_{\hat\ga\in \hat{R}_+\backslash\hat{R}'_+}\frac{\hat\ga}{\hat\ga-1}\frac{d\gl}{c'(\gl)c'(-\gl)}
    \end{align}
    It makes sense to treat the case of type $\uA$ separately first:
    \subsubsection{The cascade $C_\uA$}
    Consider coordinates $(t_0,t_1,\dots,t_n)$, and consider a partition 
    $\pi':=(k'_1,k'_2,\dots,k'_T)$ of $n$.  There is a corresponding standard Levi $R_{\pi'}\subset R'$ 
    of $\uA_{n-1}$, of type $\uA_{k'_1-1}\times\dots\times \uA_{k'_T-1}$, corresponding to a partitioning 
    of the variables $t_1,t_2,\dots,t_n$ in consecutive intervals of size $k'_1, k'_2, \dots, k'_T$.
    By abuse of notation, denote by $x=(x_1,x_2,\dots,x_T)$ the vector perpendicular to all roots of $R_{\pi'}$ 
    with its first $k'_1$ coordinates equal to $x_1$, the coordinates $t_{k'_1+1},\dots,t_{k'_1+k'_2}$ equal to $x_2$, 
    etcetera.
    There is a unique standard regular  residual subspace $L'_{\pi'}$ of $R_{\pi'}$ such that a general element  
    is of the form $c_{\pi'}+x$ where $x=(x_1,x_2,\dots,x_T)$ is as above, and 
    \begin{equation}
    c_{\pi'}=(\frac{k'_1-1}{2}, \dots, \frac{1-k'_1}{2},\frac{k'_2-1}{2}, \dots, \frac{1-k'_2}{2}, \dots,\frac{k'_T-1}{2}, \dots, \frac{1-k'_T}{2})
    \end{equation}
    We put $L_{\pi}=L'_{\pi'}+\bbR\fw'=\{(x_0,t_1,\dots,t_n)\mid \lambda\in\bbR,\ (t_1,\dots,t_n)\in L'_{\pi'} \}$.
    This is the residual subspace in standard position for $R$ corresponding to the composition $\pi=(1,\pi')$ of $n+1$ (we add one 
    part to $\pi'$ of size $1$). We denote this composition $\pi$ of $n+1$ by $(k_0,k_1,\dots,k_T)$. The first generation 
    segment $\gs_{L,0}$ in $L_\pi$ is of the form $\gs_{L,0}=[p_{L_\pi,\infty},0]\fw'+c_{L'}$ with $p_{L_\pi,\infty}\gg0$.
    
    We write $\gO^{L_\pi}=\textup{Res}_{L_\pi}(F\gO)=(F_{L_\pi})\textup{Res}_{L'_{\pi'}}({\gO'_r})dx_0$ with $F_{L_\pi}=F|_{L_\pi}$.
    The pole coroots $\hat\ga\in \hat{R}_+\backslash \hat{R}'_+$ of $F_{L_\pi}$ 
    are the ones of the form $\hat\ga_j=t_0-t_{i_j}$ for $j=1,\dots ,T$ and $i_j=k_1+k_2+\dots +k_{j-1}+1$, 
    and all these poles are simple poles. 
    Thus $L_{\pi,\hat\ga_j}=L_\pi\cap \{\hat\ga_j=1\}$ is formed by points $(x_0,t_1,\dots,t_n)\in L_\pi$ 
    such that $x_0-(x_j+\frac{k_j-1}{2})=1$, in other words $x_0=x_j+\frac{k_j+1}{2}$. 
    This yields one extra pole $\hat\ga_j$ for $L_{\pi,\hat\ga_j}$, and so $P_{L_{\pi,\hat\ga_j}}\subset \hat{R}_+$.
     We have $\textup{Ord}_{L_{\pi,\hat\ga_j}}(\gO^{L_\pi})=0$, 
    and $L_{\pi,\hat\ga_j}$ is in the $W$-orbit of a regular residual subspace in standard position $L_{\pi_j}$, where $\pi_j$ denotes 
    the partition of $n+1$ obtained from $\pi'$ by increasing the j-th part $k'_j$ by $1$. We will denote by 
    $\pi_j=(\gk_1,\dots,\gk_T)$ the resulting composition of $n+1$. 
    
    We see that $L_{\pi,\hat\ga_j}\cap \gs_{L,0}=p_{\pi,\hat\ga_j}=(\frac{k_j+1}{2},c_{\pi'})\in L_{\pi,\hat\ga_j}$, and 
    the center $c_{\pi,\hat\ga_j}$ of $L_{\pi,\hat\ga_j}$ is equal to $c_{\pi,\hat\ga_j}:=p_{\pi,\hat\ga_j}-1/2(e_0+E_j)$, where $e_0$ is 
    the first basis vector of $\bbR^{n+1}$, and $E_j$ is the $j$-th basis vector for $V^{L'_{\pi'}}=\bbR^{T}$. 
    In particular, for any coroot $\hat\gb$ we see that $|\hat\gb(p_{\pi,\hat\ga_j}-c_{\pi,\hat\ga_j})|\leq 1/2$
    while $\hat\gb(p_{\pi,\hat\ga_j}), \hat\gb(c_{\pi,\hat\ga_j})\in\mathbb{Z}/2$.
    
    This implies that when moving in $L_{\pi,\hat\ga_j}$ in a straight line segment from $p_{\pi,\hat\ga_j}$ to $c_{\pi,\hat\ga_j}$ 
    we can only possibly cross a pole hyperplane of the form $\hat\beta=1$ for a nonconstant coroot $\hat\beta$   
    at the extreme points $c_{\pi,\hat\ga_j}$ or $p_{\pi,\hat\ga_j}$. 
    
    If  we meet a pole space $M$ which is a hyperplane $M\subset L_{\pi,\hat\ga_j}$ at the center $c_{\pi,\hat\ga_j}$, then obviously $c_M=c_{\pi,\hat\ga_j}$, 
    so that the cascade of contour shifts stops at $c_{\pi,\hat\ga_j}$.  
    
    We claim that there are no pole spaces $M\subset L_{\pi,\hat\ga_j}$ of codimension one containing $p_{\pi,\hat\ga_j}$.  
    Indeed, consider the rational form $F_{L_{\pi,\hat\ga_j}}(\gl)\frac{d\nu_{L'}(\gl')}{d\lambda'}$ on $L_\pi$ near $p_{\pi,\hat\ga_j}$. 
    The Plancherel density of $\frac{d\nu'_{L_{\pi}}(\gl')}{d\gl'}$ is smooth in a neighbourhood of the line $c_{L'}+\mathbb{R}\fw'$ (which contains 
    $p_{\pi,\hat\ga_j}$ by construction), hence this regular factor can simply be restricted to $L_{\pi,\hat\ga_j}$ near $p_{\pi,\hat\ga_j}$. 
    Similar to what was observed above, the pole roots of  $F_{L_{\pi,\hat\ga_j}}(\gl)$ 
    are the ones of the form $\hat\ga_m=t_0-t_{i_m}$ for $m\in\{1,\dots ,T\}\backslash\{j\}$ 
    and $i_m=k_1+k_2+\dots +k_{m-1}+1$ and yields the equation $x_j-x_m=\frac{k_m-1}{2}-\frac{k_j-1}{2}$ for 
    $M\subset L_{\pi,\hat\ga_j}$, while at $p_{\pi,\hat\ga_j}$ we have $x_j=x_k=0$.  
    The condition that $M$ contains $p_{\pi,\hat\ga_j}$ implies therefore that $k_m=k_j$.
    However, it is well known that the Plancherel density of $d\nu'_{L_{\pi}}(\gl')$ has a (double) zero at $x_m=x_j$ if 
    $k_m=k_j$, since the analytic $R$-group is trivial for $\textup{GL}_{n+1}$. Therefore, $M$ is 
    not a pole space, proving the claim.
    The above discussion proves Theorem \ref{thm:classmain} $(i),(ii),(iv)$ for $\uX=\uA$, and shows that there 
    are only residual (hence regular, since we are in type $\uA$) pole spaces in the first two generations of $C_\uA$. 
    This also implies $(iii)$, while $(v)$ follows directly from $(iv)$ and the definitions. All pole residual pole $L$ 
    spaces in type $\uA$ are regular, and thus satisfy $\textup{Ord}_L(\gO)=0$. In the first two generations 
    of $C_\uA$ they are good by $(iv),(v)$, and then (vi) follows form Proposition \ref{lem:incl}. In higher generations 
    the residual pole spaces $L$ in $C_\uA$ are met at $c_L$ by $(i)$, proving $(vi)$ by induction on the generation 
    and Lemma \ref{lem:herit}.
    This finishes the proof for $C_\uA$.
    
    \subsubsection{The cascade $C_\uX$ for $\uX=\uB,\,\uC$ or $\uD$}
    We take ${R}'=\uX_n$.
    Let $N=2n+1$ if $\uX=\uB$, and $2n$ if $\uX=\uC,\uD$. 
    Consider the set $\Pi_{\uX,n}$ of partitions $\pi'$ of $N$ whose 
    odd parts have even multiplicity if $\uX=\uB$, or otherwise 
    whose even parts have even multiplicity. 
    
    The set $\mathcal{L}$ of $W'$-orbits of residual subspaces of ${R}'$ 
    is in bijection with $\Pi_{\uX,n}$. Given $\pi'\in\Pi_{\uX,n}$ we write 
    $\pi'$ in the form $\pi'=(k'_1,k'_1,k'_2,k'_2,\dots,k'_T,k'_T,q'_1,\dots,q'_S)$
    where the $q'_i$ are distinct and even if $\uX=\uB$, and otherwise distinct 
    and odd. We see that $\kappa'=(k'_1,\dots,k'_T)$ is a partition of $M$, 
    $q':=(q'_1,\dots,q'_S)$ is a distinct partition of "type $\uX$" of $N'$, 
    such that $2M+N'=N$.
    
    Following Moeglin we will denote the partition $\pi'$ by the 
    bipartition of $M+N'$ given by $(\kappa';q')=(k'_1,\dots,k'_T;q'_1,\dots,q'_S)$. 
    The corresponding essentially standard residual subspace $L_{\pi'}$ has as constant root 
    system the Levi subsystem ${R}_{L_{\pi'}}$ of the form 
    \begin{equation}
    {R}_{L_{\pi'}}=\uA_{k'_1-1}\times \dots \uA_{k'_T-1}\times \uX_{n'}
    \end{equation}
    where $N'=2n'$ is $\uX=\uC,\uD$, and $N'=2n'+1$ otherwise.
    We organise the coordinates of $c^\uX_{q'}$ first in strips beginning and ending 
    in the extremities of $q'$, according to a choice of an $m$-diagram ($m=0,1/2,1$ for 
    $\uX$ of type $\uD$,$\uB$, or $\uC$ respectively) corresponding to $q'$. The strips are the consecutive hooks 
    (starting from the outside of the $m$-diagram and working to the inside). Then  
    the center $c_{\pi'}$ of $L_{\pi'}$ has the form 
    \begin{equation}
    c_{\pi'}=(c^\uA_{\kappa'},c^\uX_{q'})
    \end{equation}
    where $c^\uA_{\kappa'}$ is the center of the type $\uA_M$-standard residual subspace 
    as in Case 1, and $c^\uX_{q'}$ is a residual point in $\bbR^{n'}$ of type $\uX_{n'}$ associated 
    to the distinct partition $q'$. We choose to represent this by the vector in $\bbR^{n'}$ with 
    coordinates (the representative in the $W(\uX_{n'})$-orbit is irrelevant because all 
    choices give an essentially standard pole space $L_{\pi'}$):
    \begin{equation}\label{eq:cX}
    c^\uX_{q'}=(\frac{q'_1-1}{2},\dots,\frac{1-q'_2}{2},\frac{q'_3-1}{2},\dots,\frac{1-q'_4}{2},\frac{q'_5-1}{2},\dots,\frac{q'_{S-1}-1}{2},\dots,\frac{1-q'_S}{2})
    \end{equation}
    if $S$ is even (e.g. if $\uX=\uD$ then $S$ is always even) else  
    \begin{equation}\label{eq:cXo}
    c^\uX_{q'}=(\frac{q'_1-1}{2},\dots,\frac{1-q'_2}{2},\frac{q'_3-1}{2},\dots,\frac{1-q'_4}{2},\frac{q'_5-1}{2},\dots,\frac{q'_{S}-1}{2},\dots,\frac{1}{2})
    \end{equation}
    if $S$ is odd and $\uX=\uB$, or else if $S$ is odd and $\uX=\uC$ then
    \begin{equation}\label{eq:cXoC}
    c^\uX_{q'}=(\frac{q'_1-1}{2},\dots,\frac{1-q'_2}{2},\frac{q'_3-1}{2},\dots,\frac{1-q'_4}{2},\frac{q'_5-1}{2},\dots,\frac{q'_{S}-1}{2},\dots,1)
    \end{equation}
    Recall that the $q'_i$ are all distinct here.
    Consider generic vector $x=(x_1,\dots,x_T)\in V^{L_{\pi'}}$. Then a generic element of $L_{\pi'}$ is of the form $c_{\pi'}+x$. 
    
    As in the type $\uA$-case, we induce $L_{\pi'}$ in the trivial way to ${R}$ by adding two parts of size one to $\pi'$ (or equivalently, adding 
    one part of size $1$ to $\kappa'$), giving a partition $\pi\in\Pi_{\uX,n+1}$. We again write $L_\pi=L_{\pi'}+\bbR\fw'$, 
    where $\fw'=e_0=(1,0,\dots,0)\in \bbR^{n+1}$. A generic vector $v\in L_\pi$ can be written in the form $v=c_{\pi}+x+\lambda\fw'$.
    
    Apart from a trivial factor $\frac{\lambda}{\lambda-1}$ (type $\textup{B}$) or $\frac{2\lambda}{2\lambda-1}$ (type $\textup{C}$), 
    the function $F_{L_\pi}$ now looks like:
    \begin{align}\label{F:X}
    F_{L_\pi}(v)=\prod_{j=1}^T\prod_{l=1}^{k_j}\frac{\lambda-x_j-\frac{k_j-2l+1}{2}}{\lambda-x_j-\frac{k_j-2l+1}{2}-1}
    &\frac{\lambda+x_j+\frac{k_j-2l+1}{2}}{\lambda+x_j+\frac{k_j-2l+1}{2}-1}\\
    \nonumber&\times \prod_{i=1}^{n'}\frac{\lambda-c^\uX_{q',i}}{\lambda-c^\uX_{q',i}-1}\frac{\lambda+c^\uX_{q',i}}{\lambda+c^\uX_{q',i}-1}
    \end{align}
    Hence for each $j$ we have two pole hyperplanes of $F_{L_\pi}$ in $L_\pi$ of the form $L^\pm_{\pi_j}$ given by 
    the equation $\gl=\pm x_j + \frac{k_j+1}{2}$. These are both \emph{residual} of the type $L_{\pi_j}$ with 
    $\pi_j$ the partition given by adding $1$ to the part $k'_j$ of $\kappa'$ in $\pi'$.  These poles are clearly simple, 
    hence $\textup{Ord}_{L_{\pi_i}^\pm}(\gO)=0$.
    
    Or else $\gl=\pm c^\uX_{q',i}+1$. This means that we increase 
    one of the parts $q'_j$ of $q'$ by $2$. This gives a residual subspace if the resulting 
    partition $q=(q_j)_j$ still has distinct parts. Otherwise this produces a quasi-residual subspace. 
    In both cases we denote this subspace by $L_{\pi_j}^q\subset L_{\pi}$. We note that the center of 
    these spaces $L_{\pi_j}^q$ is equal to the intersection point $p_{\pi,\ga_j}$ of $c_{\pi}+\bbR\fw'\subset L_\pi$ 
    with the corresponding pole hyperplane $\{{\hat\ga}_j=1\}$ (in both the residual and the quasi residual cases of this type).
    Hence we intersect these spaces $L_{\pi_j}^q$ in the second generation of $C_\uX$ in their centers. Observe that 
    $\textup{Ord}_{L_{\pi_i^q}}(\gO)=0$, since all the $q_i'$ are distinct. 
    
    We now analyse the second generation of $C_\uX$ in the residual spaces $L^\pm_{\pi_j}$. Both cases are equivalent, 
    and we will discuss only the case $L^+_{\pi_j}$. When we intersect $\{{\hat\ga}_j^{+}=1\}$ with the 
    line $c_{\pi'}+\bbR\fw'\subset L_\pi$ we find the point $p_{\ga_j}^+=(\frac{k_j+1}{2},c_{\pi'})\in L^+_{\pi_j}$.
    The center of $L^+_{\pi_j}$ is $c^+_{\pi_j}=p_{\ga_j}^+ -1/2(e_0+E_j)$ (compare with the type $\uA_n$ case).
    We consider the interval $I_{\ga_j}^+=[p_{\ga_j}^+,c^+_{\pi_j}]\subset L^+_{\pi_j}$, and the poles 
    of $F_{L^+_{\pi_j}}$ it can intersect. 
    
    Some non-constant coroots may vary by $1$ in this interval, but such coroots take integral values 
    at both extreme end points. Otherwise coroots may vary 
    at most $1/2$, and in such case the value at the end points 
    is half integral. In any case the pole hyperplanes of the form $H=\{\hat\gb=1\}$ 
    of a non-contant coroot ${\hat\gb}$ on $L^+_{\pi_j}$ either meet $I_{\ga_j}^+$ at $c^+_{\pi_j}$  - in 
    which case everything is the same as for $\textup{GL}_{n+1}$ -  or meets 
    $I_{\ga_j}^+$ at $p_{\ga_j}^+$. We claim that again, as in the case of $\textup{GL}_{n+1}$, 
    no pole spaces will be met as a hyperplane $H\subset L^+_{\pi_j}$ going through $p_{\ga_j}^+$. 
    Indeed, $H$ would be given either by the equation $x_j=x_i$ for some $i$ such that $k_i=k_j$
    or by the equation $x_j=0$ if $k_j=q_a$ for some odd $a$. It is well known 
    (see e.g. \cite{Sl}) that 
    the Plancherel density of  $d\nu'_{L_{\pi}}(\gl')$ has (double) zeroes at the hyperplanes $x_i=x_j$ 
    through the center $c_{\pi'}$ provided $k_i=k_j$, and at the hyperplane $x_j=0$ through the center 
    $c_{\pi'}$ unless the type $\uA$ strip of size $k_j$ can be glued into an $m$-diagram $D(q) $ ($m$ corresponding 
    to the type $\uX$ of course) corresponding to the extremities partition $q=q'=(q'_1,\dots,q'_S)$. 
    However, since $k_j=q_a$ and $\frac{q_a-1}{2}$ is an extremity of the $m$-diagram $D(q)$ 
    this is clearly \emph{not} the case, therefore in both cases we have zeroes in the Plancherel density 
    of $d\nu'_{L_{\pi}}(\gl')$ along these hyperplanes. We see (as in type $\uA$) that 
    no pole spaces are met at $p_{\ga_j}^+$. 
    
    The analysis for $L^-_{\pi_j}$ is entirely similar. 
    Observe that $P_{L^\pm_{\pi_j}}\subset \hat{R}_+\cup \hat{R}'$ in all cases (indeed, this is clear in the 
    cases $L^\pm_{\pi_j}$, and for $L^q_{\pi_j}$ where $\gl=\pm c^\uX_{q',i}+1$ is a constant, we always 
    have $\gl\geq 1$ as is clear from (\ref{eq:cX}), (\ref{eq:cXo}), (\ref{eq:cXoC}), and (\ref{F:X}).
    We have now shown Theorem \ref{thm:classmain} $(i), (ii), (iii), (iv)$, and $(v)$. The proof will now be 
    completed by the following Lemma: 
    \begin{lem}
    Theorem \ref{thm:classmain} $(vi)$ also holds for $C_\uX$ with $\uX$ one of the classical types 
    $\uX=\textup{B,C,D}$. In other words, if $(\gs,L)\in C_\uX$, where 
    $L$ is a residual pole space with $\textup{Ord}_L(\gO)=0$, and $(L_\pi,d)\in\textup{Std}(C_\uX)$ is a standard 
    datum corresponding to this segment, then $c_L\in \textup{Adm}(L)$. 
    \end{lem}
    \begin{proof}
    By Definition \ref{def:Adm} it suffices that for all standard pairs $(L_\pi,w=ud)$ representing $L$
    (i.e. $L=wL_\pi$) we have $\textup{Den}(L_\pi,w,\Sigma)_+\subset D_{L_\pi}$ and 
    $\textup{Den}(L_\pi,w,\Sigma')_+\subset D'_{L_\pi}$. By Corollary \ref{cor:cont} and 
    Theorem \ref{thm:tauadm} it is enough to show that 
    $\textup{Den}^\sim(L_\pi,w,\Sigma)_+\subset D_{L_\pi}$ for one choice of a representing standard 
    pair $(L_\pi,w)$ for $L$. 
    
    Let $\pi=(\gk,q)$, where we choose an $m$-diagram $D(q)$ representing $q$.  
    Since $\hat{P}_L\subset \hat{R}_+\cup \hat{R}'$, 
    all regular envelopes are $w$-good and thus 
    $\textup{Den}^\sim(L_\pi,w,\Sigma)_+=\textup{Den}^\sim(L_\pi,\Sigma)_+$ is independent of $w$.
    
    A regular envelope $H\supset L_\pi$ can be obtained as follows: In each of the coordinates 
    of the $j$-th nested hook of the $\lfloor(S+1)/2\rfloor$ hooks of $D(q)$, we add a coordinate $z_j$, 
    so as to obtain a strip $\fs_j:=(z_j+\frac{q_j-1}{2},\dots,z_j+\frac{1-q_{j+1}}{2})$.
    We do this for all hooks of $D(q)$ except possibly for a single hook, say the $j_0$-th, 
    in which we fix an arm starting from $m$ and running to the right extreme $\frac{q_j-1}{2}$, 
    or a leg starting from $m$ to the bottom extreme $\frac{q_{j+1}-1}{2}$, viewing this as 
    a type $\textup{\uX}$ regular $m$-diagram by itself, and 
    then adding the variable $z_{j_0}$ to the remaining boxes 
    of the hook, defining a strip $\fs_{j_0}$. Then the $j_0$-th hook corresponds to a strip of coordinates 
    in the parameterization of $H$, of the form  
    $(\frac{q_{j_0}-1}{2},\dots,m,\fs_{j_0})$ or to $(\fs_{j_0},-m,\dots,\frac{1-q_{{j_0}+1}}{2})$.
    Hence $L_\pi\subset H$ corresponds to taking the $z_j=0$ for all $j$, and 
    the center $c_\pi$ of $L_\pi$ corresponds to taking the variables $x_i=0$ and $z_j=0$ for all $i$ and $j$.
    The resulting space $H$ is a regular pole space which contains $L_\pi$. 
    We put $H$ in an essentially standard position $H=w_H(H_0)$ where $w^{-1}_H$ translates the coordinates of the fixed substrip 
    of the strip with index $j_0$ in $H$ to the extreme right. Here $H_0$ is obtained from $H$ by deleting the constant 
    substrip of the $j_0$-th strip, and moving it to the extreme right hand side.
    Thus if $x+1$ is the index of the first fixed coordinate of $H$, 
    and $y$ the index of the last fixed coordinate of $H$, then 
    $w_H^{-1}=(x+1,x+2,\dots,n)^{n-y}$. Then $w_H\in W(\hat\Delta(L_\pi))$ (compare with Proposition \ref{prop:essstdhull}
    and Corollary \ref{rem:std}).
    
    Let us now consider the denominators in $\textup{Den}^\sim(L_\pi,\Sigma)_+$ of the form $x_p- x_q+a$ for $p<q$ and 
    some constants $a$. 
    Let $H_0=w_H^{-1}(H)$ where $H$ is a regular envelope of $L_\pi$ as above. In view of Proposition \ref{lem:incl}, 
    on $H_0$ such denominators are of the form $x_p- x_q+c+1$ with $c$ is running over the highest weights  
    of the $\mathfrak{sl}_2$-submodule corresponding to the restricted coroots on $H_0$ which differ from $x_p-x_q$ by a 
    constant. This submodule associated to the product $\ft_p\times \ft_q$ of two type $\uA$-strips of length $k_p$ and $k_q$ respectively 
    is easily seen 
    to be equivalent to the tensor product $V_{k_p}\otimes V_{k_q}$ (where $V_{n}$ denotes the irreducible 
    $\mathfrak{sl}_2$-module of dimension $n$).
    Therefore $c\in \{\frac{|\gk_p-\gk_q|}{2},\dots,\frac{|\gk_p+\gk_q|}{2}\}$ in this case. Pulling back with $w_H^{-1}$ and 
    restricting to $L$ does not change this result, because $(w_H^{-1})^*$ permutes the set of positive coroots which are 
    nonconstant on $L_\pi$ (since $w_H\in W(\hat\Delta(L_\pi))$). 
    
    For denominators of the form $x_p+x_q+a$ the analysis is the same.  
    The denominators of these types in $\textup{Den}^\sim(L_\pi,\Sigma)_+$ are therefore in $D_{L_\pi}$, as required. 
    
    Now let us look at denominators of the form $x_i+c+1$ in $\textup{Den}^\sim(L_\pi,\Sigma)_+$ for some constant $c$.                                                                                                                    
    In a regular envelope $H$ of $L_\pi$ these arise from denominators $x_i\pm z_j+c+1$ on $H_0$
    corresponding to a product $\ft_i\times \fs_j$, and the analysis of 
    the $c$ corresponding to denominators of this kind is similar as in the ``two type $\textup{A}$-strips'' case above, 
    except for the remark that the 
    strips $\fs_j$ given by $(z_j+\frac{q_j-1}{2},\dots,z_j+\frac{1-q_{j+1}}{2})$ are not balanced at the center of $H_0$. 
    Consider the functions $t_a\pm t_b+1$ where $a$ varies in the $i$-th type $\textup{A}$-strip $\ft_i$ 
    (with $0\leq i\leq \fT$) and $b$ varies in the $j$-th ``type $\textup{X}$''-strip $\fs_j$ (with $1\leq j\leq S$).  
    These restrict (after pulling back to $H$ with $w_H^{-1}$) to denominators of the form $x_i\pm z_j+c+1$ on $H$.
    The constants $c$ that appear can conveniently be arranged in a rectangle whose boxes are parameterized by tuples 
    $(a,b)$ corresponding to the Cartesian product of the two strips, filled with the corresponding values 
    $t_a\pm t_b$. For $\textup{Den}^\sim(L_\pi,\Sigma)_+$ we see that the multiplicities of such denominators on $H$ have 
    downward jump from $c$ to $c+1$ starting from the rightmost 
    of the two diagonals which go through the top left or through the bottom right box onward. 
    Hence the corresponding denominators of this type on $H$ have constants equal to $1$ plus 
    at least $c_i^{H,j}:=\min\{ c_i^{H,j,+},   c_i^{H,j,-}\}$ where 
    $c_i^{H,j,+}=\max\{\frac{q_j-\gk_i}{2},\frac{\gk_i-q_{j+1}}{2}\}$ and 
    $c_i^{H,j,-}=\max\{\frac{q_{j+1}-\gk_i}{2},\frac{\gk_i-q_j}{2}\}$ repectively.
    All these denominators on $H$ restrict to functions of the form $x_i+c+1$ on $L_\pi$. 
    In type  $\textup{B}$ also the coroots $t_a$ (with $a$ in $\ft_i$) have such restrictions on $L_\pi$, 
    which is why we include in this case the additional denominator on $H$ of the form $x_i+{}^\textup{B}c_i^{H}+1$, 
    where ${}^\textup{B}c_i^{H}=\frac{\gk_i-1}{2}$.
    We need to take the intersection of such sets of denominators over all possible $H$ as 
    constructed above. In other words if we put $c_i=\max\{c_i^H\mid L_\pi\subset H\}$ with 
    $c_i^H:=\min\{\{ c_i^{H,j}\mid 1\leq j\leq S\}\cup \{{}^\textup{X}c_i^{H}\mid \textup{X}=\textup{B}\}\}$, 
    then admissiblity of $c_{L_\pi}$ with respect to these denominators amounts to checking 
    that for all $i$, we have $c_i\geq 0$.
    Now $c_i^{H,j}:=\min\{ \max\{\frac{q_j-\gk_i}{2},\frac{\gk_i-q_{j+1}}{2}\}, \max\{\frac{q_{j+1}-\gk_i}{2},\frac{\gk_i-q_j}{2}\}\}<0$ 
    iff $\gk_i$ is strictly between $q_i$ and $q_{j+1}$. 
    
    However, we claim that for every given $i$, we can choose a regular envelope $H$ such that $c_i^{H}\geq 0$.
    Namely, first consider an $H'$ which is of type $\textup{A}$, which means that we chose an $m$-diagram $q_L$ for $q$ 
    and added an indeterminate to all nested hooks of the $m$-diagram $q_L$, thus defining a envelope $H'$, with 
    $\textup{A}$-strips $\ft_i$ and $\textup{A}$-strips $\fs'_j$. We choose this $m$-diagram in a precise way using the 
    procedure \cite[Lemma 5.20; Figure 5]{Sl2} such that the $a$-th nested hook has  
    extremities of absolute value $\{\frac{q'_a-1}{2},\frac{q'_{a+1}-1}{2}\}$ with $q'_1\geq q'_2\geq q'_3\dots\geq q'_{S}$.
    Since $L_\pi$ is residual, all these inequalities are strict, and this $m$ diagram is well defined.
    \footnote{For any pole space in the first and second generation of the cascade, in this sequence of inequalities of extremities 
    there is \emph{at most one} equality (i.e. if $L_\pi$ is the standard form of a non-residual second generation pole space of 
    type $L_{\pi_j}^q$). In such cases, this $m$-diagram is still well defined.}
    Now suppose that $c^{H'}_i<0$. 
    Because of the choice of the $m$-diagram, there is at most one hook of $L_q$ 
    such that $\gk_i$ is strictly in between the extremities $q'_j$ and $q'_{j+1}$
    of this hook. If this occurs then we replace $\fs'_j$ by a pair 
    of the form $(c_{q,j},\fs_j)=(\frac{q'_j-1}{2},\dots,m,\fs_{j})$ or to $(\fs_j,c_{q,j})=(\fs_j,-m,\dots,\frac{1-q'_{j+1}}{2})$ (as described 
    before), such that the largest absolute coordinate value in $c_{q,j}$ is larger than $\frac{\gk_i-1}{2}$. 
    This a new regular envelope $H$ of $L$. 
    Finally we let $w_H\in    W(\hat{\Delta}(L_0))$ be such that $H=w_H(H_0)$, with $H_0$ in essentially standard position, 
    with the same type $\textup{A}$-strips as $H$ except for the substrip $\fs_j$ (of $\fs'_j$) 
    corresponding to $z_j$,  and a regular type $\textup{X}$-strip $c_{q,j}$. 
    
    Clearly we have $c_i^{H,k}=c_i^{H',k}\geq 0$ if $k\not=j$. Also clearly $c_i^{H,j}\geq 0$, by construction of $H$.
    But $H$ also has the additional type $\textup{X}$-strip $c_{q,j}$, corresponding to denominators 
    of the form $x_i+c+1$, which are the denominators on the (smaller rank) regular type $\textup{X}$-standard residual 
    pole space $H_i:=\ft_i\times(c_{q,j})=(\ft_i,c_{q,j})$, whose center is at $x_i=0$. 
    With $H_i$ being a regular standard residual pole space, these denominators are in $D_{H_i}$, which means that 
    the occurring $c$ are $\geq 0$ by Proposition \ref{lem:incl}. This proves that $c_i^{H}\geq 0$, as claimed. 
    \end{proof}
    This result completes the proof of Theorem \ref{thm:classmain} for the types $\uB$, $\uC$ and $\uD$ as well. 
    \end{proof}
   
   \section{Appendix C: The pseudocode for enveloping denominators}\label{a:EnvDen}
Let $(L_0,w)$ be a standard pair in $\textup{SP}(C)$. Write $w = ud$, $u\in W'$ and $d$ minimal in the coset $W'w$, so that $(L_0,d)\in \textup{Std}(C)$. The computation of the enveloping denominators 
$\textup{Den}^\sim(L_0,w,\Sigma)$ and $\textup{Den}^\sim(L_0,w,\Sigma')$ break down into two parts: (a)
compute the set of all $w$-good regular envelopes of $L_0$ and (b) computing the denominators 
$\textup{Den}(H,w,\Sigma)$ and $\textup{Den}(H,w,\Sigma')$ for each $w$-regular envelope $H$.

\subsection{On $w$-good regular envelopes}
Given any pole set $L$, the set $\hat{P}_L$ uniquely determines $L$. So our task is to find all $\hat{P}_H\subset \hat{P}_{L_0}$ such that $H = \cap_{\hat{\alpha} \in \hat{P}_H}\{\hat{\alpha} = 1\}\supset L_0$ and $H$ is regular\footnote{In general, checking the whole power set of $\hat{P}_L$ is impractical as the cardinality of $\hat{P}_L$ can be large.}. 
For that, we need the auxiliary procedures $\mathtt{RegComps}$ and $\mathtt{MaxSP}$.
\begin{defn}
Given two coroots $\hat{\alpha},\hat{\beta}\in \hat{P}_{L_0}$, we say that $\hat{\alpha}$ is a neighbor of $\hat{\beta}$ if $\hat{\alpha}-\hat{\beta} \in \hat{R}\cup\{0\}$, and we write $\hat{\alpha}\sim'_{\textup{ngh}}\hat{\beta}$. The relation $\sim'_{\textup{ngh}}$ is symmetric and reflexive and we let $\sim_{\textup{ngh}}$ denote the transitive closure of $\sim'_{\textup{ngh}}$.
\end{defn}
The procedure $\mathtt{RegComps}$ receives $\hat{P}_{L_0}$ as input and returns $\{\mathtt{C}_1,\ldots,\mathtt{C}_p\}$, the partition of $\hat{P}_{L_0}$ under $\sim_{\textup{ngh}}$. The  $\mathtt{MaxSP}$ procedure takes as input $(\hat{\alpha},\mathtt{C})$ with $\hat{\alpha}\in\mathtt{C}\subset\hat{P}_{L_0}$ and $\mathtt{C}$ a $\sim_{\textup{ngh}}$-equivalence class and returns a set $\{\mathtt{M}_1,\ldots,\mathtt{M}_q\}$ of subsets of $\mathtt{C}$ with the properties that
\begin{itemize}
\item for all $j$, $\hat{\alpha}\in\mathtt{M}_j$,  the affine space $M_j = \cap_{\hat{\beta}\in \mathtt{M}_j}\{\hat\beta = 1\}$ is regular and $\mathtt{M}_j$ is maximal for these properties,
\item for all $j\neq k$ and any pair $(\hat{\beta}_j,\hat{\beta}_k) \in\mathtt{M}_j\setminus\{\hat\alpha\}\times\mathtt{M}_k\setminus\{\hat\alpha\}$, we have $\hat{\beta}_j \not\sim'_{\textup{ngh}}\hat{\beta}_k$.
\end{itemize}
With these procedures in place, we compute $\mathtt{RegEnv}(L_0) = \{H\mid H\supset L_0, H\textup{ regular}\}$ by following steps:
\begin{enumerate}
\item Compute $\mathtt{RegComps}(\hat{P}_{L_0}) = \{\mathtt{C}_1,\ldots,\mathtt{C}_p\}$.
\item For each $\mathtt{C}_j=\{\hat{\alpha}_{j,1},\ldots,\hat{\alpha}_{j,n_j}\}$, compute $\mathtt{MSp}_j = \cup_{k=1}^{n_j} \mathtt{MaxSp}(\hat{\alpha}_{j,k},\mathtt{C}_j)$.
\item Compute  $\mathtt{RegEnv}(L_0)$ in the following way: for each $j$, $\mathtt{MSp}_j = \{\mathtt{M}_{j,1},\ldots,\mathtt{M}_{j,m_j}\}$ is a set of subsets of $\hat{P}_{L_0}$. Then define $\mathtt{PRegEnv}(L_0)$ as the image of the function $\prod_{j=1}^p\mathtt{MSp}_j \to \textup{Power Set}(\hat{P}_{L_0})$ defined by sending  \[(\mathtt{M}_{1,i_1},\ldots,\mathtt{M}_{p,i_p})\mapsto \mathtt{M}_{1,i_1}\cup\cdots\cup\mathtt{M}_{p,i_p}.\]
We then define $\mathtt{RegEnv}(L_0)\cong \mathtt{PRegEnv}(L_0)$ via identifying $H\leftrightarrow \hat{P}_H$, with $\hat{P}_H\in \mathtt{PRegEnv}(L_0)$.
\end{enumerate}
Once we have $\mathtt{RegEnv}(L_0)$, we compute $\mathtt{GRegEnv}(L_0,w)$, the set of all $w$-good regular envelopes by testing, for each $H\in\mathtt{RegEnv}(L_0)$ 
if the condition $d(\hat{P}_H)\subseteq \hat{R}_+\cup \hat{R}'$, where $w = ud$ and $d$ minimal in the coset $W'w$.

\subsection{On enveloping denominators} For any standard pole space $M_0$ with standard parabolic datum $\mathtt{Pd}_{M_0} = (J_{M_0},\gamma_{M_0})$, recall that $\gamma_{M_0} \in \bbQ^{|J_{M_0}|}$ is a tuple of rational numbers satisfying $\gamma_{M_0} = (\gamma_{M_0,\hat{\alpha}})_{\hat{\alpha}\in J_{M_0}}$. Here, for each $\hat{\alpha}\in J_{M_0}$, we have
\[
\gamma_{M_0,\hat{\alpha}} = \langle \hat{\alpha},\fw_{\hat{\alpha}}\rangle,
\]
where $\fw_{\hat\alpha}$ is the fundamental weight dual to the simple coroot $\hat{\alpha}$. We define the procedure $\mathtt{stdWDD}(M_0)$ that returns a generic element $\lambda_{M_0}\in M_0$ given by
\begin{equation}\label{eq:stdWDD}
\lambda_{M_0} = \sum_{\hat\alpha\in J_{M_0}}\gamma_{M_0,\hat{\alpha}}\fw_{\hat\alpha} + \sum_{\hat\beta\notin J_{M_0}}x_{\hat\beta}\fw_{\hat\beta}.
\end{equation}
Here, $x_{\hat\beta}$ with $\hat\beta\notin J_{M_0}$ are to be interpreted as indeterminates. These will be the variables of the affine linear expressions that will appear in the denominator computations when we pair $\langle \hat\alpha,\lambda_{M_0}\rangle$ with coroots $\hat\alpha$ that are non-constant on $M_0$.

Given our pair $(L_0,d)\in\textup{Std}(C)$, let $\mathtt{Pd}_{L_0} = (J_{L_0},\gamma_{L_0})$ be the standard parabolic datum of $L_0$. Let also $W_{J_{L_0}}$ denote the parabolic subgroup of $W$ generated by the reflections in $J_{L_0}$. We set, once and for all
\[
\lambda_{L_0} = \mathtt{stdWDD}(L_0).
\]
Having computed $\mathtt{GRegEnv}(L_0,w)$,  we determine the set $\textup{Den}^{\sim}(L_0,w,\Sigma)$ by the following  steps:
\begin{itemize}
\item[(1)] For each $H\in \mathtt{GRegEnv}(L_0,w)$, compute $(w_H,H_0)$ using the $\mathtt{stdData}$ procedure (see Remark \ref{rem:AlgIniRems}), so that $H = w_H(H_0)$. Here, we can guarantee that $w_H\in W_{J_{L_0}}$ (see Proposition \ref{prop:essstdhull}).
\item[(2)] Compute $\lambda_{H_0} =  \mathtt{stdWDD}(H_0)$.
\item[(3)] Compute the list $D = [\langle \hat\alpha,\lambda_{H_0}\rangle + 1]_{\hat\alpha\in\hat{R}_+}$ for some fixed ordering of $\hat{R}_+$.
\item[(4)] Determine the set $\textup{Den}(H_0,1,\Sigma)$ in the following way: in view of Proposition \ref{lem:incl}, we interpret the list $D$ of (3) in terms of weights of the algebra $\hat\fq^{H_0}$ on $\hat{\fu}^{H_0}$. Then, $\textup{Den}(H_0,1,\Sigma)$ is given by the highest weights, which can be computed from (3) by analyzing the downward jumps of multiplicities.
\item[(5)] Compute $\textup{Den}(H,w,\Sigma)|_{L_0}$ by substitution of the equation (on the variables  $x_{\hat\beta}$) obtained by solving $w_H(\lambda_{H_0})=\lambda_{L_0}$.
\item[(6)] Following Definition \ref{def:hull}, we disregard the constants in $\textup{Den}(H,w,\Sigma)|_{L_0}$ and intersect over all $H \in \mathtt{GRegEnv}(L_0,w)$ to obtain $\textup{Den}^{\sim}(L_0,w,\Sigma)$.
\end{itemize}

Similarly, we determine the set $\textup{Den}^{\sim}(L_0,w,\Sigma')$ by the following  steps:
\begin{itemize}
\item[(1)] Compute $(w_H,H_0)$ using the $\mathtt{stdData}$ procedure so that $H = w_H(H_0)$ and $w_H\in W_{J_{L_0}}$.
\item[(2)] Compute $d_H$, the minimal coset representative in $W' w_H$ and $\hat{R}_{d_H} = \hat{R}(d_H) \cup d_H^{-1}(\hat{R}'_+)$.
\item[(3)] Compute $\lambda_{H_0} =  \mathtt{stdWDD}(H_0)$.
\item[(4)] Compute the ordered list $[\langle \hat\alpha,\lambda_{H_0}\rangle]_{\hat\alpha\in\hat{R}_{d_H}}$.
\item[(5)] By viewing $\textup{Den}(H_0,1,\Sigma')$ as lowest weights as in Proposition \ref{lem:incl}, we compute $\textup{Den}(H_0,1,\Sigma')$ from the list in (3) by analyzing the upward jumps of multiplicities in that list.
\item[(6)] Compute $\textup{Den}(H,w,\Sigma')|_{L_0}$ by substitution of the equation (on the variables  $x_{\hat\beta}$) obtained by solving $w_H(\lambda_{H_0})=\lambda_{L_0}$.
\item[(7)] Disregard the constant expressions in $\textup{Den}(H,w,\Sigma')|_{L_0}$ and intersect over all $H \in \mathtt{GRegEnv}(L_0,w)$ to obtain $\textup{Den}^{\sim}(L_0,w,\Sigma')$.
\end{itemize}
    
    \section{Appendix D: The special line $\textup{E}_7(a4)$ of $\textup{E}_8$}\label{a:Special}
    In this section we discuss a special segment $(\sigma_{L_{sp}},L_{sp})\in C$ inside the pole space $L_{sp}=s_8s_7(L_{sp,0})$ 
    of type $\textup{E}_7(a4)$\index{$L_{sp,0}=\textup{E}_7(a4)$!$L_{sp}$ the special pole line of $E_8$}. 
    The treatment of this segment is more complicated than any other segment in $C$, since it cannot be treated by the theory of admissible sets as discussed in \ref{s:Adm}. Let us explain the problem in more detail.
    
    In the context of this section we will write $L:=L_{sp}$ and $L_0:=L_{sp,0}$. 
    Writing vectors in basis of fundamental weights (in the Bourbaki ordering), a general point in $L$ is of the form 
    $(1,0,0,1,0,1,x,-1-x)$. The center of $L$ is $(1,0,0,1,0,1,-11/2,9/2)$. 
    Note that $L$ is residual, and this special segment is of the form $([p^N_L,c_L],L_{sp})$ with  $p^N_L=\sigma_{N,0}\cap L$, 
    where $N$ is the induced (first generation) two dimensional residual pole space of type $\textup{E}_6(a3)$ of $\textup{E}_8$, 
    with center initial segment $\sigma_{N,0}=[p_{N,\infty},c_N]$. Writing vectors in basis of fundamental weights (in the Bourbaki 
    ordering), we have  $c_N=(1,0,0,1,0,1,-4,0)$ and $p_{N,\infty}= (1,0,0,1,0,1,-4,X)$ with $X$ very large. 
    Hence $p^N_L=(1,0,0,1,0,1,-4,3)$, hence we are required to move the base point along this segment 
    $\sigma_L=[p^N_L,c_L]$ spawned by the intersection of $L$ with $\sigma_{N,0}$. 
    What makes this segment $\sigma_L$  special is the fact that $\textup{Ord}_L(\Omega)=1$ while $\sigma_L\not\subset\textup{Adm}(N)$. 
    The only troublesome denominator is the denominator $x+5$ of $\Sigma_W(\phi)$ on $N$ using the coordinates 
    $\gl_{(x,y)}=(1,0,0,1,0,1,x,y)$ for a general point in $N$. The critical strip for this denominator is at $\textup{Re}(x)\in (-5,-4)$, 
    which is entirely inside $\sigma_L+iV^L$. 
    We note that ($\sigma_L,L)$ is the unique segment in $C$ which belongs to a pole space of positive order, and 
    for which $\sigma_L\not\subset\textup{Adm}(N)$
    for the parent pole space $N\supset L$ responsible for the initial point of $\sigma_L$. 
    \begin{lem}\label{lem:specialholo}
    Let $d=s_7s_8$, and put $\gl_x=(1,0,0,1,0,0,1,x)\in L_0$. 
    We have that $\Sigma_W(\phi)|_L$, $\rho(\hat\ga_7+5)\Sigma_W(\phi)|_N$ and $\Sigma'_{W'}(\psi)|_N$ are holomorphic 
    at the points $d\gl_x=(1,0,0,1,0,1,x,-1-x)\in L\subset N$ for $\textup{Re}(d\gl_x)\in\sigma_L$ (i.e.
     $\textup{Re}(x)\in[-\frac{11}{2},-4]$). 
    Likewise 
     $\Sigma_W(\phi)|_{L_0}$, $\rho(\hat\ga_8+5)\Sigma_W(\phi)|_{N^0}$ and $\Sigma'_{W'd}(\psi)|_{N^0}$ are holomorphic 
     at $\gl_x\in L_0\subset N^0$ for $\textup{Re}(\gl_x)\in d^{-1}(\sigma_L)$.
    \end{lem}
    \begin{proof}
     Since $\sigma_L\subset \textup{Adm}(L)$, the first statement follows directly from Theorem \ref{thm:main}. 
    The second and third statements follows from the computation of $\textup{Den}^\sim(N,d,\Sigma')$ and 
    $\textup{Den}^\sim(N,d,\Sigma')$. The next statements follows from the first three by Remark \ref{rem:warn},  
    together with the observation that, with $\gl_{(x,y)}=(1,0,0,1,0,1,x,y)\in N$, 
    \begin{equation}
    \frac{r(\gl_{(x,y)})}{r(d^{-1}(\gl_{(x,y)}))}|_N
    =\frac{\rho(\hat\ga_8)\rho(\hat\ga_7+\hat\ga_8)}{\rho(\hat\ga_8+1)\rho(\hat\ga_7+\hat\ga_8+1)}|_N
    =\frac{\rho(y)\rho(x+y)}{\rho(y+1)\rho(x+y+1)}|_N
    \end{equation}
    which is invertibly holomorphic on $L=\{x+y+1=0\}$ for $\textup{Re}(x)\in [-\frac{11}{2},-4]$. 
    \end{proof}
    By the above, in order to deal with the segment $(\gs_L,L)\in C$, we just need to verify that  
    the meromorphic integrand $\partial_{n'}(\Sigma_W(\phi)\Sigma'_{W'}(\psi))|_L$ has no poles 
    at $d\gl_x$ if $\rho(x+5)=0$ (with $d=s_8s_7$), with $n'$ the normal of $L$ in $N$. 
    Equivalently we need to 
    show that: 
    \begin{equation}\label{eq:intL0}
    \partial_n(\Sigma_W(\phi)\Sigma'_{W'd}(\psi))|_{L_0}
    \end{equation}
    (where $n$ is the normal of $L_0\subset N^0:=s_7s_8(N)$) has no poles at 
    $(1,0,0,1,0,0,1,x)\in L_0$ if $\rho(x+5)=0$. 
    By Lemma \ref{lem:specialholo} it suffices to prove that 
    \begin{equation}
    \partial_n(\Sigma_W(\phi))|_{L_0}
    \end{equation}
    is holomorphic at $(1,0,0,1,0,0,1,x)\in L_0$ if $\rho(x+5)=0$. 
    The normal $n$ can be taken equal to $n=\ga_h-2(\ga_7+\ga_8)$, where 
    $\ga_h$ is the highest root of $\textup{E}_8$.
    \subsection{Computation of the derivative}
    Let us consider in detail the expression for $\Sigma$ on $L_0$, with $L_0$ the one dimensional pole space of type $\textup{E}_7(a4)$ 
    in $\textup{E}_8$. The general point for $L_0$ has the form $\gl=(1,0,0,1,0,0,1,x)$.  
    Let $F_0=\Delta_{L_0}=\{\ga_2,\ga_3,\ga_5,\ga_6\}$ be the base of the standard Levi 
    subgroup fixing $L_0$ pointwise
    \index{$L_{sp,0}=\textup{E}_7(a4)$!$F_0$ simple roots orthogonal to $L_{sp,0}$}. 
    
    The Weyl group of type $E_8$ contains $-1$, and using that we 
    have the decomposition 
    \begin{equation}
        \hat{R}_+:=\hat{R}(w)\sqcup \hat{R}(-w)
    \end{equation}
    We rewrite $\Sigma_W(\phi)$ as 
    \begin{equation}
        \Sigma_W(\phi)(\gl):=\sum_{w\in W}c(w\gl)\frac{r(\gl)}{r(-w\gl)}\phi(w\gl)
    \end{equation}
    Now recall that 
    \begin{equation}
        c(w\gl):=\prod_{\ga\in \hat{R}(w)}(\frac{{\hat\ga}-1}{{\hat\ga}})
        (\prod_{\gb\in \hat{R}(-w)}\frac{{\hat\gb}+1}{{\hat\gb}})(\gl)
    \end{equation}
    Now let us take $w$ minimal in $W/W_{F_0}$, and $\sigma\in W_{F_0}$. Then 
    \begin{equation}\label{eq:InversionMinCoset}
        \hat{R}(w\sigma)=\hat{R}(\sigma)\sqcup \sigma^{-1}\hat{R}(w)
    \end{equation}
    Therefore we decompose $\hat{R}_+$ in the following way:
    \begin{equation}
        \hat{R}_+=\sigma^{-1}\hat{R}(w)\sqcup \sigma^{-1}\hat{R}(-w)\backslash \hat{R}_+(F_0)\sqcup \hat{R}_+(F_0)
    \end{equation}
    where we remark the union of the first two sets is $W_{F_0}$-stable. 
    So we can write:
    \begin{equation}
        c(w\sigma\gl):=\prod_{\ga\in \hat{R}(w)}(\frac{{\hat\ga}-1}{{\hat\ga}})
        \prod_{\gb\in \hat{R}(-w)\backslash \hat{R}_+(F_0)}(\frac{{\hat\gb}+1}{{\hat\gb}})
        \prod_{\gc\in \hat{R}_+(F_0)}(\frac{{\hat\gc}+1}{{\hat\gc}})(\sigma\gl)
    \end{equation}
    We note that the coroots which are constant equal to $0$ on $L$ are contained is $\hat{R}_+(F_0)$,  
    In the computation below we will use the following notations: $\phi^w(\gl):=\phi(w\gl)$,  $\rho^*(x):=\frac{\rho(x)}{x}$, 
    $r_{F_0}(\gl):=\prod_{\gc\in \hat{R}_+(F_0)}\rho({\hat\gc}(\gl))$,
    $r^*_{F_0}(\gl):=\prod_{\gc\in \hat{R}_+(F_0)}\rho^*({\hat\gc}(\gl))$, 
    $r_{F_0}^{(-1)}(\gl):=r_{F_0}(-\gl)$ and 
    $r_{F_0}^{*,(-1)}(\gl):=\prod_{\gc\in \hat{R}_+(F_0)}\rho^*({\hat\gc}(\gl)+1)$, 
    and finally $\Phi^w=\frac{\phi^w}{r^{*,(-1)}_{F_0}}$. 
    The partial sums over the left cosets of $W_{F_0}$\index{$W_{F_0}$}
    regularize these apparent poles of $\Sigma$ on $L$, which becomes apparent by writing:
    \begin{equation}
         \Sigma_W(\phi)(\gl)=\sum_{w\in W^{F_0}}\sum_{\sigma\in W_{F_0}}c(w\sigma\gl)\frac{r(\gl)}{r(-w\sigma\gl)}\phi(w\sigma\gl)=
         \sum_{w\in W^{F_0}}A_w(\gl)
    \end{equation}
    with or each $w\in W^{F_0}$\index{$W_{F_0}$!$W_{F_0}$, minimal length left $W_{F_0}$-coset representatives}:
    \begin{align}\label{eq:Aw}
        &A_w(\gl):=\prod_{\gamma\in \hat{R}_+(F_0)}(\frac{\rho({\hat\gc})}{{\hat\gc}})(\gl)\\
        \nonumber &\sum_{\sigma\in W_{F_0}}\epsilon(\sigma)
        \prod_{\substack{\ga\in \hat{R}(-w)\backslash \hat{R}_+(F_0)\\
        \gb\in \hat{R}(w)\\
        \gc\in \hat{R}_+(F_0)}}
        (\frac{{\hat\ga}+1}{{\hat\ga}})
        (\frac{{\hat\gb}-1}{{\hat\gb}})({\hat\gc}+1)(\frac{\rho({\hat\ga})}{\rho({\hat\ga}+1)})
        (\frac{\phi^w}{r^{(-1)}_{F_0}})(\sigma\gl)=\\
        \nonumber &r^*_{F_0}(\gl) 
        \sum_{\sigma\in W_{F_0}}\epsilon(\sigma)
        \prod_{\substack{\ga\in \hat{R}(-w)\backslash \hat{R}_+(F_0)\\
        \gb\in \hat{R}(w)}}
        (\frac{{\hat\gb}-1}{{\hat\gb}})(\frac{\rho^*({\hat\ga})}{\rho^*({\hat\ga}+1)})
        (\frac{\phi^w}{r^{*,(-1)}_{F_0}})(\sigma\gl)=\\
        \nonumber &r^*_{F_0}(\gl) 
        \sum_{\sigma\in W_{F_0}}\epsilon(\sigma)
        \prod_{\substack{\ga\in \hat{R}(-w)\backslash \hat{R}_+(F_0)\\
        \gb\in \hat{R}(w)}}
        (\frac{{\hat\gb}-1}{{\hat\gb}})(\frac{\rho^*({\hat\ga})}{\rho^*({\hat\ga}+1)})
        \Phi^w(\sigma\gl)
    \end{align}
    Here we used the remark that (with $\sigma_{F_0}\in W_{F_0}$ is the longest 
    element of $W_{F_0}$):
    \begin{align}
         \frac{r(\gl)}{r(-w\sigma\gl)}
         \nonumber &=\prod_{\substack{\ga\in \sigma^{-1}(\hat{R}(-w)\backslash \hat{R}_+(F_0))\\
         \gc\in \hat{R}_+(F_0)\backslash \hat{R}(\sigma)}}(\frac{\rho({\hat\ga})}{\rho({\hat\ga}+1)})
         (\frac{\rho({\hat\gc})}{\rho({\hat\gc}+1)})(\gl)\\
         \nonumber &=\prod_{\ga\in \hat{R}(-w)\backslash \hat{R}_+(F_0)}(\frac{\rho({\hat\ga})}{\rho({\hat\ga}+1)})(\sigma\gl)
        \frac{r(\gl)}{r(\sigma\sigma_{F_0}\gl)}\\
        \nonumber &=\prod_{\ga\in \hat{R}(-w)\backslash \hat{R}_+(F_0)}(\frac{\rho({\hat\ga})}{\rho({\hat\ga}+1)})(\sigma\gl)
        \frac{r_{F_0}(\gl)}{r^{(-1)}_{F_0}(\sigma\gl)}\\
        \nonumber &=r_{F_0}(\gl)\prod_{\substack{\ga\in \sigma^{-1}(\hat{R}(-w)\backslash \hat{R}_+(F_0))\\
         \gc\in \hat{R}_+(F_0)}}(\frac{\rho({\hat\ga})}{\rho({\hat\ga}+1)\rho({\hat\gc}+1)})(\sigma\gl)
    \end{align}
    The main point in this manipulation is that $r^{*,(-1)}_{F_0}$ is holomorphic 
    and non-vanishing on $L_0$. This remark shows that for each $w\in W^{F_0}$ the 
    expression $A_w(\gl)$ is meromorphic on $L_0$, since the alternating sum over 
    $W_{F_0}$ is divisible by $\prod_{\gc\in \hat{R}_+(F_0)}{\hat\gc}$. 
    Note that the operator $\phi\to A_w(\gl)$ (with $\gl\in L_0$) has support at $w\gl$, and 
    this is the characterising property of the summands $A_w(\gl)$ of $\Sigma_W(\phi)|_L$.
    We concentrate now the individual $A_w$, and consider $\Phi^w$ as an  
    indeterminate meromorphic function which is holomorphic in a open neighborhood of $L_0$. 
    
    Using Weyl's trick for taking the limit for $\gl$ to $L_0$ we see for a known $c>0$, 
    we have for $\gl\in L_0$:
    \begin{align}
        A_w(\gl)&=cr_{F_0}(\gl)\partial_{\ga_2}\partial_{\ga_3}
        \partial_{\ga_5}\partial_{\ga_6}\partial_{\ga_5+\ga_6}\\
        \nonumber &\left(\prod_{\substack{\ga\in \hat{R}(-w)\backslash \hat{R}_+(F_0)\\
        \gb\in \hat{R}(w)}}
        (\frac{{\hat\gb}-1}{{\hat\gb}})
        (\frac{\rho^*({\hat\ga})}{\rho^*({\hat\ga}+1)})\Phi^w\right)(\gl)
    \end{align}
    \begin{lem}\label{lem:avn}
    Let $n=\ga_h-2(\ga_7+\ga_8)$ be the normal vector of $L_0\subset N^0$. 
    Let $n_0=\frac{1}{24}\Sigma_{\gs\in W_{F_0}}\gs(n)=n-2\go$ where $\go=\frac{1}{3}(\ga_5+\ga_6)$
    is a fundamental weight of the Levi root subsystem with basis $\{\ga_5,\ga_6\}$. 
    Then $\partial_{n-n_0}( \Sigma_W(\phi))|_{L_0}$ is holomorphic. 
    \end{lem}
    \begin{proof}
    From the cocycle relation (\ref{eq:trans}) we see that $r_{F_0}^{-1}\Sigma_W(\phi)$ is $W_{F_0}$-invariant. 
    Hence $\partial_{n-n_0}(r_{F_0}^{-1}\Sigma_W(\phi))|_{L_0}=0$. Using Lemma \ref{lem:specialholo} 
    and the remark that $r^{-1}_{F_0}$ is holomorphic and non-vanishing on $L_0$, we obtain the result. 
    \end{proof}
    Recall from Lemma \ref{lem:specialholo} that the denominator $\frac{1}{\rho(x+5)}$ does not appear in 
    $A_w|_{L_0}$ (where as always we put $\gl=\gl_x=(1,0,0,1,0,0,1,x)\in L_0$). However, we need to analize 
    the derivative $\partial_n(A_w)|_{L_0}$ with $n$ the normal $n=\ga_h-2(\ga_7+\ga_8)$ of $L_0\in N^0$, 
    or equivalently (by Lemma \ref{lem:avn})  $\partial_{n_0}(A_w)|_{L_0}$. 
    \begin{cor}\label{cor:B}
    Assume that the zeroes of $\rho(x)$ are contained in $\textup{Re}(x)\in (0,1)$.   
    Let $\fR>0$ be such that $\sigma_L$ is contained in the ball $B_{\fR}(0)\subset V$ with radius $\fR$ 
    centered at the origin. The meromorphic integrand (\ref{eq:intL0}) is holomorphic at the points of 
    $d^{-1}(\sigma_L+iV^L)$ for all $\phi,\psi\in \cP^\fR(V_\bbC)$ if for all $w\in W^{F_0}$, for all 
    $a\in\bbC$ such that $\rho(a)=0$, and for all meromorphic functions $\Phi$ on $V_\bbC$ 
    which are holomorphic at $\gl_{a-5}\in L_{\bbC,0}\subset V_\bbC$, 
    \begin{equation}\label{eq:Bw}
        B_w(\gl)=\partial_{n_0}\partial_{\ga_2}\partial_{\ga_3}
        \partial_{\ga_5}\partial_{\ga_6}\partial_{\ga_5+\ga_6}\\
        \nonumber \left.\left(\prod_{\substack{\ga\in \hat{R}(-w)\backslash \hat{R}_+(F_0)\\
        \gb\in \hat{R}(w)}}
        (\frac{{\hat\gb}-1}{{\hat\gb}})
        (\frac{\rho^*({\hat\ga})}{\rho^*({\hat\ga}+1)})\Phi\right)\right\vert_{L_0}
    \end{equation}
    is holomorphic at $\gl_{a-5}\in L_{\bbC,0}$. 
    \end{cor}
    Let us get rid of distractions, and absorb all irrelevant meromorphic factors which are invertibly holomorphic at $\gl$ 
    under the differentiation signs in $B_w$ in the meromorphic function $\Phi$. 
    Set ${\hat{R}}_{x+n}:=\{{\hat\ga}\in {\hat{R}}_+\mid {\hat\ga}(\gl_x)=x+n\}$. 
    Put ${\hat{R}}_{x+n}(-w)={\hat{R}}(-w)\cap {\hat{R}}_{x+n}$. 
    Then we obtain (with $\hat{P}_{L_0}(w):=\hat{P}_{L_0}\cap {\hat{R}}(w)$):
    \begin{cor}\label{cor:C}
    With the same assumptions and notations of Corollary \ref{cor:B}: 
    The meromorphic integrand (\ref{eq:intL0}) is holomorphic at the points of 
    $d^{-1}(\sigma_L+iV^L)$ for all $\phi,\psi\in \cP^\fR(V_\bbC)$ if for all $w\in W^{F_0}$, for all 
    $a$ such that $\rho(a)=0$, and for all meromorphic functions $\Phi$ on 
    $V_\bbC$ which are holomorphic at $\gl_x=\gl_{a-5}$, 
    \begin{equation}\label{eq:Cw}
        C_w(\gl)=\partial_{n_0}\partial_{\ga_2}\partial_{\ga_3}
        \partial_{\ga_5}\partial_{\ga_6}\partial_{\ga_5+\ga_6}\\
        \nonumber \left.\left(
        (\frac{\prod_{{\hat\gb}\in {\hat{R}}_{x+5}(-w)}\rho({\hat\gb})}
        {\prod_{{\hat\ga}\in {\hat{R}}_{x+4}(-w)}\rho({\hat\ga}+1)})
        \prod_{{\hat\gc}\in \hat{P}_{L_0}(w)}({\hat\gc}-1)\Phi\right)\right\vert_{L_0}
    \end{equation}
    is holomorphic at $\gl_{a-5}\in L_{\bbC,0}$. 
    \end{cor}
    We will represent a vector $v$ in $V$ by the column vector in the coordinates of the basis of fundamental weights of $V$, 
    and of $V^*$ in the dual basis $\{\hat\ga_1,\dots,{\hat\ga}_8\}$ of simple coroots (in the usual ``Bourbaki'' order).
    \begin{prop}\label{prop:sets}
    We have that $\hat{P}_{L_0}=\hat{P}_{L_0}(-1)$ is the set of columns of the 
    following $17\times 8$ matrix: 
    \begin{equation}
    P_{L_0}:=
        \begin{pmatrix}
        1&0&0&1&0&0&0&0&0&0&0&0&0&0&0&0&0\\
        0&0&0&0&1&0&0&0&1&1&0&0&0&1&1&0&1\\
        0&0&0&1&0&1&0&0&1&0&1&0&0&1&0&1&1\\
        0&1&0&0&1&1&1&0&1&1&1&1&0&1&1&1&1\\
        0&0&0&0&0&0&1&0&0&1&1&1&1&1&1&1&1\\
        0&0&0&0&0&0&0&1&0&0&0&1&1&0&1&1&1\\
        0&0&1&0&0&0&0&1&0&0&0&0&1&0&0&0&0\\
        0&0&0&0&0&0&0&0&0&0&0&0&0&0&0&0&0
        \end{pmatrix}
    \end{equation}
    and ${\hat{R}}_{x+4}(-1)$ equals the set of columns of the $7\times 8$ matrix:
    \begin{equation}
    {\hat{R}}_{x+4}(-1):=
        \begin{pmatrix}
        1&1&1&1&1&0&1\\
        1&1&1&1&1&1&1\\
        1&2&1&2&1&1&2\\
        2&2&2&2&2&2&2\\
        1&1&2&2&2&2&2\\
        1&1&1&1&2&2&2\\
        1&1&1&1&1&2&1\\
        1&1&1&1&1&1&1
        \end{pmatrix}
    \end{equation}
    and ${\hat{R}}_{x+5}(-1)$ equals the set of columns of the $8\times 8$ matrix:
    \begin{equation}
    {\hat{R}}_{x+5}(-1):=
        \begin{pmatrix}
        1&1&1&1&1&1&1&1\\
        1&1&2&1&1&2&1&2\\
        2&1&2&2&2&2&2&2\\
        3&2&3&3&2&3&3&3\\
        2&2&2&2&2&2&3&3\\
        1&2&1&2&2&2&2&2\\
        1&2&1&1&2&1&1&1\\
        1&1&1&1&1&1&1&1\\
        \end{pmatrix}
    \end{equation}
    \end{prop}
    Consider the orthogonal decomposition $V=V_{L_0}\oplus V^{L_0}$, where $V_{L_0}$ is spanned 
    by the $\textup{E}_7$ roots and $V^{L_0}$ by the highest root $\ga_h$. Accordingly we have the 
    decomposition $V^*=(V^{L_0})^\perp\oplus (V_{L_0})^\perp$ where $(V^{L_0})^\perp$ is the span 
    of the $\textup{E}_7$-coroots, while  $(V_{L_0})^\perp$ is spanned by the highest coroot $\hat\ga_h$.
    The differentiations in (\ref{eq:Cw}) are in directions of $V_{L_0}$ only, hence we can view 
    $\hat\ga_h$ as a constant function. Via the inner product $\bbR\ga_h$ and $\bbR\hat\ga_h$ are naturally 
    identified via $\ga_h\longleftrightarrow \hat{\ga}_h$. In this way we identify the field $K$ of rational functions in 
    $\ga_h$ with the field of rational functions in $\hat{\ga}_h$. 
    
    We consider $c_{L_0}\in V_{L_0}$ as the origin. We have a $K$-linear isomorphism 
    $S_K(V_{L_0})\to P_K(V_{L_0})$\index{$S_K(V_{L_0})$ symmetric algebra of $V_{L_0}$}
    \index{$P_K(V_{L_0})$ algebra of polynomials on  $V_{L_0}$}
    denoted by $s\to s^*$, defined by $v\to (v,\cdot)-(v,c_{L_0})$. We 
    denote its inverse also by $p\to p^*$. 
    The symmetric $K$-bilinear pairing $\langle\cdot,\cdot\rangle_{L_0}: P_K(V_{L_0})\times P_K(V_{L_0})\to K$ 
    given by $\langle p,q\rangle_{L_0}:=\partial(p^*)(q)(c_{L_0})$ is symmetric and non-degenerate. 
    Hence we have $\langle p,qr\rangle_{L_0}=\langle \partial(q^*)p,r\rangle_{L_0}$. 
    Extend this to a $\cM_{L_0}$-bilinear pairing 
    $P_{\cM_{L_0}}(V_{L_0})\times \cM^{L_0}(V_\bbC)\to \cM_{L_0}$ 
    where $\cM_{L_0}$ denotes the field of meromorphic functions on $L_{\bbC,0}$ extended to 
    $V_{\bbC}$ via the projection along $V_{\bbC,L_0}$, and $\cM^{L_0}(V_\bbC)$ 
    the field of meromorphic functions on $V_\bbC$ which do not have $L_{\bbC,0}$ in their singular set. 
    Here we embed $P_K(V_{L_0})$ as a subspace of $\cM^{L_0}(V_\bbC)$ via the decomposition 
    $V=V_{L_0}\oplus V^{L_0}$ (observe that the localisation of $P_K(V_{L_0})$ at $c_{L_0}$, 
    which is isomorphic to the localisation $P^{L_0}(V)$ of $P(V)$ at $L_{\bbC,0}$, is the subalgebra 
    of rational function in $\cM^{L_0}(V_\bbC)$).
    \begin{defn}
    We call $p\in P_K(V_{L_0})$ $F_0$-harmonic if $\langle p,q \rangle_{L_0}=0$ for all $q$ in the ideal 
    $I_{L_0}^+\subset P_K(V_{L_0})$ generated by the space $P_K(V_{L_0})^{W_{F_0}}_+$ of $W_{F_0}$-invariant 
    polynomials which vanish on $L_0$. We denote the vector space of $F_0$-harmonic polynomials by 
    $\cH_{K,F_0}\subset P_K(V_{L_0})$, and we write $\cH_{K,F_0}^+=n_0^*\cH_{K,F_0}\oplus \cH_{K,F_0}\subset P_K(V_{L_0})$.
    \end{defn}
    The condition $p\in\cH_{F_0}$ is equivalent to $\partial(q^*)(p)=0$ for all $q\in  P_K(V_{L_0})^{W_{F_0}}_+$, which 
    implies that $\cH_{K,F_0}\subset  P_K(V_{F_0})$\index{$L_{sp,0}=\textup{E}_7(a4)$!$V_{F_0}$ subspace spanned by $\ga\in F_0$} 
    (viewing this algebra as subalgebra of $P_K(V_{L_0})$ 
    via the orthogonal decomposition $V_{L_0}=V_{F_0}\oplus V^{F_0}_{L_0}$)
    \index{$L_{sp,0}=\textup{E}_7(a4)$!$V^{F_0}$ orthogonal complement of $V_{F_0}$}.
    \begin{defn}\label{defn:quotfld}
    For any $a\in\bbC$, let $\hat{\cO}_a$ denote the ring of formal power series in $x-(a-5)=\frac{1}{2}(\hat\ga_h-1-2a)$
    at $\gl_{a-5}$ on $L_{\bbC,0}$, and let $\hat{K}_a$ be its quotient field.
    \end{defn}
    \begin{prop}
    The subspaces $\cH_{F_0}\subset P_K(V_{L_0})$ and $\cH_{F_0}^+\subset P_K(V_{L_0})$ are modules 
    over $P_K(V_{L_0})$ via the action defined by the operators $p\to\partial(p^*)$. 
    $\cH_{F_0}$ is generated 
    by $\Pi_{F_0}:=\prod_{\ga\in \hat{R}_{F_0,+}}\hat{\ga}$. 
    Using the notations
    of definition \ref{defn:quotfld} 
    the above actions naturally extend to actions 
    on $\cH_{\hat{K}_a,F_0}$ and on $\cH^+_{\hat{K}_a,F_0}$ respectively,  
    by the algebra of the formal power series on $V_\bbC$ at $\gl_{a}\in L_{\bbC,0}$.  
    \end{prop}
    \begin{proof}
    This all straightforward from the definitions. A formal power series on $V_\bbC$ at $\gl_{a-5}$ 
    as a formal power series in $\hat\ga_i-\hat\ga_i(c_{L_0})$ for $i=1,\dots,7$ with coefficients in $\hat{\cO}_a$. 
    The action of the coordinate $\hat\ga_i-\hat\ga_i(c_{L_0})$ on $\cH_{F_0}^+$ is 
    given by $\partial(\ga_i)$, which is obviously nilpotent, hence the result. 
    \end{proof}
    \begin{defn}\label{defn:ops}
    Given $w\in W$ we define the operator 
    \begin{align*}
    \Delta(w):\cH_{\hat{K}_a,F_0}^+&\to\cH_{\hat{K}_a,F_0}^+\\
    h&\mapsto \prod_{\hat\gc\in \hat{P}_{L_0}(w)}\partial(\gc)(h)
    \end{align*}
    Let $S(-w)=\hat{R}_{x+5}(-w)\sqcup\hat{R}_{x+4}(-w)$, and let $G_w\subset W_{F_0}$ denote the 
    stabiliser of $S(-w)$. Let 
    \begin{equation}
    P^{G_w}:\cH_{\hat{K}_a,F_0}^+\to\cH_{\hat{K}_a,F_0}^+
    \end{equation} 
    denote the projector 
    on the $G_w$-invariant subspace. 
    
    Observe that $n_0^*$ is $W_{F_0}$-invarariant, so for 
    $n_0^*h_1+h_2\in\cH_{{\hat{K}_a},F_0}^+$ and $w\in W$ we have   
    $P^{G_w}(n_0^*h^1+h^2)=n_0^*P^{G_w}(h_1)+P^{G_w}(h_2)$ for any $w\in W$. 
    \end{defn}
    \begin{thm}\label{thm:spec}
    Let $x=a\in \bbC$ be a zero of $\rho$ of order $e$, so that $\rho(x)=(x-a)^e\rho^0(x)$ 
    with $e>0$ and $\rho^0(a)\not=0$. 
    Let $\delta:S(-w)\to\{0,1\}$ be the function on $S(-w)$ which is equal to $0$ on 
    $\hat{R}_{x+5}(-w)$ and equal to $1$ on $\hat{R}_{x+4}(-w)$. 
    Let $m:S(-w)\to\{\pm 1\}$ be given by $m(\hat{\ga})=(-1)^{\delta(\hat\ga)}$.
    Denote by $N_w=|\hat{R}_{x+5}(-w)|-|\hat{R}_{x+4}(-w)|$. 
    We use the same assumptions and notations of Corollary \ref{cor:B}. 
    Let $\cI_{w}=P^{G_w}(\Delta(w)(\cH_{F_0}^+))$. For all $\phi,\psi\in \cP^\fR(V_\bbC)$, the meromorphic integrand (\ref{eq:intL0}) 
    is holomorphic at $\gl=\gl_x$ for $x=a-5$ if for all $w\in W^{F_0}$, 
    for all $e>0$, and for all $h\in \cI_{w}^N$ of homogeneous degree $N$ such that 
    $(eN_w-N)<0$:  
    \begin{equation}\label{eq:Dw}
        \left\langle h, \prod_{{\hat\ga}\in S(-w)}({\hat\ga}-a+\delta(\hat\ga))^{em(\hat\ga)}\right\rangle_{L_0}=0
    \end{equation} 
    \end{thm}
    \begin{proof}
    Observe that $\Phi$ runs over $\cM^{\{\gl_{a-5}\}}$ in Corollary \ref{cor:C}. 
    Restricting the pairing $\langle\cdot,\cdot\rangle_{L_0}$ above to $\cM^{\{\gl_{a-5}\}}_{L_0}\subset \cM_{L_0}$ and 
    $\cM^{\{\gl_{a-5}\}}\subset \cM^{L_0}$ (where $\cM^{\{\gl_{a-5}\}}_{L_0}\subset \cM^{\{\gl_{a-5}\}}$ denotes the 
    ring of meromorphic functions on $L_0$ which are regular at $\gl_{a-5}$, and extended to $V_\bbC$ via the decomposition 
    $V_\bbC=V_{\bbC,L_0}\oplus V^{L_0}_\bbC$) yields a $\cM^{\{\gl_{a-5}\}}$-bilinear pairing  
    $\langle\cdot,\cdot\rangle_{L_0}:P_{\cM^{\{\gl_{a-5}\}}_{L_0}}(V_{L_0})\times \cM^{\{\gl_{a-5}\}}(V_\bbC)\to \cM^{\{\gl_{a-5}\}}_{L_0}$. 
    This enables us to rewrite (\ref{eq:Cw}) as follows (using that for $\ga\in \hat{R}_{F_0}$ we have $\gb^*=\hat\gb$): 
    \begin{equation}\label{eq:cwrew}
     C_w(\gl)=\left\langle n_0^*\prod_{\gb\in \hat{R}_{{F_0},+}}\hat\gb,
        \prod_{{\hat\ga}\in S(-w)}({\hat\ga}-a+\delta(\hat\ga))^{em(\hat\ga)}
        \prod_{{\hat\gc}\in \hat{P}_{L_0}(w)}({\hat\gc}-1)\Phi\right\rangle_{L_0}
    \end{equation}
    Write $\Phi=\Phi_N+\Phi^N$ where $\Phi_N\in  P_{\cM^{\{\gl_{a-5}\}}_{L_0}}(V_{L_0})$ and 
    $\Phi^N$ has vanishing order at least $N$ at $L_0$. 
    The summand with $\Phi^N$ in (\ref{eq:cwrew}) obviously vanishes identically. 
    Now $\Phi_N\in P_{\cM^{\{\gl_{a-5}\}}_{L_0}}(V_{L_0})$ is arbitrary, and thus the elements 
    $\partial(\Phi_N^*)(n_0^*\prod_{\gb\in \hat{R}_{{F_0},+}}\hat\gb)$ span $\cH_{\cM^{\{\gl_{a-5}\}}_{L_0},F_0}^+$ over 
    $\cM^{\{\gl_{a-5}\}}_{L_0}$ 
    when we vary $\Phi$ (first note that $P(V_{F_0})$ generates $n_0^*\cH_{F_0}$, then it is easy to see 
    that $P(V_{L_0})$ generates all of $\cH_{F_0}^+$).
    Hence the pairing $C_w(\gl)$ takes the general form 
    \begin{equation}
        \left\langle h, \prod_{{\hat\ga}\in S(-w)}({\hat\ga}-a+\delta(\hat\ga))^{em(\hat\ga)}\right\rangle_{L_0}
    \end{equation} 
    where $h$ varies over the $\cM^{\{\gl_{a-5}\}}_{L_0}$-span of $\textup{Im}(\Delta(w)):=\Delta(w)(\cH_{F_0}^+)$. 
    The pairing is clearly invariant with respect to $W_{F_0}$, 
    hence when $\sigma\in G_w$ we have 
    \begin{equation}
        \left\langle h, \prod_{{\hat\ga}\in S(-w)}({\hat\ga}-a+\delta(\hat\ga))^{em(\hat\ga)}\right\rangle_{L_0}=
         \left\langle \sigma(h), \prod_{{\hat\ga}\in S(-w)}({\hat\ga}-a+\delta(\hat\ga))^{em(\hat\ga)}\right\rangle_{L_0}
    \end{equation} 
    This shows that for all $h\in \textup{Im}(\Delta(w))$ we have: 
    \begin{equation}\label{eq:Cfin}
        \left\langle h, \prod_{{\hat\ga}\in S(-w)}({\hat\ga}-a+\delta(\hat\ga))^{em(\hat\ga)}\right\rangle_{L_0}=
         \left\langle P^{G_w}(h), \prod_{{\hat\ga}\in S(-w)}({\hat\ga}-a+\delta(\hat\ga))^{em(\hat\ga)}\right\rangle_{L_0}
    \end{equation} 
    The operator $P^{G_w}\circ\Delta(w)$ is homogeneous of degree $-|\hat{P}_{L_0}(w)|$, hence $\cI_w$
    is the direct sum of its homogeneous components. Hence Corollary \ref{cor:C}
    shows that we need to show that for all $h\in\cI_w^N$ homogeneous of degree $N$, (\ref{eq:Cfin}) is 
    holomorphic at $\gl_x=\gl_{a-5}$. On the other hand, obviously 
    \begin{equation}\label{eq:CfinN}
        \left\langle h, \prod_{{\hat\ga}\in S(-w)}({\hat\ga}-a+\delta(\hat\ga))^{em(\hat\ga)}\right\rangle_{L_0}=
        E(h,S(-w),e)(x-a+5)^{eN_w-N}
    \end{equation} 
    for some constant $E(h,S(-w),e)\in\bbC$ depending on $w, e$ and $h\in\cI_w^N$. If $(eN_w-N)<0$ then 
    this is holomorphic at $\gl_x=\gl_{a-5}$ iff $E(h,S(-w),e)=0$, as was to be proven.
    \end{proof}
    The constants $E(h,S(-w),e)$\index{$E(h,S(-w),e)$} introduced in the proof of Theorem \ref{thm:spec} play an important role in our analysis. The next Lemma determines these explicitly:
    \begin{lem}\label{lem:cost}
    Fix $w\in W$, $S=S(-w)$ and $h\in\cI_w^N$ as in Theorem \ref{thm:spec}. 
    The constant $E(h,S,e)\in\bbC$ of (\ref{eq:CfinN}) in the proof of Theorem \ref{thm:spec} is  
    linear in $h\in\cI_w^N$ and polynomial in $e$ of degree at most $N\leq 6-|\hat{P}_{L_0}(w)|$ 
    (since  $\cI_w^N=0$ for $N>6-|\hat{P}_{L_0}(w)|$). Write $h=\sum_{Q}c_Q Q$ where 
    $Q$ varies over monomials in $P(V_{F_0})$. 
    Then $E(h,S,e)=\sum_Qc_Q E(Q,S,e)$ where, for $Q=\prod_{i=1}^Nv_i^*$ with $v_i\in V_{F_0}$, we have: 
    \begin{align}\label{eq:EQ}
    E(Q,S,e)=\sum_{\substack{(\pi^+,\pi^-)\\\mathrm{\ bipartition\, of\, }N\\\forall i: \pi^+_i\leq e}}
    (-1)^{|\pi^-|}\prod_{i,j}\binom{e}{\pi^+_i}&\binom{e+\pi^-_j-1}{\pi^-_j}\\
    \nonumber&\sum_{\substack{d:S\to\bbZ_{\geq0}:\\(\pi^+(d),\pi^-(d))=(\pi^+,\pi^-)}}
    \textup{Perm}\left((v_i,s^{d,0}_j)\right)
    \end{align}
    Here the product $\prod_{i,j}$ runs over all parts of $\pi^-$ and of $\pi^+$, and for a map $d:S\to\bbZ_{\geq0}$
    we define the partition $\pi^+(d)$ by ordering the images of $d|_{S_+}$ in descending order, and likewise 
    $\pi^-(d)$ by ordering the images of $d|_{S_-}$ in descending order. Finally for each $d:S\to\bbZ_{\geq0}$
    such that $(\pi^+(d),\pi^-(d))=(\pi^+,\pi^-)$ we fix a map $q^{d,0}:\{1,\dots,N\}\to S$, $j\to q^{d,0}_j\in S$ 
    such that for all $s\in S$, $d_s=|(q^{d,0})^{-1}(s)|$ (any such map will do). 
    \end{lem}
    \begin{proof}
    Given a tuple $(\xi_1,\dots,\xi_N)\in V^N$ we define $E(\xi_1,\dots,\xi_N,S,e)\in \bbC$ by 
    \begin{equation}
    \prod_{i=1}^N\partial_{v_i}\left.\left(\prod_{{\hat\ga}\in S(-w)}({\hat\ga}-a+\delta(\hat\ga))^{em(\hat\ga)}\right)\right|_{\gl=\gl_x}=
    E(v_1,\dots,v_N;S,e)(x-a+5)^{eN_w-N}
    \end{equation}
    (clearly the result of this differentiation and restriction will be of this form). 
    Observe that 
     $E(v_1,\dots,v_N;S,e)$ is symmetric and multilinear in the $v_i$. Hence the formula factors 
     through the symmetric algebra $S^N(V)$, and $E(v_1,\dots,v_N;S,e)=E(Q^*;S,e)$. 
     Now $S^N(V)$ is spanned by the powers $v^N$ 
     where $v$ varies over $V$, hence it suffices to verify the formula for $v_1=\dots=v_N=v$. 
     In this case, we have the multinomial formula for the power of the derivation $\partial_v^N$ applied 
     to $\prod_{k=1}^Kf_k$:
     \begin{equation}
     \partial_v^N(\prod_{k=1}^K f_k)=\sum_{d=(d_1,d_2,\dots,d_K)}\binom{N}{d_1,\dots,d_K}
     \prod_{k=1}^K \partial^{d_k}_v(f_k)
     \end{equation}
     Now take $f_k$ of the form: $f_k=(\hat\ga_k+c_k)^{em_k}$ with $m_k\in\{\pm 1\}$ and $e>0$, 
     where $(\hat\ga_k+c_k)(\gl_x)=x-a+5$ for all $k$. Then (with $[K]_\pm=m^{-1}(\pm 1)\subset [K]:=\{1,\dots,K\}$):
     \begin{align}
    \prod_{k=1}^K \partial^{d_k}_v(f_k)(\gl_x)=
    \prod_{k_+\in [K]_+}\binom{e}{d_{k_+}}d_{k_+}!\prod_{k_-\in [K]_-}&\binom{e+d_{k_-}-1}{e}(-1)^{d_{k_-}}d_{k_-}!\\ 
    \nonumber&\prod_{k=1}^K\hat\ga_k(v)^{d_k}(x-a+5)^{e(K_+-K_-)-N}
    \end{align}
    Finally observe that 
    \begin{equation}
    \textup{Perm}\left.\left((v_i,s^{d,0}_j)\right)\right|_{(v_1,\dots,v_N)=(v,\dots,v)}=N!\prod_{k=1}^K\hat\ga_k(v)^{d_k}
    \end{equation}
    We conclude that the right hand side of (\ref{eq:EQ}) equals $E(v,v,\dots,v;S,e)$ indeed, 
    finishing the proof.
    \end{proof}
    \begin{cor}\label{cor:finalspec}
    With the notations and setup as in Theorem \ref{thm:spec}. 
    The meromorphic integrand (\ref{eq:intL0}) is holomorphic at the points of 
    $d^{-1}(\sigma_L+iV^L)$ for all $\phi,\psi\in \cP^\fR(V_\bbC)$ if for all $w\in W^{F_0}$, for all 
    $N\leq 6-|\hat{P}_{L_0}(w)|$ and all $h\in \cI_w^N$ such that $(eN_w-N)<0$ we have
    $E(h,S(-w),e)=0$.
    \end{cor}
    \begin{proof}
    Lemma \ref{lem:cost} and the proof of Theorem \ref{thm:spec} show that the vanishing condition in 
    Theorem \ref{thm:spec} is equivalent to the vanishing $E(h,S(-w),e)=0$ whenever $(eN_w-N)<0$. 
    This condition is independent of $a$, and the result follows.
    \end{proof}
    \begin{thm}\label{thm:mainspec}
    With the notations and setup as in Theorem \ref{thm:spec}. 
    The meromorphic integrand (\ref{eq:intL0}) is holomorphic at the points of $d^{-1}(\sigma_L+iV^L)$, 
    for all $\phi,\psi\in \cP^\fR(V_\bbC)$. 
    \end{thm}
    \begin{proof}
    The vanishing condition of Corollary \ref{cor:finalspec} has been verified by computer algebra, using the explicit expression for 
    $E(h,S(-w),e)$ of Lemma 
    \ref{lem:cost}.
    \end{proof}

\section{Appendix E: The pseudocode for detecting poles in the special case}\label{a:SpecialCode}
The presence of an $\Omega$ pole space in $C$ of higher order which is intersected by a segment of one of its parent pole spaces in $C$ at a point different from its center, is a serious challenge for our method, since this forces us to consider the singularities of individual summands, 
depending on $w\in W$, of 
certain derivatives of 
$\Sigma_W(\phi)\Sigma_{W'}(\psi)$  restricted to $L$. Fortunately, there is  only one such case, $L_{sp}$. 
This situation occurs in 
the cascade for $\textup{E}_7\subset \textup{E}_8$. 

In order to verify Corollary \ref{cor:finalspec}, the following steps need to be made:
\begin{itemize}
\item[(a)] Classify all $w\in W=W(\textup{E}_8)$ for which the set $\hat{P}_{L_0}(w) = \hat{P}_{L_0} \cap \hat{R}(w)$ has cardinality $|\hat{P}_{L_0}(w)|\leq 6$.
\item[(b)] Compute $E(h,S(-w),e)$ for all $w$ in (a), for all $h\in \cI_w^N$ and for all $e$ satisfying $(eN_w-N)<0$.
\end{itemize}
We note that from (\ref{eq:EQ}) the computation of each $E(h,S(-w),e)$ in (b) is explicit and straightforward to be programmed in a computer, so the important step is to minimize as much as possible the cases in which such computations are needed.
\subsection{Classification of $w\in W$ as in (a)}
A straightforward approach to swipe the whole Weyl group $W = W(\textup{E}_8)$, compute  $\hat{P}_{L_0}(w)$ and select the elements for which
$|\hat{P}_{L_0}(w)|\leq 6$ is not practical, due to the large cardinality of $W(\textup{E}_8)$. The strategy we use is the following. Firstly, we decompose 
$W = W_{8/7} \times W_{7/6} \times W_{6}$, with $W_6 = W(\textup{E}_6)$, $W_{7/6} $ and $W_{8/7}$ complete sets of minimal coset representatives for, respectively, the coset space $W(\textup{E}_7)/ W(\textup{E}_6)$ and $W(\textup{E}_8)/ W(\textup{E}_7)$. Any $w\in W$ is thus decomposed uniquely as
\begin{equation}\label{eq:WE8dec}
w = \eta\tau u,
\end{equation}
with $u\in W_6,\tau\in W_{7/6} $ and $\eta\in W_{8/7}$.

Secondly, using the inclusion $L\subset N$ ($N$ as in Lemma \ref{lem:specialholo}), the fact that $N$ has $\Omega$-order $0$ and is a residual space of type $\textup{E}_6(a3)$ (so that  $|\hat{Z}_N \cap \hat{R}_+|=3$), and $(\sigma_{sp},L_{sp})$ is obtained by intersecting $L\subset N$ with a segment in $N$, we have the following.
\begin{prop}
For (a), it suffices to classify the elements $w\in W(\textup{E}_8)$ with $|\hat{P}_{L_0(w)}|\leq 6$ for which $w=\eta\tau u$ as in (\ref{eq:WE8dec}) and $|\hat{P}_N(u)|\leq 3$.
\end{prop}
\begin{proof}
In order for $w\in W$ to be eligible to be checked for the vanishing in (b), the summands of $\Sigma_W(\phi)$ related to $w$ must not vanish on $N$ first, which translates in the condition $|\hat{P}_N(w)|\leq 3$. Further, the coroot system of constant roots $\hat{R}_N$ is of type $\textup{E}_6$. Using that $\hat{R}(\tau u) = \hat{R}(\tau)\cup\tau^{-1}\hat{R}(u)$ (see (\ref{eq:InversionMinCoset})) it is straightforward to compute $|P_N(\tau u)| = |P_N(u)| = 3$, and similarly,  $|P_N(\eta\tau u)| = |P_N(u)| = 3$, for all $\tau\in W_{7/6} $ and $\eta\in W_{8/7}$. Finally, as $|\hat{P}_{L}(w)|  = |\hat{P}_{L_0}(w)|$ for all $w$ (since $L = d(L_0)$ with $d=s_8s_7$), the claims follows.
\end{proof}
The classification of $u\in W_{6}$ such that $|\hat{P}_N(u)|\leq 3$ is doable by swiping through all elements of $W(\textup{E}_6)$. Furthermore, it is also straightforward to compute the minimal coset representatives in 
$W_{7/6}$ and $W_{8/7}$ by identifying them with suitable subsets of the root system $\hat{R}(\textup{E}_8)$ of type $\textup{E}_8$. In the next proposition, we embed $W(\textup{E}_6)\subset W(\textup{E}_7) \subset W(\textup{E}_8)$ naturally as standard parabolic subgroups.
\begin{prop}\label{p:MinCosetE}
Let $\hat{R}(\textup{E}_8)_{+}^{(1)} = \{\hat{\alpha} \in \hat{R}(\textup{E}_8)_{+}\mid \langle \hat{\alpha},\fw_8\rangle=1\}$, where $\fw_8$ is the fundamental weight dual to the simple coroot $\hat{\alpha}_8$. Then,
the following holds true:
\begin{itemize}
\item[(i)] The set of minimal coset representatives $W_{7/6}$ is in bijection with $\hat{R}(\textup{E}_8)_{+}^{(1)}$.
\item[(ii)] The set of minimal coset representatives $W_{8/7}$ is in bijection with $\hat{R}(\textup{E}_8)$.
\end{itemize}
\end{prop} 
\begin{proof}
For (i), one checks that $W(\textup{E}_7)$ acts transitively on $\hat{R}(\textup{E}_8)_{+}^{(1)}$ and the stabilizer of $\hat\alpha_8\in\hat{R}(\textup{E}_8)_{+}^{(1)}$ is $W(\textup{E}_6)$. For (ii), we use the fact that $W(\textup{E}_8)$ acts transitively on $\hat{R}(\textup{E}_8)$ and the isotropy group of the highest coroot $\hat{\alpha}_h$ is $W(\textup{E}_7)$.
\end{proof}
Using Proposition \ref{p:MinCosetE}, we determine $W_{7/6}$ and $W_{8/7}$ by finding, for each coroot $\hat{\alpha}$ in the respective set, the minimal length element $w_{\hat\alpha}$ that such that $w_{\hat\alpha}(\hat\alpha)$ is in the closure of the dominant chamber. Such algorithm is easy to implement. Furthermore, we note that the elements of Proposition \ref{p:MinCosetE} can be taken modulo $W_{F_{0}}$, the standard parabolic subgroup of $W$ that fixes $L_0$ pointwise. We therefore determine, by swiping over all elements of $W(\textup{E}_6)$, the sets 
\[
U_m^{F_0} := \{u \in W(\textup{E}_6) \mid u \textup{ of minimal length in } W(\textup{E}_6)/W_{F_0} \textup{ and }|\hat{P}_N(u)| = m\},
\]
for $m\in\{0,1,2,3\}$. The cardinality $|U_m^{F_0}|$ equals, respectively, $1,3,60$ and $150$ for $m=0,1,2,3$. We let $U^{F_0} = \cup_{m=0}^3 U_m^{F_0}$.

\subsection{Computation of $E(h,S(-w),e)$}

Formula (\ref{eq:EQ}) for $E(h,S(-w),e)$ is explicit and straightforward to be programmed. In view of (\ref{eq:WE8dec}), each triple $(u,\tau,\eta)\in U^{F_0}\times W_{7/6}\times W_{8/7}$ determines $w=u\tau\eta$ from which we determine the set $S(-w)$ needed to compute $E(h,S(-w),e)$. We note that the input $h = \sum_Qc_QQ$ is such that the coefficient $c_Q$ is given by the permanent of matrices with entries given by pairing $\hat\alpha \in \hat{P}_{L_0}$ and $\beta\in \hat{R}_{L_0,+}\cup \{n_0\}$. In practice, many of these coefficients $c_Q$ are zero and the cases left to be computed by formula (\ref{eq:EQ}) are relatively small compared with the cardinalities $|U^{F_0}\times W_{7/6}\times W_{8/7}|$ and $|W(\textup{E}_8)|$. 
    
\section*{Acknowledgements}
During the period over which this project has taken place, we thank the financial supports given by the ERC-advanced
grant 268105; the grants of Agence Nationale de la Recherche with references ANR-08-BLAN-0259-02 and ANR-13-BS01-0012 FERPLAY; the Archimedes LabEx ANR-11-LABX-0033, the foundation A*MIDEX ANR-11-IDEX-0001-02, funded by the program ``Investissements d’Avenir” led by the ANR; the Engineering and Physical Sciences Research Council grant EP/N033922/1 (2016) and the special research fund (BOF) from Ghent University BOF20/PDO/058. The computations in this paper were performed by using Maple™, which is a trademark of Waterloo Maple Inc.

    \printindex
    \end{document}